\documentclass[9pt,leqno,a4paper]{article}
\usepackage{xcolor}
\usepackage{amsfonts}
\usepackage{amsmath, amsfonts, amsthm, amssymb, amscd}
\usepackage[mathcal]{euscript}
\usepackage{mathrsfs}
\usepackage{dsfont}
\usepackage{ulem}

\newtheorem{theorem}{Theorem}[section]
\newtheorem{proposition}[theorem]{Proposition}
\newtheorem{lemma}[theorem]{Lemma}
\newtheorem{corollary}[theorem]{Corollary}

\newtheorem{remark}[theorem]{Remark}

\numberwithin{equation}{section}

\usepackage{a4wide}
\usepackage{multirow}
\usepackage{tabularx}
\usepackage{lscape}

\newcommand \Kcal {\mathcal K}
\newcommand \Hcal {\mathcal H}

\newcommand \xb {\bar {x}}
\newcommand \delb {\bar {\del}}
\newcommand \gb {\bar g}
\newcommand \hb{\bar h}

\newcommand \Tb {\overline {T}}
\newcommand \Phib{\overline{\Phi}}
\newcommand \Psib{\overline{\Psi}}

\newcommand \mb{\overline m}
\newcommand \Qb{\overline Q}
\newcommand \omegab{\bar{\omega}}
\newcommand \rhob{\bar{\rho}}

\newcommand \del \partial
\newcommand \delu {\underline{\del}}

\newcommand \Ec{E_{\text{con}}}

\newcommand \RR{\mathbb{R}}

\newcommand {\vep}{\varepsilon}

\makeatletter
\def\hlinew#1{%
  \noalign{\ifnum0=`}\fi\hrule \@height #1 \futurelet
   \reserved@a\@xhline}
\makeatother

\let\oldmarginpar\marginpar
\renewcommand\marginpar[1]{\-\oldmarginpar[\raggedleft\footnotesize #1]%
{\raggedright\footnotesize #1}}

\begin{document}
\title{A conformal-type energy inequality on hyperboloids\\and its application to quasi-linear wave equation in $\RR^{3+1}$
\footnote{The present work is supported by the China Postdoctoral Science Foundation, via the Special Financial Grant (2017T100732). }}

\author{Yue MA\footnote{School of Mathematics and Statistics, Xi'an Jiaotong University, yuemath@xjtu.edu.cn.}\ \&\ Hongjing HUANG\footnote{The First Academy, 2667 Bruton Boulevard Orlando, Florida, United States. Email: hongjinghuang@thefirstacademy.org}}
\maketitle

\begin{abstract}
 In the present work, we will develop a conformal inequality in the hyperbolic foliation context which is analogous to the conformal inequality in the classical time-constant foliation context. Then as an application, we will apply this a priori estimate to the problem of global existence of quasi-linear wave equations in three spatial dimensions under null condition. With the aid of this inequality, we can establish more precise decay estimates on the global solution.
\end{abstract}

\section{Introduction}
In this article we will develop a conformal inequality in the hyperbolic foliation context analogue to the classical time-constant foliation context. Then we will apply this estimate to the problem of global existence of regular solution associated to small regular initial data (also called the {\sl global stability } for short in the following text) for quasi-linear wave equation in three spatial dimensions. More precisely, we are going to regard the following quasi-linear wave equation:
\begin{equation}\label{eq main}
\Box u  + Q^{\alpha\beta\gamma}\del_{\gamma}u\del_{\alpha}\del_{\beta}u = 0
\end{equation}
where $Q$ is supposed to be a {\sl null} cubic form.

One can also consider a system with semi-linear terms such as:
\begin{equation}\label{eq main'}
\Box u + Q^{\alpha\beta\gamma}\del_{\gamma}u\del_{\alpha}\del_{\beta}u = N^{\alpha\beta}\del_{\alpha}u\del_{\beta}u
\end{equation}
with $N$ a {\sl null} quadratic form. But through an algebraic observation one can show that by a change of known
$$
v = u + \frac{\sigma}{2}u^2, \quad \sigma = N^{00}
$$
the semi-linear term can be eliminated with some high-order correction terms, which are negligible in dimension three.

\subsection{The conformal inequality on hyperboloids}
In the classical flat foliation context, the conformal energy inequality is a well-known $L^2$ type estimate which controls more quantities than the classical energy inequality. We recall that for
$$
\Box u = f,
$$
the conformal energy $\Ec(t,u)$ is defined as
$$
\Ec(t,u) : =\int_{\RR^3}\big(|Su|^2 + \sum_a|L_au|^2 + \sum_{a\neq b}|\Omega_{ab}u|^2 + u^2\big)(x,t)\ dx
$$
where
$$
Su: = t\del_tu + x^a\del_au, \quad L_a u: = x^a\del_tu + t\del_au,\quad \Omega_{ab}u := x^a\del_bu - x^b\del_au.
$$
And by using the multiplier $\left((r^2+t^2)\del_t + 2x^at\del_a + 2t \right)u$, we have the following estimate (see for example \cite{A1} section 6.7):
There exists a positive constant $C$ such that, for all $u\in C^2\left(\RR_x^3\times[0,T)\right)$, sufficiently decaying when $|x|\rightarrow +\infty$, and all $t<T$,
\begin{equation}\label{eq 1 30-05-2017}
\Ec(t,u)^{1/2}\leq C\Ec(0,u)^{1/2} + C\int_0^t\big\|\sqrt{r^2+t^2}f(x,t)\big\|_{L^2_x(\RR^n)}\ dx.
\end{equation}

From the above estimates we can see the following facts. For the homogeneous wave equation $\Box u = 0$, if we differentiate the equation with respect to $\del^IL^J$,
$$
\Box \del^IL^Ju = 0
$$
and by \eqref{eq 1 30-05-2017}, we see that $\Ec(t,\del^IL^Ju)^{1/2}$ is bounded by the initial conformal energy. This leads to the fact that
\begin{equation}\label{eq 1 01-11-2017}
\|S\del^IL^J u(t,\cdot )\|_{L^2(\RR^3)},\quad \|L_a\del^IL^J u(t,\cdot )\|_{L^2(\RR^3)},\quad \|\Omega_{ab}\del^IL^J u(t,\cdot )\|_{L^2(\RR^3)},\quad
\|\del^IL^J u(t,\cdot)\|_{L^2(\RR^3)}
\end{equation}
are bounded. Then by some estimates on commutators and the Klainerman-Sobolev type inequality, we will see that
$$
Su,\quad L_au,\quad \Omega_{ab}u
$$
are bounded by $(|t-r|+1)^{-1/2}(t+1)^{-1}$, and especially:
\begin{equation}\label{eq 1 29-08-2017}
u\sim (t+1)^{-1}(|t-r|+1)^{-1/2}, \quad \del_{\alpha} u\sim (t+1)^{-1}(|t-r|+1)^{-3/2}.
\end{equation}
Compared with the classical energy-vector field argument, this conformal energy-vector field method supplies better decay bounds on the solution $u$: the classical energy only gives the following decay bounds:
\begin{equation}\label{eq 2 01-11-2017}
\del_{\alpha}u\sim (t+1)^{-1}(|t-r|+1)^{-1/2}.
\end{equation}

In the present article, we will establish a parallel estimate in the hyperbolic foliation context. More precisely, we will use the following multiplier:
$$
(s/t)\left((t^2+r^2)\del_t + 2x^at\del_a + 2t\right)u
$$
where $s = \sqrt{t^2-r^2}$ is the hyperbolic time. In the flat (Minkowski metric) case, this energy is written as
$$
\Ec(s,u) = :\int_{\Hcal_s}\big|s(s/t)^2\del_tu\big|^2 + (s/t)^2\sum_a\big|L_au\big|^2 + (s/t)u^2 dx.
$$
Here $\Hcal_s = \{x\in\RR^{1+3}|t = \sqrt{s^2+r^2}\}$, and for a $f$ defined in $\RR^{1+3}$, we define
$$
\int_{\Hcal_s}fdx := \int_{\RR^3}f(\sqrt{s^2+|x|^2},x)dx.
$$
For the linear equation
$$
\Box u = f
$$
the following estimate holds:
\begin{equation}
\Ec(s,u)^{1/2}\leq \Ec(s_0,u) + C\int_{s_0}^s \tau \|f\|_{L^2(\Hcal_{\tau})}d\tau
\end{equation}
For the homogeneous linear wave equation $\Box u = 0$, the above quantity is conserved. For the quasi-linear wave equation \eqref{eq main}, we will prove that this energy is also bounded up to the highest order.

Combined with the global Sobolev's inequality, the above energy bounds leads to the following terms bounded by $C(t+1)^{-3/2}$:
$$
(s/t)L_a u ,\quad s(s/t)^2\del_t \quad (s/t)u,
$$
and these bounds gives the following decay rate:
$$
u\sim (t+1)^{-1}(|t-r|+1)^{-1/2}, \quad \del_{\alpha} u\sim (t+1)^{-1}(|t-r|+1)^{-3/2}
$$
which coincides with the classical conformal energy bounds.  We will prove that in the quasi-linear case, these decay rates still hold.

\subsection{The global existence result for quasi-linear wave equation}
The problem on the stability of quasi-linear wave equations or systems has attracted lots of attentions of the mathematical community. Our method to be presented belongs to the ``vector-field method'' which was introduced by S. Klainerman for wave equation (\cite{Kl1}) and for Klein-Gordon equation (\cite{Kl2}). This method is then extended to many different cases.

The global existence of \eqref{eq main} has been established by S. Klainerman in \cite{Kl1}.
His method is based on the time-constant foliation and standard energy inequalities.

The hyperbolic foliation and hyperbolic variables are introduced by S. Klainerman firstly for the analysis on Klein-Gordon equation in \cite{Kl2}, see also \cite{H1} for the ``alternative energy method''. This method is then revisited and applied by P. LeFloch et al. in \cite{PLF-MY-book} on system composed by wave equations and Klein-Gordon equations. The application of the hyperbolic foliation has the advantage that it does not require the scaling invariance of the system (for example, the wave-Klein-Gordon system).

However, the scaling invariance of the wave equation does supply more conserved quantities. The conformal inequality \eqref{eq 1 30-05-2017} is essentially due to the scaling invariance of the wave equation (i.e. it does not hold for Klein-Gordon equation, which does not enjoy the scaling invariance), and it leads to better decay estimates (\eqref{eq 1 29-08-2017} compared with \eqref{eq 2 01-11-2017} supplied by the classical energy). In this article, we show that this conformal energy inequality on hyperboloids also leads to the global stability of \eqref{eq main} and it gives more precise decay rate for the global solution.
%
%
%

\section{The conformal inequality on hyperboloids: flat case}
\subsection{The hyperbolic variables and hyperbolic frame}
In this subsection we briefly recall the notion of the hyperbolic frame.

We denote by $\Kcal = \{(t,x)\in\RR^{1+3}|t>|x|+1\}$ the interior of the light cone. In this region we introduce the following parametrization by the hyperbolic variables:
$$
\xb^0 := s := \sqrt{t^2-|x|^2},\quad \xb^a: = x^a.
$$
The canonical frame associate to this parametrization is called the hyperbolic frame, denoted by:
$$
\delb_0: = \del_s = \frac{s}{t}\del_t,\quad \delb_a : = \frac{x^a}{t}\del_t + \del_a.
$$
The transition matrix between the hyperbolic frame and the natural frame is
$$
\delb_\alpha = \Phib_{\alpha}^{\beta}\del_{\beta}
$$
with
$$
\left(\Phib_{\alpha}^{\beta}\right)_{\alpha\beta} :=
\left(
\begin{array}{cccc}
s/t &0 &0 &0
\\
x^1/t &1 &0 &0
\\
x^2/t &0 &1 &0
\\
x^3/t &0 &0 &1
\end{array}
\right)
$$
and its inverse is
$$
\left(\Psib_{\alpha}^{\beta}\right)_{\alpha\beta} =
\left(
\begin{array}{cccc}
t/s &0 &0 &0
\\
-x^1/s &1 &0 &0
\\
-x^2/s &0 &1 &0
\\
-x^3/s &0 &0 &1
\end{array}
\right)
$$
thus
$$
\del_{\alpha} = \Psib_{\alpha}^{\beta}\delb_{\beta}.
$$

Now we introduce the following notation. Let $T$ be a two-tensor, we recall that it can be written in the natural frame $\{\del_{\alpha}\}$ or in the hyperbolic frame $\{\delb_{\alpha}\}$. That is
$$
T = T^{\alpha\beta}\del_{\alpha}\otimes\del_{\beta} = \Tb^{\alpha\beta}\delb_{\alpha}\otimes\delb_{\beta}
$$
with
$$
\Tb^{\alpha\beta} = T^{\alpha'\beta'}\Psib_{\alpha'}^{\alpha}\Psib_{\beta'}^{\beta}.
$$
For a cubic form $Q$, we also have similar notation:
$$
Q = Q^{\alpha\beta\gamma}\del_{\alpha}\otimes\del_{\beta}\otimes\del_{\gamma} = \Qb^{\alpha\beta\gamma}\delb_{\alpha}\otimes\delb_{\beta}\otimes\delb_{\gamma},
$$
and
$$
\Qb^{\alpha\beta\gamma} = Q^{\alpha'\beta'\gamma'}\Psib_{\alpha'}^{\alpha}\Psib_{\beta'}^{\beta}\Psib_{\gamma'}^{\gamma}.
$$

We also recall the dural co-frame to the hyperbolic frame:
$$
\omegab^0 = (t/s)dt - \sum_a(x^a/s)dx^a,\quad \omegab^a = dx^a.
$$

We recall the Minkowski metric $m^{\alpha\beta}$ written in the hyperbolic frame:
$$
\mb^{\alpha\beta} =
\left(
\begin{array}{cccc}
1 &x^1/s &x^2/s &x^3/s
\\
x^1/s &-1 &0 &0
\\
x^2/s &0 &-1 &0
\\
x^3/s &0 &0 &-1
\end{array}
\right).
$$
And we remark that for a second order differential operator $g^{\alpha\beta}\del_{\alpha}\del_{\beta}$, we see that
$$
g^{\alpha\beta}\del_{\alpha}\del_{\beta}
= g^{\alpha\beta}\Psib_{\alpha}^{\alpha'}\Psib_{\beta}^{\beta'}\delb_{\alpha'}\delb_{\beta'}
+ g^{\alpha\beta}\del_{\alpha}\left(\Psib_{\beta}^{\beta'}\right)\delb_{\beta'} = \gb^{\alpha\beta}\delb_{\alpha}\delb_{\beta}
+ g^{\alpha\beta}\del_{\alpha}\left(\Psib_{\beta}^{\beta'}\right)\delb_{\beta'}.
$$
Here we list out $\del_{\alpha}\left(\Psib_{\beta}^{\beta'}\right)$:
\begin{subequations}
\begin{equation}\label{eq 2 14-08-2017}
\del_t\left(\Psib_0^0\right) = \frac{-r^2}{s^3} = -s^{-1}\Psib_0^0\Psib_0^0 + s^{-1},\quad \del_t\left(\Psib_a^0\right) = \frac{x^a t}{s^3} = -s^{-1}\Psib_0^0\Psib_a^0
\end{equation}
and
\begin{equation}\label{eq 3 14-08-2017}
\del_a\left(\Psib_0^0\right) = \frac{tx^a}{s^3} = -s^{-1}\Psib_0^0\Psib_a^0,\quad \del_a\left(\Psib_b^0\right) = -\frac{\delta_{ab}}{s} - s^{-1}\Psib_a^0\Psib_b^0.
\end{equation}
\end{subequations}
The rest components of $\del_{\alpha}\left(\Psib_{\beta}^{\beta'}\right)$ are zero.

Now we recall that
$$
\Box = m^{\alpha\beta}\del_{\alpha}\del_{\beta}
$$
thus in hyperbolic frame:
\begin{equation}\label{eq 1 12-05-2017}
\aligned
\Box =& \mb^{\alpha\beta}\delb_{\alpha}\delb_{\beta} + m^{\alpha\beta}\del_{\alpha}\left(\Psib_{\beta}^{\beta'}\right)\delb_{\beta'}
=\delb_s\delb_s + 2(x^a/s)\delb_s\delb_a- \sum_a\delb_a\delb_a + \frac{3}{s}\delb_s.
\endaligned
\end{equation}

\subsection{Hyperbolic decomposition of $\Box$}
We write \eqref{eq 1 12-05-2017} into the following form:
\begin{equation}\label{eq1-11-02-2017}
\aligned
\Box u =& \delb_s\left(\delb_s + (2x^a/s)\delb_a\right)u - \sum_a\delb_a\delb_au + \frac{2x^a}{s^2}\delb_au + \frac{3}{s}\delb_su
\\
=&\delb_s\left(\delb_s + (2x^a/s)\delb_a\right)u - \sum_a\delb_a\delb_au + s^{-1}\left(\delb_s + (2x^a/s)\delb_a\right)u + \frac{2}{s}\delb_su
\\
=&s^{-1}\del_s\left(s\delb_s u+ 2x^a\delb_au\right) - \sum_a\delb_a\delb_au + \frac{2}{s}\delb_su
\endaligned
\end{equation}
We denote by
$$
Ku := \left(s\delb_s + 2x^a\delb_a\right)u = \frac{t^2+r^2}{t}\del_tu + 2r\del_ru
$$

\subsection{The energy identity}
We make the following calculation:
\begin{equation}\label{eq1-10-05-2017}
\aligned
sKu \cdot \Box u =& \left(s\left(\delb_s + (2x^b/s)\delb_b\right)u\right)\delb_s\left(s\left(\delb_s + (2x^a/s)\delb_a\right)u\right)
\\
&-\sum_a\delb_a\left(s^2\left(\delb_su + (2x^b/s)\delb_bu\right)\cdot\delb_au\right) + s^2\sum_a\delb_a\left(\left(\delb_su + (2x^b/s)\delb_bu\right)\right)\delb_au
\\
&+2s\delb_su\cdot \left(\delb_su + (2x^b/s)\delb_bu\right)
\\
=& \frac{1}{2}\delb_s\left(|Ku|^2\right)
 +\sum_a\left(s^2\delb_a\delb_su\cdot \delb_au + 2s|\delb_au|^2 + s^2(2x^b/s)\delb_a\delb_bu\cdot \delb_a u\right)
\\
&+2s\delb_su\cdot \left(\delb_su + (2x^b/s)\delb_bu\right) - \sum_a\delb_a\left(s^2\left(\delb_su + (2x^b/s)\delb_bu\right)\cdot\delb_au\right)
\\
=& \frac{1}{2}\delb_s\left(|Ku|^2\right) + \frac{1}{2}\sum_a\delb_s\left(|s\delb_au|^2\right)
\\
& + 2s\left(\delb_su\cdot \left(\delb_su + (2x^b/s)\delb_bu\right) - \sum_a|\delb_au|^2\right)
 \\
 &+ \sum_a\delb_b\left(sx^b|\delb_au|^2\right) - \sum_a\delb_a\left(s^2\left(\delb_su + (2x^b/s)\delb_bu\right)\cdot\delb_au\right).
\endaligned
\end{equation}

We also remark that

$$
\aligned
u\cdot\Box u =&\del_s(u\delb_s u) - (\delb_su)^2 + s^{-1}\delb_s\left(2x^au\delb_a u\right)- (2x^a/s)\delb_au\delb_su
\\
&- \sum_a\delb_a(u\delb_au) + \sum_a(\delb_au)^2
+ \frac{n}{2}s^{-1}\delb_s (u^2)
\\
=&s^{-1}\delb_s(su\delb_su) + s^{-1}\delb_s\left(u^2\right) + s^{-1}\delb_s\left(2x^au\delb_au\right) - \sum_a\delb_a\left(u\delb_au\right)
\\
&-\left(|\delb_su|^2 + (2x^a/s)\delb_su\delb_au - \sum_a|\delb_au|^2 \right)
\endaligned
$$
thus
$$
\aligned
su\cdot\Box u =& \delb_s(su\delb_su) + \delb_s(u^2) + \delb_s\left(2x^a u\delb_au\right) - \sum_a\delb_a(su\delb_au)
\\
&-s\left((\delb_su)^2 + 2(x^a/s)\delb_su\delb_au - \sum_a(\delb_au)^2\right).
\endaligned
$$

So we see that
\begin{equation}\label{eq 2 10-05-2017}
\aligned
&2su\cdot \Box u + 2s\left((\delb_su)^2 + 2(x^a/s)\delb_su\delb_au - \sum_a(\delb_au)^2\right)
\\
=& \delb_s\left(2su\delb_su + 4x^au\delb_au+ 2u^2\right)
- \sum_a\delb_a\left(2su\delb_au\right).
\endaligned
\end{equation}
%

Then we combine \eqref{eq1-10-05-2017} and \eqref{eq 2 10-05-2017},
\begin{equation}\label{eq 3 10-05-2017}
2s\left(Ku + 2u\right)\Box u = \delb_s\left(|Ku|^2+\sum_a|s\delb_au|^2 + 4uKu + 4u^2\right) + \delb_a(v^a)
\end{equation}
with
$$
v^a := -4su\delb_au + 2sx^a\sum_{b}|\delb_bu|^2 - 2s\delb_au\cdot Ku
$$

We define the {\bf flat energy}
\begin{equation}\label{eq 4 10-05-2017}
\aligned
\Ec(s,u):=&\int_{\Hcal_s}\left((Ku)^2+\sum_a(s\delb_au)^2 + 4uKu + 4u^2\right)dx
\\
=&\int_{\Hcal_s}\left(\sum_a(s\delb_au)^2 +\left(Ku + 2u\right)^2\right)dx
\endaligned
\end{equation}
Then \eqref{eq 3 10-05-2017} leads to 
$$
2\int_{\Hcal_s}s\left(K+2\right)u\Box u dx = \frac{d}{ds}\Ec(s,u)
$$
Remark that $\|(Ku+2u)\|_{L^2(\Hcal_s)}\leq \Ec(s,u)^{1/2}$,
$$
\frac{d}{ds}\Ec(s,u)^{1/2}\leq 2s\|\Box u\|_{L^2(\Hcal_s)}
$$
which leads to
\begin{equation}\label{eq 5 10-05-2017}
\Ec(s,u)^{1/2}\leq \Ec(s_0,u)^{1/2} + 2\int_{s_0}^s\tau\|\Box u\|_{L^2(\Hcal_\tau)} d\tau.
\end{equation}
We conclude the above calculation by the following {\bf conformal energy estimate in flat case}
\begin{proposition}
Let $u$ be a function defined in $\Kcal_{[s_0,s_1]}$, sufficiently regular and vanishes near the conical boundary $\del\Kcal := \{(t,x)|t = r+1\}$. Then the estimate \eqref{eq 5 10-05-2017} holds.
\end{proposition}

\subsection{Analysis on the flat energy}
It is clear that the flat energy can control the following quantities:
$$
\|Ku + 2u\|_{L^2(\Hcal_s)},\quad s\|\delb_au\|_{L^2(\Hcal_s)}
$$
and in this subsection we will prove the following proporty:
\begin{proposition}
Let $n=3$ and $u$ be a sufficiently regular function defined in the region $\Kcal_{[s_0,s_1]}$ and vanishes near the conical boundary. Then the following inequality holds for $s_0 \leq s\leq s_1$:
\begin{equation}\label{eq 6 23-05-2017}
\|(s/r)u\|_{L^2(\Hcal_s)}\leq 2\Ec(s,u)^{1/2},
\end{equation}
\begin{equation}\label{eq 1 25-05-2017}
\|s^2t^{-1}\delb_su\|_{L^2(\Hcal_s)}\leq C\Ec(s,u)^{1/2}
\end{equation}
and
\begin{equation}\label{eq 7 22-07-2017}
\|s(s/t)^2\del_a u\|_{L^2(\Hcal_s)}\leq C\Ec(s,u)^{1/2}
\end{equation}
with $C$ a universal constant.
\end{proposition}
\begin{proof}
For \eqref{eq 6 23-05-2017}, we recall the Hardy's inequality for $\RR^3$. Let $w$ be a sufficiently regular function defined in $\RR^3$ and decreases sufficiently fast at infinity, then
\begin{equation}\label{eq 2 25-05-2017}
\|w/r\|_{L^2(\RR^3)}\leq 2\sum_a^3\|\del_a w\|_{L^2(\RR^3)}.
\end{equation}
We define (for a fixed $s$) $\tilde{u}_s:\RR^3\rightarrow \RR$:
$$
\tilde{u}_s(x) := su(\sqrt{s^2+r^2},x).
$$
We see that
$$
\del_a \tilde{u}_s = s\delb_au
$$
and we apply \eqref{eq 2 25-05-2017} with $w = \tilde{u}_s$ (remark that $\tilde{u}_s$ is compactly supported), this leads to
$$
\|su/r\|_{L^2(\Hcal_s)}\leq 2\sum_a\|s\delb_au\|_{L^2(\Hcal_s)}\leq 2\Ec(s,u)^{1/2}
$$
which is \eqref{eq 6 23-05-2017}.

To establish \eqref{eq 1 25-05-2017}, we remark that
$$
Ku + 2u = s\delb_su + 2x^a\delb_au + 2u
$$
which leads to
$$
\aligned
(s^2/t)|\delb_su|\leq (s/t)|Ku+2u| + 2\sum_a|(x^a/t)s\delb_au| + 2|(s/t)u|
\endaligned
$$
thus
$$
\aligned
\|(s^2/t)\delb_su\|_{L^2(\Hcal_s)}\leq& \|(s/t)(Ku+2u)\|_{L^2(\Hcal_s)} + 2\sum_a\left\|s\delb_a u\right\|_{L^2(\Hcal_s)} + 2\|(s/r)u\|_{L^2(\Hcal_s)}
\\
\leq& 7\Ec(s,u)^{1/2}.
\endaligned
$$

For the third one, we remark that
$$
s(s/t)^2\del_a u = s(s/t)^2\delb_au - s(s/t)^2(x^a/t)\del_tu
$$
which leads to
$$
\|s(s/t)^2\del_au\|_{L^2(\Hcal_s)}\leq \|s\delb_au\|_{L^2(\Hcal_s)} + \|s(s/t)\delb_su\|_{L^2(\Hcal_s)}\leq C\Ec(s,u)^{1/2}.
$$

\end{proof}
\begin{remark}
We see that when $n=3$ and when the flat energy is satisfies the following increasing condition:
$$
\Ec(s,u)^{1/2}\leq Cs^{\delta}
$$
then
$$
\|sr^{-1}u\|_{L^2(\Hcal_s)}\leq  Cs^{\delta}.
$$
\end{remark}

\section{The conformal inequality on hyperboloids: curved case}
\subsection{Differential identity}\label{subsec energy-indentity}
We suppose that $g^{\alpha\beta}$ is a (symmetric) Lorantzian metric defined in $\Kcal$, sufficiently regular, and coincides with the Minkowski metric near the light cone $\del\Kcal$. Then we remark the following calculation:
$$
\aligned
g^{\alpha\beta}\del_{\alpha}\del_{\beta}u =& \gb^{\alpha\beta}\delb_{\alpha}\delb_{\beta}u + g^{\alpha\beta}\del_{\alpha}\left(\Psib_{\beta}^{\beta'}\right)\delb_{\beta'}u
\\
=& \gb^{\alpha\beta}\delb_{\alpha}\delb_{\beta}u + g^{\alpha\beta}\del_{\alpha}\left(\Psib_\beta^0\right)\delb_su.
\endaligned
$$
and we thus have
$$
\aligned
g^{\alpha\beta}\del_{\alpha}\del_{\beta}u =& \delb_s\left(\gb^{00}\delb_su + 2\gb^{0a}\delb_au\right) - \delb_s\gb^{00}\delb_su - 2\delb_s\gb^{a0}\delb_au
\\
&+\gb^{ab}\delb_a\delb_bu + g^{\alpha\beta}\del_{\alpha}\left(\Psib_\beta^0\right)\delb_su
\\
=&\delb_s\left(\gb^{00}\delb_su + 2\gb^{0a}\delb_au\right) + s^{-1}\left(\gb^{00}\delb_su + 2\gb^{a0}\delb_au\right) + \gb^{ab}\delb_a\delb_bu
\\
&-\del_s\gb^{00}\delb_su - 2s^{-1}\left(\gb^{a0} + s\delb_s\gb^{a0}\right)\delb_au
+ \left( g^{\alpha\beta}\del_{\alpha}\left(\Psib_\beta^0\right) - s^{-1}\gb^{00}\right)\delb_su
\\
=&s^{-1}\delb_s\left(s\left(\gb^{00}\delb_su + 2\gb^{0a}\delb_au\right)\right) + \gb^{ab}\delb_a\delb_bu
\\
&-\delb_s\gb^{00}\delb_su - 2s^{-1}\left(\gb^{a0} + s\delb_s\gb^{a0}\right)\delb_au
+ \left( g^{\alpha\beta}\del_{\alpha}\left(\Psib_\beta^0\right) - s^{-1}\gb^{00}\right)\delb_su.
\endaligned
$$
We denote by
$$
\mathscr{K}_g = s\left(\gb^{00}\delb_s + 2\gb^{a0}\delb_a\right).
$$
and we calculate the following relations:
$$
\gb^{0a} + s\del_s\gb^{0a} = s\Psib_{\alpha}^0\del_sg^{a\alpha} + (s/t)g^{0a}
$$
and
$$
\del_s\gb^{00} = \Psib_{\alpha}^0\Psib_{\beta}^0\del_sg^{\alpha\beta} - 2s^{-3}\left(r^2\gb^{00} + x^ax^bg^{ab}\right) - 2g^{a0}\Psib_a^0\frac{t^2+r^2}{ts^2}.
$$
and
\begin{equation}\label{eq 3 25-05-2017}
\aligned
g^{\alpha\beta}\del_{\alpha}\Psib_{\beta}^0 =& -s^{-1}\left(g^{00}(r/s)^2 - 2g^{a0}\frac{t}{s}\cdot\frac{x^a}{s} + g^{ab}\left(\frac{x^a}{s}\cdot\frac{x^b}{s} + \delta_{ab}\right)\right)
\\
=& -s^{-1}\left(g^{00}(t/s)^2 - 2g^{a0}\frac{t}{s}\cdot\frac{x^a}{s} + g^{ab}\frac{x^a}{s}\cdot\frac{x^b}{s}\right) + s^{-1}\left(g^{00} - \sum_ag^{aa}\right)
\\
=& -s^{-1}\gb^{00} + s^{-1}\left(g^{00} - \sum_ag^{aa}\right).
\endaligned
\end{equation}

Then as the flat case, we use the multiplier $s\mathscr{K}_gu$ and see that
$$
\aligned
s\mathscr{K}_gu\cdot s^{-1}\del_s\left(\mathscr{K}_gu\right) = \frac{1}{2}\delb_s\left(|\mathscr{K}_gu|^2\right)
\endaligned
$$
$$
\aligned
s\mathscr{K}_gu\cdot \gb^{ab}\delb_a\delb_bu =& s\delb_a\left(\mathscr{K}_gu\cdot \gb^{ab}\delb_bu\right) - s\delb_a(\mathscr{K}_gu)\gb^{ab}\delb_bu - s\delb_a\gb^{ab}\cdot \mathscr{K}_gu\delb_bu
\endaligned
$$
where
$$
\aligned
s\delb_a(\mathscr{K}_gu)\gb^{ab}\delb_bu =& s^2\delb_a\left(\gb^{00}\delb_su + 2\gb^{c0}\delb_cu\right)\cdot \gb^{ab}\delb_bu
\\
=&s^2\gb^{00}\delb_a\delb_su\cdot\gb^{ab}\delb_bu + 2s^2\gb^{c0}\delb_a\delb_cu\cdot \gb^{ab}\delb_bu
\\
&+s^2\delb_a\gb^{00}\cdot \gb^{ab}\delb_su\delb_bu + 2s^2\delb_a\gb^{c0}\cdot \gb^{ab}\delb_cu\delb_bu
\\
=&\frac{s^2}{2}\delb_s\left(\gb^{00}\gb^{ab}\delb_au\delb_bu\right) + s^2\delb_c\left(\gb^{c0}\gb^{ab}\delb_au\delb_bu\right)
\\
&-\frac{s^2}{2}\delb_s\left(\gb^{00}\gb^{ab}\right)\delb_au\delb_bu - s^2\delb_c\left(\gb^{c0}\gb^{ab}\right)\delb_au\delb_bu
\\
&+s^2\delb_a\gb^{00}\cdot \gb^{ab}\delb_su\delb_bu + 2s^2\delb_a\gb^{c0}\cdot \gb^{ab}\delb_cu\delb_bu
\\
=&\frac{1}{2}\delb_s\left(s^2\gb^{00}\gb^{ab}\delb_au\delb_bu\right) + \delb_c\left(s^2\gb^{c0}\gb^{ab}\delb_au\delb_bu\right)
\\
&-s\gb^{00}\gb^{ab}\delb_au\delb_bu - \frac{s^2}{2}\delb_s\left(\gb^{00}\gb^{ab}\right)\delb_au\delb_bu -s^2\delb_c\left(\gb^{c0}\gb^{ab}\right)\delb_au\delb_bu
\\
&+s^2\delb_a\gb^{00}\cdot \gb^{ab}\delb_su\delb_bu + 2s^2\delb_a\gb^{c0}\cdot \gb^{ab}\delb_cu\delb_bu,
\endaligned
$$
$$
s\mathscr{K}_gu\cdot \left(g^{\alpha\beta}\del_{\alpha}\left(\Psib_{\beta}^0\right)-s^{-1}\gb^{00}\right)\delb_su = s\left(sg^{\alpha\beta}\del_{\alpha}\left(\Psib_{\beta}^0\right) - \gb^{00}\right)\left(\gb^{00}\delb_su\delb_su + 2\gb^{a0}\delb_au\delb_su\right).
$$

Thus we see that
\begin{equation}\label{eq 1 21-05-2017}
\aligned
s\mathscr{K}_gu\cdot g^{\alpha\beta}\del_{\alpha}\del_{\beta}u =& \frac{1}{2}\delb_s\left(|\mathscr{K}_gu|^2 - s^2\gb^{00}\gb^{ab}\delb_au\delb_bu \right) + \delb_a(v_g^a)
\\
&+ s\left(\gb^{00}\gb^{ab} + s\delb_c\left(\gb^{0c}\gb^{ab}\right) - 2s\delb_c\gb^{0a}\cdot\gb^{cb}\right)\delb_au\delb_bu
\\
&+\left(sg^{\alpha\beta}\del_{\alpha}\left(\Psib_{\beta}^0\right) - \gb^{00}\right)\mathscr{K}_gu\cdot \delb_su
\\
&\quad + \frac{s^2}{2}\delb_s\left(\gb^{00}\gb^{ab}\right)\delb_au\delb_bu - s^2\delb_a\gb^{00}\cdot\gb^{ab}\delb_su\delb_bu
\\
&\quad - s\mathscr{K}_gu\cdot \delb_s\gb^{00}\delb_su - 2\mathscr{K}_gu\cdot \delb_s\left(s\gb^{a0}\right)\cdot \delb_au
- s\delb_a\gb^{ab}\cdot \delb_b u \mathscr{K}_gu
\endaligned
\end{equation}
with
$$
v_g^a = s\mathscr{K}_g\cdot \gb^{ab}\delb_b u - s^2\gb^{a0}\gb^{cb}\delb_cu\delb_bu,
$$
$$
\aligned
su\cdot g^{\alpha\beta}\del_{\alpha}\del_{\beta}u =& su\cdot \gb^{\alpha\beta}\delb_{\alpha}\delb_{\beta}u + su\cdot g^{\alpha\beta}\del_{\alpha}\left(\Psib_{\beta}^0\right)\delb_su
\\
=&\delb_s\left(u\cdot \mathscr{K}_g u + \frac{1}{2}\left(sg^{\alpha\beta}\del_{\alpha}\left(\Psib_\beta^0\right)-\delb_s(s\gb^{00})\right)u^2\right)
+\delb_a\left(s u\gb^{ab}\delb_bu\right)
\\
&-s\gb^{\alpha\beta}\delb_{\alpha}u\delb_{\beta}u
\\
&+\frac{1}{2}\delb_s\left(\delb_s(s\gb^{00})-sg^{\alpha\beta}\del_{\alpha}\left(\Psib_\beta^0\right)\right)u^2
-2\delb_s\left(s\gb^{a0}\right)u\delb_au - \delb_a\left(s\gb^{ab}\right)u\delb_bu.
\endaligned
$$
We denote by
$$
N_g := sg^{\alpha\beta}\del_{\alpha}\left(\Psib_\beta^0\right) - \delb_s(s\gb^{00})
$$
and by \eqref{eq 3 25-05-2017} we see that
$$
N_g = g^{00} - \sum_ag^{aa} - 2\gb^{00} - s\delb_s\gb^{00}.
$$
When $g^{\alpha\beta}  = m^{\alpha\beta}$, we see that $N_m = n-1$.

Then we see that
\begin{equation}\label{eq 2 21-05-2017}
\aligned
N_gsu\cdot g^{\alpha\beta}\del_{\alpha}\del_{\beta}u
=& -N_gs\gb^{\alpha\beta}\delb_{\alpha}u\delb_{\beta}u
+\delb_s\left(N_gu\cdot \mathscr{K}_gu + \frac{1}{2}N_g^2u^2\right)
+\delb_a\left(N_g\cdot s u\gb^{ab}\delb_bu\right)
\\
&-\delb_sN_g\cdot \left(u\cdot \mathscr{K}_gu + \frac{1}{2}N_gu^2\right) - \delb_aN_g\cdot su\gb^{ab}\delb_bu
\\
&-\frac{1}{2}N_g\delb_sN_g\cdot u^2 -N_g\left(2\delb_s\left(s\gb^{a0}\right)u\delb_au + \delb_a\left(s\gb^{ab}\right)u\delb_bu\right).
\endaligned
\end{equation}

We combine \eqref{eq 1 21-05-2017} and \eqref{eq 2 21-05-2017}, and see that
\begin{equation}\label{eq 3 21-05-2017}
\aligned
s\left(\mathscr{K}_gu + N_gu\right)\cdot g^{\alpha\beta}\del_{\alpha}\del_{\beta}u
=& \frac{1}{2}\delb_s\left(|\mathscr{K}_gu + N_gu|^2  - s^2\gb^{00}\gb^{ab}\delb_au\delb_bu\right)
+ \delb_a(w_g^a)
\\
& + R_g(\nabla u,\nabla u) + S_g[\nabla u]\cdot (\mathscr{K}_g+N_g)u + T_g[u]
\endaligned
\end{equation}
where $R_g(\nabla u,\nabla u)$ is a quadratic form acting on the gradient of $u$:
$$
\aligned
R_g(\nabla u,\nabla u) :=& s\left(\gb^{00}\gb^{ab} + s\delb_c\left(\gb^{0c}\gb^{ab}\right) - 2s\delb_c\gb^{0a}\cdot\gb^{cb}\right)\delb_au\delb_bu
\\
&+(sg^{\alpha\beta}\del_{\alpha}\Psib_{\beta}^0-\delb_s(s\gb^{00}))\mathscr{K}_gu\cdot \delb_su - N_gs\gb^{\alpha\beta}\delb_{\alpha}u\delb_{\beta}u,
\\
& + \frac{s^2}{2}\delb_s\left(\gb^{00}\gb^{ab}\right)\delb_au\delb_bu - s^2\delb_a\gb^{00}\gb^{ab}\delb_su\delb_bu
\\
=&sL_g^{ab}\delb_a u \delb_bu + N_g \mathscr{K}_gu\cdot \delb_su - N_gs\gb^{\alpha\beta}\delb_{\alpha}u\delb_{\beta}u
\\
& + \frac{s^2}{2}\delb_s\left(\gb^{00}\gb^{ab}\right)\delb_au\delb_bu - s^2\delb_a\gb^{00}\gb^{ab}\delb_su\delb_bu
\endaligned
$$
with
$$
L_g^{ab}: = \gb^{00}\gb^{ab} + s\delb_c\left(\gb^{0c}\gb^{ab}\right) - 2s\delb_c\gb^{0a}\cdot\gb^{cb},
$$
$$
\aligned
\mathscr{K}_g u\cdot S_g[\nabla u] :=& - (\mathscr{K}_g + N_g)u\cdot\left(2\delb_s(s\gb^{a0})\delb_au + s\delb_a\gb^{ab}\delb_b u\right)
\endaligned
$$
and
$$
\aligned
T_g[u] := -\delb_sN_g\cdot u\left(\mathscr{K}_g+N_g\right)u  - su\cdot \gb^{ab}\delb_aN_g\delb_b u.
\endaligned
$$
In \eqref{eq 3 21-05-2017},
$$
\aligned
w_g^a =& v_g^a + N_gsu\cdot \gb^{ab}\delb_bu
\\
=& s\mathscr{K}_g\cdot \gb^{ab}\delb_b u - s^2\gb^{a0}\gb^{cb}\delb_cu\delb_bu + N_gsu\cdot \gb^{ab}\delb_bu.
\endaligned
$$

We analyse the structure of $R_g(\nabla u,\nabla u)$:  remark the coefficient $L_g^{ab}$ satisfies the following property: when $g^{\alpha\beta} = m^{\alpha\beta}$,
$$
L_m^{ab} = 2\mb^{ab}.
$$
We also see that
$$
\aligned
L_g^{ab} - L_m^{ab} =& \gb^{00}\gb^{ab} - \mb^{00}\mb^{ab} + s\delb_c\left(\gb^{0c}\gb^{ab} - \mb^{0c}\mb^{ab}\right) - 2s\delb_c\gb^{0a}\gb^{cb} + 2s\delb_c\mb^{0a}\mb^{cb}\\
\\
=&\gb^{00}\hb^{ab} + \hb^{00}\mb^{ab} + s\delb_c\left(\hb^{0c}\gb^{ab}\right) + s\delb_c\left(\hb^{ab}\mb^{0c}\right) - 2s\left(\delb_c\gb^{0a}\hb^{cb} + \delb_c\hb^{0a}\mb^{cb}\right)
\endaligned
$$
On the other hand, we see that
$$
N_m = 2
$$
and
$$
N_g - N_m = h^{00} - \sum_ah^{aa} - 2\hb^{00} - s\delb_s\hb^{00}
$$
and
$$
\aligned
sN_g\gb^{\alpha\beta}\delb_{\alpha}u\delb_{\beta}u =& sN_g\left(\gb^{00}\delb_su + 2\gb^{a0}\delb_a u\right)\delb_su + sN_g\gb^{ab}\delb_au\delb_bu
=N_g\mathscr{K}_gu \delb_su + sN_g\gb^{ab}\delb_au\delb_bu.
\endaligned
$$
Thus we see that
\begin{equation}\label{eq R(du,du)}
\aligned
R_g(\nabla u,\nabla u) =& s\left(L_g^{ab}-L_m^{ab}\right)\delb_au\delb_bu  - s\left(N_g\gb^{ab}-2\mb^{ab}\right)\delb_au\delb_bu
\\
& + \frac{s^2}{2}\delb_s\left(\gb^{00}\gb^{ab}\right)\delb_au\delb_bu - s^2\delb_a\gb^{00}\gb^{ab}\delb_su\delb_bu
\\
=& s\left(L_g^{ab}-L_m^{ab}\right)\delb_au\delb_bu - s(N_g-N_m)\gb^{ab}\delb_au\delb_bu - 2s\hb^{ab}\delb_au\delb_bu
\\
&+ \frac{s^2}{2}\delb_s\left(\gb^{00}\gb^{ab}\right)\delb_au\delb_bu - s^2\delb_a\gb^{00}\gb^{ab}\delb_su\delb_bu.
\endaligned
\end{equation}

\subsection{Energy estimate in curved case}
We integrate the identity \eqref{eq 2 21-05-2017} in the region $\Kcal_{[s_0,s]}$, remark that we suppose that $u$ is sufficiently regular and vanishes near the conical boundary. By Stokes formula:
\begin{equation}\label{eq 1 23-05-2017}
\aligned
\int_{\Kcal_{[s_0,s]}} s\left(\mathscr{K}_gu + N_gu\right)\cdot g^{\alpha\beta}\del_{\alpha}\del_{\beta}u \ dxds =& \frac{1}{2}\left(E_{\text{con},g}(s,u) - E_{\text{con},g}(s_0,u)\right)
\\
&+ \int_{\Kcal_{[s_0,s]}}\left(R_g(\nabla u,\nabla u) + \mathscr{K}_g u\cdot S_g[\nabla u] + T_g[u]\right)\,dxds
\endaligned
\end{equation}
with
$$
E_{\text{con},g} := \int_{\Hcal_s}\left(|\mathscr{K}_gu + N_g u|^2 - s^2\gb^{ab}\delb_au\delb_bu\right)\,dxds,
$$
called the {\bf curved energy}.

Then we derive \eqref{eq 1 23-05-2017} with respect to $s$ and see that
$$
\aligned
\int_{\Hcal_s}s\left(\mathscr{K}_gu + N_gu\right)\cdot g^{\alpha\beta}\del_{\alpha}\del_{\beta}u \ dx =& \frac{d}{2ds}E_{\text{con},g}(s,u)
\\
&+ \int_{\Hcal_s}\left(R_g(\nabla u,\nabla u) + \mathscr{K}_gu\cdot S_g[u] + T_g[u]\right)\ dx
\endaligned
$$
thus
\begin{equation}\label{eq 2 23-05-2017}
\aligned
E_{\text{con},g}(s,u)^{1/2}\frac{d}{ds}E_{\text{con},g}(s,u)^{1/2}
\leq& \|s\left(\mathscr{K}_gu + N_gu\right)\cdot g^{\alpha\beta}\del_{\alpha}\del_{\beta}u\|_{L^1(\Hcal_s)}
\\
& + \|R_g(\nabla u,\nabla u) + \mathscr{K}_g u\cdot S_g[u] + T_g[u]\|_{L^1(\Hcal_s)}
\endaligned
\end{equation}
We suppose that there exists a $\kappa\geq 1$ such that
\begin{equation}\label{eq 3 23-05-2017}
\kappa^{-1} E_{\text{con},g}(s,u)^{1/2}\leq \Ec(s,u)^{1/2} \leq \kappa E_{\text{con},g}(s,u)^{1/2} ,
\end{equation}
\begin{equation}\label{eq 4 23-05-2017}
\|R_g(\nabla u,\nabla u) + \mathscr{K}_gu\cdot S_g[u] + T_g[u]\|_{L^1(\Hcal_s)}\leq \Ec(s,u)^{1/2}M_g(s,u).
\end{equation}
Then we see that from \eqref{eq 2 23-05-2017}
$$
\frac{d}{ds}E_{\text{con},g}(s,u)^{1/2}\leq \|sg^{\alpha\beta}\del_{\alpha}\del_{\beta}u\|_{L^2(\Hcal_s)} + M_g(s,u)
$$
Thus
\begin{equation}\label{eq 5 23-05-2017}
E_{\text{con},g}(s,u)^{1/2}\leq E_{\text{con},g}(s_0,u)^{1/2} + \int_{s_0}^s\left(\tau\|g^{\alpha\beta}\del_{\alpha}\del_{\beta}u\|_{L^2(\Hcal_\tau)} + M_g(\tau,u)\right)\, d\tau
\end{equation}
which (combined with \eqref{eq 3 23-05-2017}) leads to
\begin{equation}\label{eq energy}
\Ec(s,u)^{1/2}\leq \kappa^2\Ec(s_0,u)^{1/2} + \kappa\int_{s_0}^s\left(\tau\|g^{\alpha\beta}\del_{\alpha}\del_{\beta}u\|_{L^2(\Hcal_\tau)} + M_g(\tau,u)\right)\, d\tau.
\end{equation}
which is the conformal energy estimate in curved case.

\subsection{Analysis on curved energy}
In this subsection we will analyse the structure of the curved energy and more precisely, we will give a sufficient condition for \eqref{eq 3 23-05-2017} .
\begin{proposition}\label{prop 1 24-05-2017}
There exists a constant $\vep_s$ (with the $s$ stands for ``structure'') which depends only on $n$ such that when
\begin{equation}\label{eq 1 24-05-2017}
\left\{
\aligned
&|\hb^{00}|\leq (s/t)\vep_s,\quad |s\delb_s\hb^{00}|\leq \vep_s(s/t),
\\
&|h^{\alpha\beta}|\leq (s/t)\vep_s.
\endaligned
\right.
\end{equation}
then \eqref{eq 3 23-05-2017} holds.
\end{proposition}
\begin{proof}
Recall the structure of the curved energy, we consider first the term $\gb^{ab}\delb_au\delb_bu$. We see that
$$
\gb^{ab} = g^{\alpha\beta}\Psib_{\alpha}^a\Psib_{\beta}^b = g^{ab},
$$
in the same way $\hb^{ab} = h^{ab}$.
Thus there exists a positive constant $\vep_s$ such that if $|h^{ab}|\leq \vep_s$,
\begin{equation}\label{eq 1 16-08-2017}
\big\|s^2\gb^{ab}\delb_au\delb_bu - \sum_a|s\delb_au|^2\big\|_{L^2(\Hcal_s)}\leq C\vep_s \left\|\sum_a|s\delb_au|^2\right\|_{L^2(\Hcal_s)}
\leq C\vep_s \Ec(s,u)^{1/2}.
\end{equation}

Then we regard the first term in $E_{\text{con},g}$. We first remark that
$$
\aligned
N_g - 2 =& g^{00} - \sum_ag^{aa} - 2\gb^{00} - s\delb_s\gb^{00} - \left(m^{00} - \sum_am^{aa} - 2\mb^{00} - s\delb_s\mb^{00}\right)
\\
=& h^{00} - \sum_ah^{aa} - 2\hb^{00} - s\delb_s\hb^{00}.
\endaligned
$$
Then we see that (taking \eqref{eq 1 24-05-2017} with $\vep_s$ sufficiently small)
\begin{equation}\label{eq 1 18-08-2017}
|N_g|\leq 2 + \left|h^{00} - \sum_ah^{aa} - 2\hb^{00} - s\delb_s\hb^{00}\right|\leq 3.
\end{equation}

We see that
\begin{equation}\label{eq 2 24-05-2017}
\aligned
\mathscr{K}_g u + N_g u - \gb^{00}\left(Ku + 2u\right)
=& 2s\left(\gb^{a0} -\gb^{00}\mb^{a0} \right)\delb_au + \left(N_g - 2\gb^{00}\right)u
\\
=& 2s\left(\gb^{a0} - \mb^{a0} + \mb^{00}\mb^{a0} - \gb^{00}\mb^{a0}\right)\delb_au
+ \left(N_g - 2\gb^{00}\right)u
\\
=&2s\left(\hb^{a0}+\hb^{00}\mb^{a0}\right)\delb_au + \left(N_g - 2\gb^{00}\right)u
\endaligned
\end{equation}
where we remark that
$$
\hb^{\alpha\beta} = \gb^{\alpha\beta} - \mb^{\alpha\beta},\quad \text{and especially}\quad \hb^{00} = \gb^{00} - 1.
$$

Thus when $|\hb^{a0}|\leq \vep_s$ and $|\hb^{00}|\leq \vep_s (s/t)$, we see that for the first term in right-hand-side of \eqref{eq 2 24-05-2017}:
$$
\aligned
2s\left|\hb^{a0}+\hb^{00}\mb^{a0}\right|\delb_au \leq & C\vep_s\sum_a|s\delb_a u|.
\endaligned
$$
We remark that
$$
\hb^{a0} = h^{\alpha\beta}\Psib_{\alpha}^a\Psib_{\beta}^0 = h^{a\beta}\Psib_{\beta}^0 = h^{a0}(t/s) - h^{ab}(x^b/s)
$$
thus by \eqref{eq 1 24-05-2017}, $|\hb^{a0}|\leq C\vep_s$.

On the other hand, we see that by \eqref{eq 3 25-05-2017},
$$
\aligned
\left(N_g - 2\gb^{00}\right)u =& \left(-2\hb^{00} +  h^{00} - \sum_ah^{aa} - 2\hb^{00} - s\delb_s\hb^{00}\right)u
\\
=& \left(-4\hb^{00} + h^{00} - \sum_a\hb^{aa} - s\delb_s\hb^{00} \right)u.
\endaligned
$$
We see that under the assumption \eqref{eq 1 24-05-2017},
$$
\left|\left(N_g - 2\gb^{00}\right)u\right|\leq C\vep_s (s/t)|u|
$$
thus
$$
\left|\mathscr{K}_g u + N_g u - \gb^{00}\left(Ku + 2u\right)\right|\leq C\vep_s\sum_a|s\delb_au| + C\vep_s (s/t)|u|.
$$
Then we see that
\begin{equation}\label{eq 3 16-08-2017}
\aligned
\left|(\mathscr{K}_gu + N_gu) - (Ku+2u)\right|\leq& \left|(\mathscr{K}_gu + N_gu) - \gb^{00}(Ku+2u)\right| + |\hb^{00}||Ku+2u|
\\
\leq & C\vep_s\sum_a|s\delu_au| + C\vep_s(s/t)|u| + C\vep_s|Ku+2u|
\endaligned
\end{equation}
This leads to (recall \eqref{eq 6 23-05-2017}):
\begin{equation}\label{eq 2 16-08-2017}
\left\|(\mathscr{K}_gu + N_gu) - \left(Ku + 2u\right)\right\|_{L^2(\Hcal_s)}\leq C\vep_s\Ec(s,u)^{1/2}
\end{equation}
Then we see that (by \eqref{eq 1 16-08-2017} and \eqref{eq 2 16-08-2017})
\begin{equation}\label{eq 4 16-08-2017}
(1-C\vep_s)\Ec(s,u)^{1/2}\leq E_{\text{con},g}(s,u)^{1/2} \leq (1+C\vep_s)\Ec(s,u)^{1/2}
\end{equation}
Then we take $\vep_s$ sufficiently small and the desired result is established.
\end{proof}

\section{Commutators and decay estimate}
\subsection{Global Sobolev inequality on hyperboloids}
In this subsection we recall the Klainerman-Sobolev inequality on hyperboloids due to H\"ormander \cite{H1}:
\begin{proposition}
Let $u$ be a sufficiently regular function defined in $\Kcal$ and vanishes near the conical boundary $\del_K\Kcal$. Then the following estimate holds:
\begin{equation}\label{ineq 1 sobolev}
\sup_{\Hcal_s}\big(t^{3/2}|u|\big)\leq C\sum_{|I|+|J|\leq 2}\|\del^IL^Ju\|_{L^2(\Hcal_s)}.
\end{equation}
Here $C$ is a positive constant independent of $u$.
\end{proposition}
This result is essentially due to H\"ormander (see in detail \cite{H1}(1997), Lemma 7.6.1). For this slightly improved version, see \cite{PLF-MY-book} section 5.1. This inequality helps us to get decay estimate via $L^2$ norm. So in the following we need to control on hyperboloids the $L^2$ norm of the following quantities for $|I'|+|J'|\leq 2$
$$
\del^{I'}L^{J'}\left((s^2/t)\delb_s\del^IL^J u\right),\quad \del^{I'}L^{J'}\left(s\delb_a\del^IL^J u\right),\quad \del^{I'}L^{J'}\left((s/t)\del^IL^Ju\right),
$$
$$
\del^{I'}L^{J'}\left(s^2\delb_s\delb_a\del^IL^Ju\right),\quad \del^{I'}L^{J'}\left(st\delb_a\delb_b\del^IL^Ju\right),
$$
$$
\del^{I'}L^{J'}\left((s^2/t)\del^IL^J\delb_s u\right),\quad \del^{I'}L^{J'} \left((s/t)\del^IL^Ju\right),\quad \del^{I'}L^{J'}\left(s\del^IL^J\delb_au\right). 
$$
and
$$
\del^{I'}L^{J'}\left(s^2\del^IL^J\delb_s\delb_au\right),\quad\del^{I'}L^{J'}\left(st\del^IL^J\delb_a\delb_b u\right).
$$
These are to be bounded by the conformal energies $\Ec(s,\del^{I''}L^{J''}u)$ with $|I''|+|J''|\leq |I|+|J|+2$. To do so, we need some estimates on commutators, which are studied in the following sections.
%
%

\subsection{Commutators I}
In this section we calculate the following quantities:
$$
[L_a,\del_{\alpha}], \quad [L_a,\del^I],\quad [L^J,\del^I],
$$
$$
[\del_{\alpha},\delb_\alpha],\quad [\del^I,\delb_{\alpha}]
$$
and
$$
[L_a,\delb_{\alpha}],\quad [L^J,\delb_{\alpha}].
$$

We begin with the first group. It is easy to see that
$$
[L_a,\del_t] = -\del_a,\quad [L_a,\del_b] = -\delta_{ab}\del_b
$$
and we denote by
$$
[L_a,\del_{\alpha}] = \theta_{a\alpha}^\beta \del_{\beta}
$$
where $\theta_{a\alpha}^{\beta}$ are constants.

{\bf Important convention:} in the following we often make summation over multi-indices of order less than an integer. For the convenience of expression, we only give the upper bound of the order of a multi-index and omit the lower bound. For example when we write
$$
\sum_{|I|\leq N}
$$
we always mean
$$
\sum_{0\leq|I|\leq N}
$$
which is a sum over all multi-index of order from zero to $N$. In certain case, we take the sum from a positive order. In this case we will write
$$
\sum_{n\leq |I|\leq N}.
$$
We make the convention that when $n>N$, this sum is taken as zero.

Then we have the following decompositions:
\begin{lemma}\label{lem 1 19-08-2017}
Let $u$ be a function defined in $\Kcal$, sufficiently regular. Then the following identities hold:
\begin{equation}\label{eq 1 19-08-2017}
[L_a,\del^I]u = \sum_{|I'|=|I|}\theta_{aI'}^{I}\del^{I'}u,
\end{equation}
\begin{equation}\label{eq 2 19-08-2017}
 [L^J,\del^I]u = \sum_{|I'|=|I|\atop |J'|<|J| }\theta_{I'J'}^{JI}\del^{I'}L^{J'}u,
\end{equation}
\begin{equation}\label{eq 6 22-08-2017}
[L^J,\del_\alpha]u = \sum_{|J'|<|J|}\theta_{\alpha J'}^{J\beta}\del_{\beta}L^{J'}.
\end{equation}
where $\theta_{aI'}^{I}$ ,$\theta_{I'J'}^{JI}$ and $\theta_{\alpha J'}^{J\beta}$ are constants.
\end{lemma}
\begin{proof}
These are by induction. For the first it is an induction on $|I|$. When $|I|=1$ it is already proved. We suppose that \eqref{eq 1 19-08-2017} holds for $|I|\leq m$. Then we consider
$$
\aligned
\,[L_a,\del_{\alpha}\del^I]u =& [L_a,\del_{\alpha}]\del^Iu + \del_{\alpha}[L_a,\del^I]u
\\
=&\theta_{a\alpha}^{\beta}\del_{\beta}\del^Iu + \sum_{|I'|=|I|}\del_{\alpha}\left(\theta_{aI'}^I\del^{I'}u\right)
=\theta_{a\alpha}^{\beta}\del_{\beta}\del^Iu + \sum_{|I'|=|I|}\theta_{aI'}^I\del_{\alpha}\del^{I'}u.
\endaligned
$$
which concludes \eqref{eq 1 19-08-2017} for $|I| = m+1$.

For \eqref{eq 2 19-08-2017}, it is an induction on $|J|$. For $|J|=1$ it is already proved. Suppose that \eqref{eq 2 19-08-2017} holds for $|J|\leq m$. Then we consider
$$
\aligned
\,[L^J L_a,\del^I]u =& L^J[L_a,\del^I]u + [L^J,\del^I]L_au
=L^J\left(\sum_{|I'|=|I|}\theta_{aI'}^{I}\del^{I'}u\right) + \sum_{|I'|=|I|\atop |J'|<|J|}\theta^{JI}_{I'J'}\del^{I'}L^{J'}L_au
\\
=&\sum_{|I'|=|I|}\theta_{aI'}^IL^J\del^{I'}u + \sum_{|I'|=|I|\atop |J'|<|J|}\theta^{JI}_{I'J'}\del^{I'}L^{J'}L_au
\\
=&\sum_{|I'|=|I|}\theta_{aI'}^I\del^{I'}L^Ju + \sum_{|I'|=|I|}\theta_{aI'}^I[L^J,\del^{I'}]u
+ \sum_{|I'|=|I|\atop |J'|<|J|}\theta^{JI}_{I'J'}\del^{I'}L^{J'}L_au
\\
=&\sum_{|I'|=|I|}\theta_{aI'}^I\del^{I'}L^J u
+ \sum_{|I'|=|I|}\theta_{aI'}^I\sum_{|I''|=|I'|\atop |J''|<|J|}\theta_{I''J''}^{JI'}\del^{I''}L^{J''} u
+ \sum_{|I'|=|I|\atop |J'|<|J|}\theta^{JI}_{I'J'}\del^{I'}L^{J'}L_au
\\
=&\sum_{|I'|=|I|}\theta_{aI'}^I\del^{I'}L^Ju
+ \sum_{{|I'|=|I|,|I''|=|I'|}\atop |J''|<|J|}\theta_{aI'}^I\theta_{I''J''}^{JI'}\del^{I''}L^{J''}u
+ \sum_{|I'|=|I|\atop |J'|<|J|}\theta^{JI}_{I'J'}\del^{I'}L^{J'}L_au.
\endaligned
$$
This concludes \eqref{eq 2 19-08-2017} for $|J| = m+1$.

For \eqref{eq 6 22-08-2017} is a direct result of \eqref{eq 6 22-08-2017}.
\end{proof}

Here we also state a simple but frequently used result:\begin{lemma}\label{lem 2 29-08-2017}
The following identity holds:
\begin{equation}\label{eq 6 29-08-2017}
\del^{I_1}L^{J_1}\del^{I_2}L^{J_2}u = \sum_{|I|= |I_1|+|I_2|\atop |J|\leq|J_1|+|J_2|}\zeta^{I_1I_2J_1J_2}_{IJ}\del^IL^Ju
\end{equation}
where $\zeta^{I_1I_2J_1J_2}_{IJ}$ are constants
\end{lemma}
\begin{proof}
This is by \eqref{eq 2 19-08-2017}. We see that
$$
\aligned
\del^{I_1}L^{J_1}\del^{I_2}L^{J_2}u =& \del^{I_1}\del^{I_2}L^{J_1}L^{J_2}u + \del^{I_1}\left([L^{J_1},\del^{I_2}]L^{J_2}u\right)
\\
=& \del^{I_1}\del^{I_2}L^{J_1}L^{J_2}u + \sum_{|I_2'|=|I_2|\atop |J_1'|<|J_1|}\theta_{I_2'J_1'}^{J_1I_2}\del^{I_1}\del^{I_2'}L^{J_1'}L^{J_2}u
\endaligned
$$
where $\theta_{I_2'J_1'}^{J_1I_2}$ are constants. This proves the desired result.
\end{proof}

Now we consider the commutator of $[\del_{\alpha},\delb_a]$. For the convenience of expression, we introduce the following notion of homogeneous function. Let $u$ be a $C^{\infty}$ function defined in $\{r<t\}$ and satisfies the following condition:
\begin{equation}\label{eq 2 26-06-2017}
u(\lambda,\lambda x) = \lambda^{n} u(1,x),\quad n\in \mathbb{Z}.
\end{equation}
for all $\lambda>0$ and $\del^I u(1,x)$ are bounded in $\{|x|\leq 1\}$. Such function $u$ is called a {\bf homogeneous function of degree $n$}.  For example $x^a/t$ is a homogeneous function of degree zero. We state the following property of a homogeneous function of degree $n$:
\begin{lemma}\label{lem 1 26-06-2017}
Let $u$ be a homogeneous function of degree $n$. Then $\del^IL^Ju$ is homogeneous of degree $n-|I|$, and the following estimates holds in $\Kcal$:
\begin{equation}\label{eq 1 lem 1 26-06-2017}
\left|\del^IL^J u\right|\leq Ct^{n-|I|}
\end{equation}
with $C$ a constant determined by $u, I$ and $J$.
\end{lemma}
\begin{proof}
We remak that
\begin{equation}\label{eq 1 proof lem 1 26-06-2017}
\left|\del^Iu\right|\leq Ct^{n-|I|}
\end{equation}
which is checked directly:
$$
\aligned
u(t,x) =& t^nu(1,x/t) \Rightarrow \del_t u(t,x) = nt^{n-1}u(1,x/t) + t^n(-x^a/t^2)\del_au(1,x/t)
\\
=& t^{n-1}\left(nu(1,x/t) - (x^a/t)\del_au(1,x/t)\right)
\endaligned
$$
and
$$
\del_au = t^n(-x^a/t^2)\del_au(1,x/t) = -t^{n-1}(x^a/t)\del_au(1,x/t)
$$
where we remark that
$$
nu(1,x/t) - (x^a/t)\del_au(1,x/t),\quad (x^a/t)\del_au(1,x/t)
$$
are homogeneous of degree zero.
Thus $\del_\alpha u$ is homogeneous of degree $n-1$. Then by recurrence we see that $\del^I u$ is homogeneous of degree $n-|I|$. This leads to \eqref{eq 1 proof lem 1 26-06-2017}.

Then we prove that
\begin{equation}\label{eq 2 proof lem 1 26-06-2017}
L^Ju \quad \text{is homogeneous of degree n}
\end{equation}
This is also checked directly by
$$
L_a u(t,x) = \left(t\del_a+x^a\del_t\right) u(t,x) = -t^{n}(x^a/t)\del_au(1,x/t) + t^n(x^a/t)\left(nu(1,x/t) - (x^a/t)\del_au(1,x/t)\right)
$$
which is homogeneous of degree $n$. Thus $L^Ju$ is also homogeneous of degree zero, which leads to \eqref{eq 2 proof lem 1 26-06-2017}.

The desired result is a combination of \eqref{eq 1 proof lem 1 26-06-2017} and \eqref{eq 2 proof lem 1 26-06-2017}.
\end{proof}

Now we see that
$$
[\del_t,\delb_a] = -\frac{x^a}{t^2}\del_t,\quad [\del_a,\delb_b] = \frac{\delta_{ab}}{t}\del_t,
$$
Then we denote by
$$
[\del_{\alpha},\delb_b] = \sigma_{\alpha b}\del_t 
$$
with $\sigma_{\alpha\beta}$  homogeneous functions of degree $-1$.
Then we establish the following result:
\begin{lemma}\label{lem 2 19-08-2017} For $|I|\geq 1$,
\begin{equation}\label{eq 3 19-08-2017}
[\del^I,\delb_a] = \sum_{1\leq |J|\leq |I|}\sigma_{aJ}^I\del^J
\end{equation}
with $\sigma_{aJ}^I$ a homogeneous function of degree $|J|-|I|-1$.
\end{lemma}
\begin{proof}
This is by induction on $|I|$. For $|I|=1$ we see that holds. We suppose that \eqref{eq 3 19-08-2017} holds for $|I|\leq m$, we consider
$$
\aligned
\,[\del_{\alpha}\del^I,\delb_a] =& \del_{\alpha}\left(\sum_{1\leq |J|\leq |I|}\sigma_{aJ}^I\del^J\right) + [\del_{\alpha},\delb_a]\del^I
\\
=& \sum_{1\leq |J|\leq |I|}\sigma_{aJ}^I\del_{\alpha}\del^J + \sum_{1\leq|J|\leq|I|}\del_{\alpha}\sigma_{aJ}^I\del^J + \sigma_{\alpha a}\del_t\del^I.
\endaligned
$$
We see that $\del_{\alpha}\del^J$ is of order $|J|+1$, $\del^t\del^I$ is of order $|I|+1$; we recall that $\del_{\alpha}\sigma_{aJ}^I$ is homogeneous of order $|J|-|I|-2 = |J| - (|I|+1)-1$ (by the assumption of induction combined with lemma \ref{lem 1 26-06-2017}), and $\sigma_{\alpha a}$ is homogeneous of degree $-1$. Thus we see that the \eqref{eq 3 19-08-2017} is proved in the case $|I| = m+1$.
\end{proof}

Now we calculate:
$$
[L_a,\delb_b] = -\frac{x^b}{t}\delb_a = \Psib_b^0\delb_a.
$$
We denote by
$$
[L_a,\delb_b] = \eta_{b}\delb_a
$$
where $\eta_b$ is a homogeneous function of degree zero. Now we establish the following result:
\begin{lemma}\label{lem 1 21-08-2017}
Let $|I| = n$, then
\begin{equation}\label{eq 2 21-08-2017}
[L^I,\delb_a] = \sum_{|J|<|I|}\eta_{aJ}^{Ib}\delb_bL^J
\end{equation}
where $\eta_{aJ}^{Ib}$ are homogeneous functions of degree zero.
\end{lemma}
\begin{proof}
This is by induction on $|I|$. We see that for $|I|=1$ \eqref{eq 2 21-08-2017} holds. Suppose that \eqref{eq 2 21-08-2017} holds for $|I|\leq m$, now we consider
$$
\aligned
\,[L_aL^I,\delb_b] =& L_a([L^I,\delb_b]) + [L_a,\delb_b]L^I = L_a\left(\sum_{|I'|<|I|}\eta_{bI'}^{Ic}\delb_cL^{I'}\right) + \eta_b\delb_aL^{I}
\\
=& \sum_{|I'|<|I|}L_a\eta_{bI'}^{Ic}\delb_cL^{I'} + \sum_{|I'|<|I|}\eta_{bI'}^{Ic}L_a\delb_cL^{I'} + \eta_b\delb_aL^{I}
\\
=& \sum_{|I'|<|I|}L_a\eta_{bI'}^{Ic}\delb_cL^{I'} + \sum_{|I'|<|I|}\eta_{bI'}^{Ic}\delb_c L_a L^{I'} + \sum_{|I'|<|I|}\eta_{bI'}^{Ic}[L_a,\delb_c]L^{I'} + \eta_b\delb_aL^{I}
\\
=& \sum_{|I'|<|I|}L_a\eta_{bI'}^{Ic}\delb_cL^{I'} + \sum_{|I'|<|I|}\eta_{bI'}^{Ic}\delb_c L_a L^{I'} + \sum_{|I'|<|I|}\eta_{bI'}^{Ic}[L_a,\delb_c]L^{I'} + \eta_b\delb_aL^{I}
\\
=& \sum_{|I'|<|I|}L_a\eta_{bI'}^{Ic}\delb_cL^{I'} + \sum_{|I'|<|I|}\eta_{bI'}^{Ic}\delb_c L_a L^{I'} + \sum_{|I'|<|I|}\eta_{bI'}^{Ic}\eta_a\delb_c L^{I'} + \eta_b\delb_aL^{I}.
\endaligned
$$
We recall that $\eta_{bI'}^{Ic}$ homogeneous of degree zero so  $L^J\eta_{bI'}^{Ic}$ is again homogeneous of degree zero. This concludes \eqref{eq 2 21-08-2017} for the case $|I| = m+1$.
\end{proof}

Now we are ready to establish the following result:
\begin{lemma}\label{lem 2 21-08-2017}
\begin{equation}\label{eq 1 23-08-2017}
[\del^IL^J,\delb_a] = \sum_{|I'|\leq|I|\atop |J'|<|J|}\rhob_{a I'J'}^{IJc}\delb_c\del^{I'}L^{J'} + \sum_{1\leq|I'|\leq |I|}\rho_{a I'}^{IJ}\del^{I'}L^{J}
\end{equation}
\end{lemma}
where $\rhob_{a I'J'}^{IJc}$ are homogeneous functions of degree $(|I'|-|I|)$ and $\rho_{a I'}^{IJ}$ are homogeneous functions of degree $(|I'|-|I|-1)$. Furthermore, in $\Kcal$ for a function $u$ sufficiently regular, we have
\begin{equation}\label{eq 2 23-08-2017}
\left|[\del^IL^J,\delb_a]u\right|\leq C\sum_{{c,|I'|\leq|I|}\atop |J'|<|J|}|\delb_c\del^{I'}L^{J'}u| + Ct^{-1}\sum_{1\leq|I'|\leq |I|}\left|\del^{I'}L^J u\right|.
\end{equation}
\begin{proof}
We remark that
$$
\aligned
\,[\del^IL^J,\delb_a] =& \del^I\left([L^J,\delb_a]\right) + [\del^I,\delb_a]L^J
\\
=&\del^I\left(\sum_{|J'|<|J|}\eta_{aJ'}^{Jc}\delb_cL^{J'}\right) + \sum_{1\leq |I'|\leq |I|}\sigma_{aI'}^I\del^{I'}L^J
\\
=&\sum_{I_1+I_2=I\atop |J'|<|J|}\del^{I_1}\eta_{aJ'}^{Jc}\cdot \del^{I_2}\delb_cL^{J'} + \sum_{1\leq |I'|\leq |I|}\sigma_{aI'}^I\del^{I'}L^J
\\
=&\sum_{I_1+I_2=I\atop |J'|<|J|}\del^{I_1}\eta_{aJ'}^{Jc}\cdot \delb_c\del^{I_2}L^{J'}
+ \sum_{I_1+I_2=I\atop |J'|<|J|}\del^{I_1}\eta_{aJ'}^{Jc}[\del^{I_2},\delb_c]L^{J'}
+ \sum_{1\leq |I'|\leq |I|}\sigma_{aI'}^I\del^{I'}L^J
\\
=&\sum_{I_1+I_2=I\atop |J'|<|J|}\del^{I_1}\eta_{aJ'}^{Jc}\cdot \delb_c\del^{I_2}L^{J'}
+ \sum_{I_1+I_2=I\atop |J'|<|J|}\del^{I_1}\eta_{aJ'}^{Jc}\left(\sum_{1\leq |I_2'|\leq |I_2|}\sigma_{cI_2'}^{I_2}\del^{I_2'}L^{J'}\right)
\\
&+ \sum_{1\leq |I'|\leq |I|}\sigma_{aI'}^I\del^{I'}L^J
\\
=&\sum_{I_1+I_2=I\atop |J'|<|J|}\del^{I_1}\eta_{aJ'}^{Jc}\cdot \delb_c\del^{I_2}L^{J'}
+ \sum_{{I_1+I_2=I,|J'|<|J|}\atop 1\leq |I_2'|\leq|I_2| }\del^{I_1}\eta_{aJ'}^{Jc}\cdot \sigma_{cI_2'}^{I_2}\del^{I_2'}L^{J'}
\\
&+ \sum_{1\leq |I'|\leq|I|}\sigma_{aI'}^I\del^{I'}L^J.
\endaligned
$$
Now we recall that $\del^{I_1}\eta_{aJ'}^{Jc}$ is homogeneous of degree $-|I_1| = |I_2|-|I|$, $\del^{I_1}\eta_{aJ'}^{Jc}\cdot \sigma_{cI_2'}^{I_2}$ is homogeneous of degree $-|I_1| + |I_2'|-|I_2|-1 = |I_2'|-|I|-1$ and $\sigma_{aI'}^I$ is homogeneous of degree $|I'|-|I|-1$. Thus the desired result is established.

\eqref{eq 2 23-08-2017} is direct by \eqref{eq 1 23-08-2017}.
\end{proof}

\subsection{Commutators II}
In this subsection we consider the following quantities:
$$
\del^IL^J(s/t), \quad \del^IL^J s.
$$
And then based on these calculation, we analyse $[\del^I,\delb_s]$.

We first remark the following result:
\begin{equation}\label{eq 5 19-08-2017}
\del_t (s/t) = \frac{r^2}{t^2} s^{-1}, \quad \del_a (s/t) = -\frac{x^a}{t} s^{-1}.
\end{equation}
We denote by
$$
\del_{\alpha}(s/t) = \pi_{\alpha}s^{-1}
$$
with $\pi_{\alpha}$ a homogeneous function of degree zero.
\begin{equation}\label{eq 7 19-08-2017}
\del_t s = \frac{t}{s}, \quad \del_a s = -\frac{x^a}{s} = -\frac{x^a}{t}(t/s).
\end{equation}
We denote by
$$
\del_{\alpha}s = \rho_{\alpha}(s/t)^{-1}
$$
with $\rho_{\alpha}$ a homogeneous function of degree zero.
We also recall
$$
L_a(s/t) = -\frac{x^a}{t}(s/t),\quad L_a s = 0.
$$

We first establish the following relation:
\begin{lemma}\label{lem 3 19-08-2017}
\begin{equation}\label{eq 4 19-08-2017}
L^J(s/t) = \lambda^J(s/t)
\end{equation}
with $\lambda^J$ a homogeneous function of degree zero.
\end{lemma}
\begin{proof}
This is by induction. It is clear that \eqref{eq 4 19-08-2017} holds for $|J|=1$. Then we consider
$$
L_a L^J(s/t) = L_a\left(\lambda^J(s/t)\right) = L_a\lambda^J \cdot (s/t) + \lambda^JL_a(s/t) = \left(L_a\lambda^J  - (x^a/t)\right)(s/t).
$$
We see that by \eqref{lem 1 26-06-2017}, $L_a\lambda^J$ is homogeneous of degree zero. Furthermore $x^a/t$ is also homogeneous of degree zero. Thus $\left(L_a\lambda^J  - (x^a/t)\right)$ is homogeneous of degree zero.
\end{proof}

Then we establish the following result:
%
\begin{lemma} For $|I|\geq 1$,
\begin{equation}\label{eq 6 19-08-2017}
\del^I (s/t) = \sum_{1\leq k\leq |I|}\pi^I_k(s/t)^{-k+1}s^{-k}
\end{equation}
with $\pi$ a sum of finite many homogeneous functions of degree $(k-|I|)$. Furthermore, in $\mathcal{K}$,
\begin{equation}\label{eq 8 19-08-2017}
\left|\del^I(s/t)\right|\leq Cs^{-1}.
\end{equation}
\end{lemma}
\begin{proof}
This is also by induction on $|I|$. For $|I|=1$, we see that it is established by \eqref{eq 5 19-08-2017}. We suppose that \eqref{eq 6 19-08-2017} holds for $|I|\leq m$, and we consider (where we use \eqref{eq 5 19-08-2017} and \eqref{eq 7 19-08-2017})
$$
\aligned
\del_{\alpha}\del^I(s/t) =& \del_{\alpha}\left(\sum_{1\leq k\leq |I|}\pi^I_k(s/t)^{-k+1}s^{-k}\right)
\\
=& \sum_{1\leq k\leq |I|}\del_{\alpha}(\pi^I_k)\cdot (s/t)^{-k+1}s^{-k} + \sum_{1\leq k\leq|I|}\pi^I_k\del_{\alpha}\left((s/t)^{-k+1}\right)\cdot s^{-k}
 \\
 &+ \sum_{1\leq k\leq|I|}\pi^I_k(s/t)^{-k+1}\del_{\alpha}(s^{-k})
\\
=& \sum_{1\leq k\leq |I|}\del_{\alpha}(\pi^I_k)\cdot (s/t)^{-k+1}s^{-k} + \sum_{1\leq k\leq|I|}\pi^I_k(-k+1)(s/t)^{-k}\del_{\alpha}(s/t)\cdot s^{-k}
\\
&+ \sum_{1\leq k\leq|I|}\pi^I_k(s/t)^{-k+1}(-k)s^{-k-1}\del_{\alpha}s
\\
=&\sum_{1\leq k\leq |I|}\del_{\alpha}(\pi^I_k)\cdot (s/t)^{-k+1}s^{-k} +  \sum_{1\leq k\leq|I|}(-k+1)\pi^I_k(s/t)^{-(k+1)-1}\pi_{\alpha}\cdot s^{-(k+1)}
\\
&-\sum_{1\leq k\leq|I|}k\pi^I_k(s/t)^{-(k+1)+1}s^{-(k+1)}\rho_{\alpha}
\\
=& \del_{\alpha}(\pi_1^I)s^{-1}
+ \sum_{2\leq k \leq |I|}\left(\del_{\alpha}(\pi_K^I)+(1-k)\pi_{k-1}^I\pi_{\alpha} + (1-k)\pi_{k-1}^I\rho_{\alpha}\right)(s/t)^{-k+1}s^{-k}
\\
& + \left((1-|I|)\pi_{\alpha} - |I|\rho_{\alpha}\right)\pi_{|I|}^I(s/t)^{-(|I|+1)+1}s^{-(|I|+1)}.
\endaligned
$$
WE check that for each term the coefficients are homogeneous of degree $(k-(|I|+1))$, and this concludes the case where $|I| = m+1$.

For \eqref{eq 8 19-08-2017}, we see that in \eqref{eq 6 19-08-2017},
$$
\left|\pi_k^I(s/t)^{-k+1}s^{-k}\right|\leq Ct^{k-1}s^{-2k+1} = Ct^{k-1}s^{-2k+2}s^{-1} = C(t/s^2)^{k-1} s^{-1}.
$$
We remark that in $\mathcal{K}$, $t/s^2$ is bounded. Then \eqref{eq 8 19-08-2017} is established.
\end{proof}

Now we observe the quantity $\del^Is$ (for $|I|\geq 1$). We see that
$$
\del^I(s) = \del^I(t\cdot (s/t)) = \sum_{I_1+I_2=I}\del^{I_1}t\cdot \del^{I_2}(s/t).
$$
We see that for $|I_1|\geq 2$, $\del^{I_1}t = 0$. Thus we see
$$
\sum_{I_1+I_2=I}\del^{I_1}t\cdot \del^{I_2}(s/t) = t\del^I(s/t) + \sum_{|I_1|=1\atop I_2+I_1=I}\del^{I_1}t\cdot \del^{I_2}(s/t)
$$
where the second term does not exist when $|I|\leq 0$. Then combined with \eqref{eq 8 19-08-2017}, we see that
\begin{equation}\label{eq 1 20-08-2017}
|\del^I s|\leq\left\{
\aligned
&C(t/s),\quad |I|\geq 1
\\
&s,\quad |I|=0.
\endaligned
\right.
\end{equation}
\begin{remark}
In the following application, we see that in $\Kcal$, because $s^2\geq t$, we have $|\del^Is|\leq Cs$. Furthermore, we see that when $|J|\geq 1$,
$$
\del^IL^J s = 0.
$$
\end{remark}

Combine \eqref{eq 4 19-08-2017} and \eqref{eq 6 19-08-2017}, we see that
\begin{lemma}
In $\Kcal$,
\begin{equation}\label{eq 2 20-08-2017}
|\del^IL^J(s/t)|\leq
\left\{
\aligned
&Cs^{-1},\quad |I|\geq 1,
\\
&Cs/t,\quad |I|=0.
\endaligned
\right.
\end{equation}
Let $n\in\mathbb{N}^*$. Then
\begin{equation}\label{eq 4 23-08-2017}
|\del^IL^J \left((s/t)^n\right)|\leq
\left\{
\aligned
&Cs^{-1}(s/t)^{n-1},\quad |I|\geq 1,
\\
&C(s/t)^n,\quad |I|=0.
\endaligned
\right.
\end{equation}
\end{lemma}
\begin{proof}
If $|I|=0$, by \eqref{eq 4 19-08-2017}, \eqref{eq 2 20-08-2017} is established. When $|I|\geq 1$, we denote by $\del^I = \del_{\alpha}\del^{I'}$ with $|I'|\geq 0$. Then
$$
\aligned
\del^IL^J(s/t) =& \del_{\alpha}\del^{I'}L^J(s/t) = \del_{\alpha}\del^{I'}\left(\lambda^J(s/t)\right) = \del_{\alpha}\left(\sum_{I_1'+I_2'=I'}\del^{I_1'}\lambda^J\cdot \del^{I_2'}(s/t)\right)
\\
=& \sum_{I_1'+I_2'=I'}\del_{\alpha}\del^{I_1'}\lambda^J\cdot \del^{I_2'}(s/t) + \sum_{I_1'+I_2'=I'}\del^{I_1'}\lambda^J\cdot \del_{\alpha}\del^{I_2'}(s/t).
\endaligned
$$
Now we recall that
$$
|\del^I(s/t)| \leq \left\{
\aligned
&Cs^{-1},\quad |I|\geq 1,
\\
&s/t,\quad |I|=0.
\endaligned
\right.
$$
and the fact that $\big|\del_{\alpha}\del^{I_1'}\lambda^J\big|\leq Ct^{-1}\leq Cs^{-1}$.
Thus \eqref{eq 2 20-08-2017} is established.

For \eqref{eq 4 22-08-2017}, we see that
$$
\del^IL^J\left((s/t)^n\right) = \sum_{I_1+\cdots I_n=I\atop J_1+\cdots J_n=J}\del^{I_1}L^{J_1}(s/t)\cdots \del^{I_n}L^{J_n}(s/t).
$$
We apply \eqref{eq 2 20-08-2017} on each factor, \eqref{eq 4 23-08-2017} is established.
\end{proof}

Then we make the following estimate:
$$
\del^IL^J(s^2/t) = \del^IL^J(s\cdot (s/t)) = \sum_{I_1+I_2=I\atop J_1+J_2=J}\del^{I_1}L^{J_1}s \cdot \del^{I_2}L^{J_2}(s/t).
$$
Then we see that by when $|J_1|\geq 1$, we see that $\del^{I_1}L^{J_1}s=0$. So we see that
$$
\del^IL^J(s^2/t) = \sum_{I_1+I_2=I\atop J_1+J_2=J}\del^{I_1}s \cdot \del^{I_2}L^{J}(s/t).
$$
Now we see that when $|I_1|\geq 1$, we see that
$$
|\del^{I_1}s \cdot \del^{I_2}L^{J}(s/t)|\leq C.
$$
When $|I_1|=0$, suppose that $|I|\geq 1$,  thus
$$
|s\cdot\del^{I}L^{J}(s/t)| \leq C.
$$
When $|I|=0$, we see that
$$
|sL^J(s/t)\leq Cs^2/t.
$$
Thus we conclude that
\begin{equation}\label{eq 4 22-08-2017}
|\del^IL^J(s^2/t)|\leq
\left\{
\aligned
&C,\quad |I|\geq 1,
\\
&Cs^2/t,\quad |I|=0.
\endaligned
\right.
\end{equation}

Then we also establish the following result:
\begin{lemma}\label{lem 4 31-08-2017}
In $\Kcal$ the following bound holds:
\begin{equation}\label{eq 3 31-08-2017}
\left|\del^I (s^{-n})\right|\leq
\left\{
\aligned
 &Cts^{-n-2},\quad |I|\geq 1,
 \\
 &s^{-n},\quad |I| =0.
\endaligned
\right.
\end{equation}
and
\begin{equation}\label{eq 3 02-09-2017}
\left|\del^IL^J(t/s)^n\right|\leq
\left\{
\aligned
&C(t/s)^{n+1}s^{-1},\quad |I|\geq 1,
\\
& (t/s)^n,\quad |I|=0.
\endaligned
\right.
\end{equation}
\end{lemma}
\begin{proof}
This is by the identity of Fa\`a di Bruno. We denote by
$$
\aligned
f:\RR^+ &\rightarrow \RR^+
\\
x &\rightarrow x^{-n}.
\endaligned
$$
Then we see that for $|I|\geq 1$,
$$
\del^I f(s) = \sum_{1\leq k\leq |I|}\sum_{I_1+I_2+\cdots I_k = I}f^{(k)}(s)\del^{I_1}s\del^{I_2}s\cdots \del^{I_k}s.
$$
We see that in the above expression,
$$
\big|f^{(k)}(s)\big|\leq Cs^{-n-k}, \quad |\del^{I_j}s|\leq C(t/s)
$$
This we see that
$$
\big|f^{(k)}(s)\del^{I_1}s\del^{I_2}s\cdots \del^{I_k}s\big|\leq C s^{-k-n}(t/s)^k\leq Cs^{-n}(t/s^2)^k.
$$
Recall that $k\geq 1$ and in $\Kcal$, $s^2\geq t$, \eqref{eq 3 31-08-2017} is established.


For \eqref{eq 3 02-09-2017}, we see that
$$
\aligned
\del^IL^J\left(t^ns^{-n}\right) =& \sum_{I_1+I_2=I\atop J_1+J_2=J}\del^{I_1}L^{J_1}t^{n}\cdot\del^{I_2}L^{J_2}(s^{-n}).
\endaligned
$$
Then we see that $t^n$ is homogeneous of degree $n$ thus $\left|\del^{I_1}L^{J_1}t^{n}\right|\leq Ct^{n-|I_1|}$. Thus by \eqref{eq 3 31-08-2017}, \eqref{eq 3 02-09-2017} is proved.

\end{proof}

For simplicity of expression, we introduce the following notation:
$$
\Lambda^{IJ}: = \del^IL^J(s/t).
$$
Then by \eqref{eq 6 29-08-2017}, the following estimate is direct:
\begin{equation}\label{eq 5 05-05-2017}
\del^IL^J\Lambda^{I'J'}\leq
\left\{
\aligned
&Cs^{-1}\quad |I|+|I'|\geq 1,
\\
&C(s/t)\quad |I|+|I'|=0.
\endaligned
\right.
\end{equation}

Now we are ready to calculate the commutator $[\del^IL^J,\delb_s]$. We have the following result:
\begin{lemma}\label{lem 2 22-08-2017}
Let $u$ be a function defined in $\Kcal$, sufficiently regular and vanishes near the conical boundary. Then the following estimate holds:
\begin{equation}\label{eq 8 22-08-2017}
\left|[\del^IL^J,\delb_s]u\right| \leq  Cs^{-1}\sum_{|I'|<|I|}|\del_t\del^{I'}L^Ju| + C(s/t)\sum_{{\alpha,|I'|\leq|I|}\atop|J'|<J}|\del_\alpha\del^{I'}L^{J'}u|.
\end{equation}
where $C$ is a constant determined by $I,J$.
\end{lemma}
\begin{proof}
We remark that
$$
\aligned
\,[\del^IL^J,\delb_s] =& [\del^IL^J,(s/t)\del_t] = \sum_{I_1+I_2=I,J_1+J_2=J\atop|I_2|+|J_2|<|I|+|J|}\Lambda^{I_1J_1}\del^{I_2}L^{J_2}\del_t + (s/t)[\del^IL^J,\del_t]
\\
=&\sum_{I_1+I_2=I,J_1+J_2=J\atop|I_2|+|J_2|<|I|+|J|}\Lambda^{I_1J_1}\del_t\del^{I_2}L^{J_2}
+ \sum_{I_1+I_2=I,J_1+J_2=J\atop|I_2|+|J_2|<|I|+|J|}\Lambda^{I_1J_1}\del^{I_2}[L^{J_2},\del_t] + (s/t)\del^I[L^J,\del_t]
\\
=&\sum_{I_1+I_2=I,J_1+J_2=J\atop|I_2|+|J_2|<|I|+|J|}\Lambda^{I_1J_1}\del_t\del^{I_2}L^{J_2}
+ \sum_{I_1+I_2=I\atop J_1+J_2=J}\Lambda^{I_1J_1}\del^{I_2}[L^{J_2},\del_t]
\\
=&\sum_{I_1+I_2=I,J_1+J_2=J\atop|I_2|+|J_2|<|I|+|J|}\Lambda^{I_1J_1}\del_t\del^{I_2}L^{J_2}
 +\sum_{I_1+I_2=I, J_1+J_2=J\atop|J_2'|<|J_2|}\Lambda^{I_1J_1}\theta_{0J_2'}^{J_2\beta}\del_{\beta}\del^{I_2}L^{J_2'}.
\endaligned
$$
We decompose the first term in right-hand-side and see that
$$
\aligned
\,[\del^IL^J,\delb_s]
=&\sum_{I_1+I_2=I\atop|I_2|<|I|}\Lambda^{I_1O}\del_t\del^{I_2}L^{J}
\\
&+\sum_{I_1+I_2=I,J_1+J_2=J\atop|J_2|<|J|}\Lambda^{I_1J_1}\del_t\del^{I_2}L^{J_2}
+\sum_{I_1+I_2=I, J_1+J_2=J\atop|J_2'|<|J_2|}\Lambda^{I_1J_1}\theta_{0J_2'}^{J_2\beta}\del_{\beta}\del^{I_2}L^{J_2'}.
\endaligned
$$
Were $\Lambda^{IO}$ means in $\del^I(s/t)$ where $|J|=0$.

Recall the bound on $\Lambda^{IJ}$, the desired result is direct.
\end{proof}

Finally  we establish the following estimates:
\begin{lemma}\label{lem 6 31-08-2017}
Let $u$ be a function defined in $\Kcal$, sufficiently regular, then
\begin{equation}\label{eq 8 31-08-2017}
\big|[\del^IL^J,\delb_s\delb_s]u\big|\leq C\sum_{\alpha,\beta,|I'|\leq|I|,|J'|\leq|J|\atop |I'|+|J'|<|I|+|J|}|(s/t)^2\del_{\alpha}\del_{\beta}\del^{I'}L^{J'}u|
+ C\sum_{\alpha,|I'|\leq|I|,|J'|\leq|J|\atop |I'|+|J'|<|I|+|J|}|t^{-1}\del_{\alpha}\del^{I'}L^{J'}u|.
\end{equation}
\end{lemma}
\begin{proof}
We remark that
$$
\delb_s\delb_s = (s/t)^2\del_t\del_t + t^{-1}(r/t)^2\del_t.
$$
Then we see that
$$
[\del^IL^J,\delb_s\delb_s]u = [\del^IL^J,(s/t)^2\del_t\del_t]u + [\del^IL^J ,t^{-1}(r/t)^2\del_t]u =: T_1+T_2
$$
We see that
$$
[\del^IL^J,(s/t)^2\del_t\del_t]u = \sum_{{I_1+I_2=I, J_1+J_2=J}\atop |I_2|+|J_2|\leq |I|+|J|-1}\del^{I_1}L^{J_1}(s/t)^2\del^{I_2}L^{J_2}\del_t\del_tu
+ (s/t)^2[\del^IL^J,\del_t\del_t]u.
$$
We see that by \eqref{eq 6 22-08-2017}:
\begin{equation}\label{eq 2 01-09-2017}
\aligned
\,[\del^IL^J,\del_t\del_t]u =& \del_t\left([\del^{I}L^{J},\del_t]u\right) + [\del^{I}L^{J},\del_t]\del_tu
\\
=&\sum_{|J'|<|J|}\theta_{0J'}^{J\beta}\del_t\del_{\beta}\del^{I}L^{J'}u
 + \sum_{|J'|<|J|}\theta_{0J'}^{J\beta}\del_{\beta}\del^{I}L^{J'}\del_tu
 \\
=&2\sum_{|J'|<|J|}\theta_{0J'}^{J\beta}\del_t\del_{\beta}\del^{I}L^{J'}u
+\sum_{|J'|<|J|}\theta_{0J'}^{J\beta}\del_{\beta}\del^{I}\left([L^{J'},\del_t] u\right)
\\
=& 2\sum_{|J'|<|J|}\theta_{0J'}^{J\beta}\del_t\del_{\beta}\del^{I}L^{J'}u
 + \sum_{|J'|<|J|\atop|J''|<|J'|}\theta_{0J'}^{J_2\beta}\theta_{0J''}^{J'\gamma}\del_{\beta}\del_{\gamma}\del^{I}L^{J''}u.
\endaligned
\end{equation}
This leads to
\begin{equation}\label{eq 1 01-09-2017}
[\del^{I}L^{J},\del_t\del_t u]\leq C\sum_{\alpha,\beta\atop|J'|<|J|}|\del_{\alpha}\del_{\beta}\del^{I}L^{J'}u|
\end{equation}
By the above inequality we also see that:
$$
\aligned
|\del^{I_2}L^{J_2}\del_t\del_tu| \leq& |\del_t\del_t\del^{I_2}L^{J_2}u| + |[\del^{I_2}L^{J_2},\del_t\del_t]u|
\\
\leq& C\sum_{{\alpha,\beta}\atop|J_2'|\leq|J_2|}|\del_{\alpha}\del_{\beta}\del^{I_2}L^{J_2'}u|
\endaligned
$$
Thus we see that (recall $|\del^IL^J(s/t)^2|\leq C(s/t)^2$)
$$
|T_1|\leq C\sum_{\alpha,\beta,|I'|\leq |I|,|J'|\leq |J|\atop|I'|+|J'|\leq|I|+|J|-1}|(s/t)^2\del_{\alpha}\del_{\beta}\del^{I'}L^{J'}u|
$$

For the term $T_2$, by \eqref{eq 6 22-08-2017} and the fact that $r^2/t^3$ is homogeneous of degree $-1$:
$$
\aligned
|T_2|\leq& \sum_{I_1+I_2=I,J_1+J_2=J\atop|I_2|+|J_2|\leq |I|+|J|-1}\big|\del^{I_1}L^{J_1}(r^2/t^3)\cdot\del^{I_2}L^{J_2}\del_tu\big| + \big|(r^2/t^3)[\del^IL^J,\del_t]u\big|
\\
\leq& \sum_{\alpha,|I'|\leq|I|,|J'|\leq |J|\atop |I'|+|J'|\leq|I|+|J|-1}|t^{-1}\del_\alpha\del^{I'}L^{J'} u|.
\endaligned
$$

The bounds on $T_1$ and $T_2$ concludes the desired result.

\end{proof}

\begin{lemma}\label{lem 3 03-09-2017}
Let $u$ be a function defined in $\Kcal$ and  sufficiently regular. Then the following estimate holds:
\begin{equation}
\aligned
\big|[\del^IL^J,\delb_s\delb_a]u\big|
\leq& C\sum_{|I''|\leq|I|\atop|J''|\leq|J|}\left(t^{-1}|\delb_s\del^{I''}L^{J''}u| + (s/t^2)|\del^{I''}L^{J''}u|\right)
\\
&+ C(s/t^2)\sum_{a,|I''|\leq|I|\atop|J''|\leq|J|}|\delb_a\del^{I''}L^{J''}u|.
\endaligned
\end{equation}
\end{lemma}
\begin{proof}
$$
[\del^IL^J,\delb_s\delb_a]u = [\del^IL^J,\delb_s]\delb_au + \delb_s\left([\del^IL^J,\delb_a]u\right)
$$
For the first term in right-hand-side of the above equation, we see that by \eqref{eq 8 22-08-2017},
$$
\aligned
\big|[\del^IL^J,\delb_s]\delb_au\big|\leq& Cs^{-1}\sum_{|I'|<|I|}|\del_t\del^{I'}L^J\delb_au| + C(s/t)\sum_{{\alpha,|I'|\leq|I|}\atop|J'|<J}|\del_\alpha\del^{I'}L^{J'}\delb_au|
\\
= &Cs^{-1}\sum_{|I'|<|I|}|\del_t\del^{I'}L^J\left(t^{-1}L_au\right)|
+ C(s/t)\sum_{{\alpha,|I'|\leq|I|}\atop|J'|<J}|\del_\alpha\del^{I'}L^{J'}\left(t^{-1}L_a\right)u|
\\
\leq& Cs^{-1}t^{-1}\sum_{|I'|<|I|\atop |J'|\leq|J|}|\del_t\del^{I'}L^JL_au|
+ C(s/t^2)\sum_{{\alpha,|I'|\leq|I|}\atop|J'|<J}|\del_\alpha\del^{I'}L^{J'}L_au|
\\
\leq& C(s/t^2)\sum_{{\alpha\atop|I'|+|J'|\leq|I|+|J|}}|\del_\alpha\del^{I'}L^{J'}u|
\endaligned
$$

For the term $\delb_s\left([\del^IL^J,\delb_a]u\right)$, we see that by applying \eqref{eq 1 23-08-2017},
\begin{equation}\label{eq 1 05-09-2017}
\aligned
\delb_s\left([\del^IL^J,\delb_a]u\right) =&\sum_{|I'|\leq|I|\atop |J'|<|J|}\delb_s\left(\rhob_{a I'J'}^{IJc}\delb_c\del^{I'}L^{J'}u\right) + \sum_{1\leq|I'|\leq |I|}\delb_s\left(\rho_{a I'}^{IJ}\del^{I'}L^{J}u\right)
\endaligned
\end{equation}
Now we for the first term in the right-hand-side, we see that
$$
\delb_s\left(\rhob_{a I'J'}^{IJc}\delb_c\del^{I'}L^{J'}u\right) = \delb_s\rhob_{a I'J'}^{IJc}\cdot \delb_c\del^{I'}L^{J'}u
+ \rhob_{a I'J'}^{IJc}\delb_s\delb_c\del^{I'}L^{J'}u
$$
We see that
$$
\left|\delb_s\rhob_{a I'J'}^{IJc}\right|\leq C(s/t^2)
$$
thus
\begin{equation}\label{eq 2 05-09-2017}
\left|\delb_s\rhob_{a I'J'}^{IJc}\cdot \delb_c\del^{I'}L^{J'}u\right|\leq C(s/t^2)\big|\delb_c\del^{I'}L^{J'}u \big|,
\end{equation}

Now we consider $\rhob_{a I'J'}^{IJc}\delb_s\delb_c\del^{I'}L^{J'}u$. This is by the following calculation (by \eqref{eq 1 19-08-2017}):
$$
\aligned
\delb_s\delb_c\del^{I'}L^{J'}u =& \delb_s\left(t^{-1}L_c\del^{I'}L^{J'}u\right) = \delb_s\left(t^{-1}\del^{I'}L_cL^{J'}u\right) + \sum_{|I''|=I'}\delb_s\left(t^{-1}\theta_{cI''}^{I'}\del^{I''}L^{J'}u\right)
\endaligned
$$
Thus we see that
$$
\big|\rhob_{a I'J'}^{IJc}\delb_s\delb_c\del^{I'}L^{J'}u\big|
\leq C\sum_{|I''|\leq|I|\atop |J''|\leq|J|}\left((s/t^2)|\del^{I''}L^{J''}u| + t^{-1}|\delb_s\del^{I''}L^{J''}u|\right).
$$

For the second term in \eqref{eq 1 05-09-2017}, we see that
$$
\delb_s\left(\rho_{a I'}^{IJ}\del^{I'}L^{J}u\right) = \delb_s\rho_{a I'}^{IJ}\cdot \del^{I'}L^{J}u + \rho_{a I'}^{IJ}\cdot\delb_s\del^{I'}L^Ju
$$
and recall that $\rho_{a I'}^{IJ}$ is homogeneous of degree $\leq -1$. Thus we see that $|\delb_s\rho_{a I'}^{IJ}|\leq Cs/t^2$. So we see that
$$
\big|\delb_s\left(\rho_{a I'}^{IJ}\del^{I'}L^{J}u\right)\big|\leq C\sum_{|I''|\leq|I|\atop|J'|\leq|J|}\left(\big|t^{-1}\delb_s\del^{I''}L^{J''}u\big|
  + (s/t^2)\big|\del^{I''}L^{J''}u\big|\right).
$$
Thus we see that the desired estimate is established.
\end{proof}

\begin{lemma}\label{lem 1 05-09-2017}
Let $u$ be a function defined in $\Kcal$ sufficiently regular. Then the following estimate holds:
\begin{equation}\label{eq 3 05-09-2017}
\big|[\del^IL^J,\delb_a\delb_b]u\big|\leq C\sum_{c,|I'|\leq|I|\atop|J'|\leq|J|}\left(t^{-1}|\delb_c\del^{I'}L^{J'}u| + t^{-2}|\del^{I'}L^{J'}u|\right).
\end{equation}
\end{lemma}
\begin{proof}
We remark that
\begin{equation}\label{eq 4 05-09-2017}
\aligned
\,[\del^IL^J,\delb_a\delb_b]u =& [\del^IL^J,\delb_a] \delb_bu + \delb_a\left([\del^IL^J,\delb_b]u\right)
\endaligned
\end{equation}

For the first term in right-hand-side of the above equation, we see that by \eqref{eq 1 23-08-2017},
$$
\aligned
\,[\del^IL^J,\delb_a] \delb_bu =& \sum_{|I'|\leq|I|\atop |J'|<|J|}\rhob_{a I'J'}^{IJc}\delb_c\del^{I'}L^{J'}\delb_bu
 + \sum_{1\leq|I'|\leq |I|}\rho_{a I'}^{IJ}\del^{I'}L^{J}\delb_bu
\endaligned
$$
We see that (by homogeneity of $\rhob_{a I'J'}^{IJc}$ and $t^{-1}$)
$$
\aligned
\big|\rhob_{a I'J'}^{IJc}\delb_c\del^{I'}L^{J'}\delb_bu\big| =&\big|\rhob_{a I'J'}^{IJc}\delb_c\del^{I'}L^{J'}\left(t^{-1}L_bu\right)\big|
\\
\leq &C\sum_{c,|I''|\leq|I|\atop|J''|<|J|}\left(t^{-1}|\delb_c\del^{I''}L^{J''}u| + t^{-2}|\del^{I''}L^{J''}u|\right).
\endaligned
$$
By \eqref{eq 2 23-08-2017}
$$
\aligned
\big|\rho_{aI'}^{IJ}\del^{I'}L^{J}\delb_bu\big|\leq&
Ct^{-1}\sum_{{c,|I''|\leq|I'|}\atop |J''|<|J|}|\delb_c\del^{I''}L^{J''}u| + Ct^{-2}\sum_{1\leq|I''|\leq |I'|}\left|\del^{I''}L^Ju\right|
\endaligned
$$

For the second term in right-hand-side of \eqref{eq 4 05-09-2017}, we see that by \eqref{eq 1 23-08-2017}
$$
\aligned
\delb_a\left([\del^IL^J,\delb_b]u\right) =&\sum_{|I'|\leq|I|\atop |J'|<|J|}\delb_a\left(\rhob_{b I'J'}^{IJc}\delb_c\del^{I'}L^{J'}u\right) + \sum_{1\leq|I'|\leq |I|}\delb_a\left(\rho_{b I'}^{IJ}\del^{I'}L^{J}u\right).
\endaligned
$$
By homogeneity, we see that
$$
\aligned
\left|\delb_a\left([\del^IL^J,\delb_b]u\right)\right|
\leq& C\sum_{c,|I'|\leq|I|\atop|J'|<|J|}\left(|\delb_a\delb_c\del^{I'}L^{J'}u| + t^{-1}|\delb_c\del^{I'}L^{J'}u|\right)
\\
&+C\sum _{1\leq|I'|\leq|I|}\left(t^{-1}|\delb_a\del^{I'}L^{J}u| + t^{-2}|\del^{I'}L^{J}u|\right).
\endaligned
$$
Now we see that
$$
\aligned
|\delb_a\delb_c\del^{I'}L^{J'}u| =& \big|\delb_a\left(t^{-1}L_c\del^{I'}L^{J'}u\right)\big| \leq C t^{-2}|L_c\del^{I'}L^{J'}u| + Ct^{-1}|\delb_aL_c\del^{I'}L^Ju|
\\
\leq& C\sum_{|I''|=|I'|\atop |J''|\leq|J'|}\left(t^{-2}|\del^{I''}L^{J''}u| + t^{-1}|\delb_a\del^{I''}L^{J''}u|\right)
\endaligned
$$

The above bounds conclude the desired result.
\end{proof}

\subsection{Estimates based on commutators I}
In this subsection, we will control the following terms
\begin{equation}\label{eq 3 24-08-2017}
\|(s^2/t)\del^IL^J\delb_su\|_{L^2(\Hcal_s)},\quad \|s\del^IL^J\delb_au\|_{L^2(\Hcal_s)},\quad \|(s/t)\del^IL^J u\|_{L^2(\Hcal_s)}
\end{equation}
where $|I|+|J|\leq N$.
\begin{equation}\label{eq 6 05-09-2017}
\|s^2\del^IL^J\delb_s\delb_au\|_{L^2(\Hcal_s)},\quad \|st\del^IL^J\delb_a\delb_b u\|_{L^2(\Hcal_s)}
\end{equation}
where $|I|+|J|\leq N-1$.
We have the following result
\begin{lemma}\label{lem 2 24-08-2017}
Let $u$ be a function defined in $\Kcal$, sufficiently regular. Then the terms in \eqref{eq 3 24-08-2017} and \eqref{eq 6 05-09-2017} are bounded by
$$
\sum_{|I|+|J|\leq N}\Ec(s,\del^IL^Ju)^{1/2}.
$$
\end{lemma}
\begin{proof}
These are by apply the decomposition of commutators. For the first term in \eqref{eq 3 24-08-2017},
$$
\aligned
\big|(s^2/t)\del^IL^J\delb_su\big|\leq& \big|(s^2/t)\delb_s\del^IL^Ju\big| + \big|(s^2/t)[\del^IL^J,\delb_s]u\big|
\\
\leq &\big|(s^2/t)\delb_s\del^IL^Ju\big| + C(s/t)\sum_{|I'|<|I|}|\del_t\del^{I'}L^Ju| + Cs(s/t)^2\sum_{{\alpha,|I'|\leq|I|}\atop|J'|<J}|\del_\alpha\del^{I'}L^{J'}u|
\\
\leq& \big|(s^2/t)\delb_s\del^IL^Ju\big| + C\sum_{|I'|<|I|}|(1+s^2/t)\delb_s\del^{I'}L^Ju|
 \\
 &+ C\sum_{{a,|I'|\leq|I|}\atop|J'|<J}|s(s/t)^2\del_a\del^{I'}L^{J'}u|
\endaligned
$$
thus (remark that for the second term, $s^2/t\geq 1\geq s/t$ in $\Kcal$)
$$
\big\|(s^2/t)\del^IL^J\delb_su\big\|_{L^2(\Hcal_s)}\leq C\sum_{|I|+|J|\leq N}\Ec(s,\del^IL^Ju)^{1/2}.
$$

For the second term we apply \eqref{eq 2 23-08-2017}, we omit the detail. The third is guaranteed by \eqref{eq 6 23-05-2017}.

For the first term in \eqref{eq 6 05-09-2017}, we see that
$$
\|s^2\del^IL^J\delb_s\delb_au\|_{L^2(\Hcal_s)}\leq \|s^2\delb_s\delb_a\del^IL^Ju\|_{L^2(\Hcal_s)} + \|s^2[\del^IL^J,\delb_s\delb_a]u\|_{L^2(\Hcal_s)}
$$
By lemma \ref{lem 3 03-09-2017}, we see that
$$
\|s^2[\del^IL^J,\delb_s\delb_a]u\|_{L^2(\Hcal_s)}\leq \sum_{|I'|+|J'|\leq N}\Ec(s,\del^{I'}L^{J'}u)^{1/2}.
$$
$$
\aligned
\|s^2\delb_s\delb_a\del^IL^Ju\|_{L^2(\Hcal_s)} =& \|s^2\delb_s\left(t^{-1}L_a\del^IL^Ju\right)\|_{L^2(\Hcal_s)}
\\
\leq &\|(s^2/t)\delb_sL_a\del^IL^Ju\|_{L^2(\Hcal_s)} + \|(s/t)^3L_a\del^IL^Ju\|_{L^2(\Hcal_s)}
\\
\leq &\|(s^2/t)\delb_s\del^IL_aL^Ju\|_{L^2(\Hcal_s)} + \|(s/t)^3\del^IL_aL^Ju\|_{L^2(\Hcal_s)}
\\
&+\|(s^2/t)\delb_s\left([L_a,\del^I]L^Ju\right)\|_{L^2(\Hcal_s)} + \|(s/t)^3[L_a,\del^I]L^Ju\|_{L^2(\Hcal_s)}
\endaligned
$$
Then by \eqref{eq 1 19-08-2017}, we see that
$$
\aligned
\|s^2\delb_s\delb_a\del^IL^Ju\|_{L^2(\Hcal_s)}\leq& \sum_{|I'|\leq|I|\atop|J'|\leq|J|+1}\|(s^2/t)\delb_s\del^{I'}L^{J'}u\|_{L^2(\Hcal_s)} + \|(s/t)^3\del^{I'}L^{J'}u\|_{L^2(\Hcal_s)}
\\
\leq&\sum_{|I'|+|J'|\leq N}\Ec(s,\del^{I'}L^{J'}u)^{1/2}
\endaligned
$$
and this proved the bound on the first term of \eqref{eq 6 05-09-2017}.

Now we regard the term $\del^IL^J\delb_a\delb_b u$.
\begin{equation}\label{eq 1 07-09-2017}
\aligned
\|st\del^IL^J\delb_a\delb_bu\|_{L^2(\Hcal_s)} =& \|st\delb_a\delb_b\del^IL^Ju\|_{L^2(\Hcal_s)} + \|st[\del^IL^J,\delb_a\delb_b]u\|_{L^2(\Hcal_s)}.
\endaligned
\end{equation}
For the first term in the right-hand-side of the above equation, we see that
$$
\aligned
\|st\delb_a\delb_b\del^IL^Ju\|_{L^2(\Hcal_s)} =& \|st\delb_a\left(t^{-1}L_b\del^IL^Ju\right)\|_{L^2(\Hcal_s)}
\\
\leq& \|s\delb_aL_b\del^IL^Ju\|_{L^2(\Hcal_s)} + \|(s/t)L_b\del^IL^Ju\|_{L^2(\Hcal_s)}
\\
\leq& \|s\delb_a\del^IL_b L^Ju\|_{L^2(\Hcal_s)} + \|(s/t)\del^I L_b L^Ju\|_{L^2(\Hcal_s)}
\\
&+\|s\delb_a\left([L_b,\del^I]L^Ju\right)\|_{L^2(\Hcal_s)} + \|(s/t)[L_b,\del^I]L^Ju\|_{L^2(\Hcal_s)}.
\endaligned
$$
Then, also by \eqref{eq 1 19-08-2017},
$$
\|st\delb_a\delb_b\del^IL^Ju\|_{L^2(\Hcal_s)} \leq C\sum_{|I'|=|I|\atop |J'|\leq|J|+1}\Ec(s,\del^{I'}L^{J'}u)^{1/2}.
$$
For the second term in right-hand-side of \eqref{eq 1 07-09-2017}, by applying lemma \ref{lem 1 05-09-2017}, we see that it is also bounded by
$$
C\sum_{|I|+|J|\leq N}\Ec(s,\del^IL^Ju)
$$
Thus the desired bound is established.
\end{proof}

We also establish a rough bound on $\del^IL^J\delb_s\delb_s$:
\begin{lemma}\label{lem 2 07-09-2017}
Let $u$ be a function defined in $\Kcal$, sufficiently regular and vanishes near the conical boundary. Then the following bound holds for $|I|+|J|\leq N-1$:
\begin{equation}\label{eq 5 07-09-2017}
\|s\delb_s\delb_s\del^IL^Ju\|_{L^2(\Hcal_s)} +\|s\del^IL^J\delb_s\delb_su\|_{L^2(\Hcal_s)}\leq C\sum_{|I'|+|J'|\leq N}\Ec(s,\del^{I'}L^{J'}u)^{1/2}.
\end{equation}
\end{lemma}
\begin{proof}
For the first term, we recall that
\begin{equation}\label{eq 9-07-09-2017}
\aligned
s\delb_s\delb_s\del^IL^Ju =& s\delb_s\left((s/t)\del_t\del^IL^Ju\right)
=& s\delb_s(s/t)\cdot\del_t\del^IL^Ju + s(s/t)\delb_s\del_t\del^IL^Ju.
\endaligned
\end{equation}
We see that for the first term in right-hand-side, by \eqref{eq 2 20-08-2017},
$$
\|s\delb_s\left((s/t)\del_t\del^IL^Ju\right)\|_{L^2(\Hcal_s)}\leq C\|(s/t)\del_t\del^IL^Ju\|_{L^2(\Hcal_s)}
\leq  C\sum_{|I'|+|J'|\leq N}\Ec(s,\del^{I'}L^{J'}u)^{1/2}.
$$
For the second term in right-hand-side of \eqref{eq 9-07-09-2017},
We remark the following calculation (where we apply \eqref{eq 6 22-08-2017}):
\begin{equation}\label{eq 6 07-09-2017}
s\del^IL^J\delb_s\delb_su = \sum_{I_1+I_2=I\atop J_1+J_2=J}s\del^{I_1}L^{J_1}(s/t)\cdot \del_t\del^{I_2}L^{J_2}\delb_su
+ \sum_{I_1+I_2=I,J_1+J_2=J\atop |J_2'|<|J_2|}\!\!\!\!\!\!\!\!\!\!s\del^{I_1}L^{J_1}(s/t)\theta_{0J_2'}^{J_2\beta}\del_{\beta}\del^{I_1}L^{J_2'}\delb_su
\end{equation}
\end{proof}
Then by applying \eqref{eq 2 20-08-2017} and the bounds on terms in \eqref{eq 3 24-08-2017}, we see that the desired result is proved.

\subsection{Estimates based on commutators II}
In view of the global Sobolev inequality \eqref{ineq 1 sobolev}, to turn the $L^2$ bounds (supplied by the energy) into $L^{\infty}$ bounds, we need to  bound some terms. To do so , we need some preparations. Through out this subsection, we denote by $u$ a sufficiently regular function defined in $\Kcal$, and the following estimates are valid in $\Kcal$.

\begin{lemma}\label{lem 1 29-08-2017}
Let $u$ be a sufficiently regular function defined in $\Kcal$. Then the following estimates hold:
\begin{equation}\label{eq 5 29-08-2017}
\left|\del^{I'}L^{J'}\delb_s\del^IL^Ju\right|\leq C\sum_{|I''|\leq|I|+|I'|\atop |J''|\leq|J|+|J'|}\big|\delb_s\del^{I''}L^{J''}u\big|
+ C\sum_{{a,|I'|\leq|I'|+|I|}\atop{|J''|<|J'|+|J|}} |(s/t)\del_a\del^{I''}L^{J''}u|,
\end{equation}
\begin{equation}\label{eq 7 29-08-2017}
\left|\del^{I'}L^{J'}\delb_a\del^IL^Ju\right|\leq  C\sum_{{c,|I''|\leq|I|+|I'|}\atop{|J''|\leq|J|+|J'|}}\big|\delb_c\del^{I''}L^{J''}u\big|
+ Cs^{-1}\sum_{{1\leq|I''|\leq|I'|+|I|}\atop|J''|\leq|J'|+|J|}\big|(s/t)\del^{I''}L^{J''}u\big|.
\end{equation}
\end{lemma}
\begin{proof}
For the first term, we see that
$$
\aligned
\del^{I'}L^{J'}\delb_s\del^IL^Ju =& \delb_s\del^{I'}L^{J'}\del^IL^Ju + [\del^{I'}L^{J'},\delb_s]\del^IL^Ju
=: T_1 +T_2.
\endaligned
$$
Then we see that by \eqref{eq 6 29-08-2017}
$$
T_1 = \sum_{|I''|\leq|I'|+|I|\atop |J''|\leq|J'|+|J|}\zeta^{I'J'IJ}_{I''J''}\delb_s\del^{I''}L^{J''}u
$$
which leads to
$$
|T_1|\leq C\sum_{|I''|\leq|I|+|I'|\atop |J''|\leq|J|+|J'|}\big|\delb_s\del^{I''}L^{J''}u\big|.
$$
On the other hand,
$$
\aligned
|T_2| \leq& Cs^{-1}\sum_{|I''|<|I'|}\big|\del_t\del^{I''}L^{J'}\del^IL^Ju\big|
 + C(s/t)\sum_{{\alpha,|I''|\leq|I'|}\atop |J''|<|J'|}\big|\del_\alpha\del^{I''}L^{J''}\del^IL^Ju\big|
\\
=:& T_3+T_4 .
\endaligned
$$
We see that for each term of $T_3$, by \eqref{eq 6 29-08-2017} and the fact that in $\Kcal$ $s^2\geq t$,
$$
\aligned
s^{-1}\big|\del_t\del^{I''}L^{J'}\del^IL^Ju\big|
\leq&  C(t/s^2)\sum_{|I'''|\leq|I|\atop |J'''|\leq|J'|}\big|(s/t)\del_t\del^{I''}\del^{I'''}L^{J'''}L^Ju\big|
\\
\leq& C\sum_{|I'''|\leq|I|\atop |J'''|\leq|J'|}\big|\delb_s\del^{I''}\del^{I'''}L^{J'''}L^Ju\big|.
\endaligned
$$
Also, for  $T_4$, when $\alpha = 0$, $\del_{\alpha} = \del_t$, then
$$
\aligned
(s/t)\big|\del_t\del^{I''}L^{J''}\del^IL^Ju\big| \leq& C\sum_{|I'''|\leq |I|\atop|J'''|\leq|J''|}\big|(s/t)\del_t\del^{I''}\del^{I'''}L^{J'''}L^Ju\big|
\\
\leq& C\sum_{|I'''|\leq |I|\atop|J'''|\leq|J''|}\big|\delb_s\del^{I''}\del^{I'''}L^{J'''}L^Ju\big|.
\endaligned
$$
For $\alpha>0$, we denote by $\alpha = a$, then
$$
(s/t)\big|\del_a\del^{I''}L^{J''}\del^IL^Ju\big|\leq C\sum_{{a,|I'''|\leq |I|}\atop|J'''|\leq|J''|}\big|(s/t)\del_a\del^{I''}\del^{I'''}L^{J'''}L^Ju\big|
$$

This leads to
$$
|T_2|\leq C\sum_{|I''|\leq|I|+|I'|\atop |J''|\leq|J|+|J'|}\big|\delb_s\del^{I''}L^{J''}u\big| + C\sum_{{a,|I'''|\leq |I|}\atop|J'''|\leq|J''|}\big|(s/t)\del_a\del^{I''}\del^{I'''}L^{J'''}L^Ju\big|.
$$
The bounds of $T_1$ and $T_2$ leads to \eqref{eq 5 29-08-2017}.

For \eqref{eq 7 29-08-2017}, we see that
$$
\aligned
\del^{I'}L^{J'}\delb_a\del^IL^Ju =& \delb_a\del^{I'}L^{J'}\del^{I}L^Ju + [\del^{I'}L^{J'},\delb_a]\del^IL^Ju =: T_3 + T_4.
\endaligned
$$
We see that by \eqref{eq 6 29-08-2017}
$$
|T_3|\leq C\sum_{|I''|\leq|I|+|I'|\atop|J''|\leq |J|+|J'|}\big|\delb_a\del^{I''}L^{J''}u\big|.
$$
For $T_4$, we apply \eqref{eq 2 23-08-2017}:
$$
\aligned
|T_4|\leq & C\sum_{{c,|I''|\leq|I'|}\atop |J''|\leq|J'|}\big|\delb_c\del^{I''}L^{J''}\del^IL^Ju\big| + Ct^{-1}\sum_{1\leq|I''|\leq|I'|}\big|\del^{I''}L^{J'}\del^IL^Ju\big|.
\endaligned
$$
Also by \eqref{eq 6 29-08-2017}:
$$
\aligned
|T_4|\leq& C\sum_{{c,|I''|\leq|I|+|I'|}\atop{|J''|\leq|J|+|J'|}}\big|\delb_c\del^{I''}L^{J''}u\big|
 + Ct^{-1}\sum_{1\leq |I''|\leq|I'|+|I|}\big|\del^{I''}L^{J''}u\big|
\\
=& C\sum_{{c,|I''|\leq|I|+|I'|}\atop{|J''|\leq|J|+|J'|}}\big|\delb_c\del^{I''}L^{J''}u\big|
+ Cs^{-1}\sum_{{1\leq|I''|\leq|I'|+|I|}\atop|J''|\leq|J'|+|J|}\big|(s/t)\del^{I''}L^{J''}u\big|
\endaligned
$$
Now the bounds on $|T_3|$ and $|T_4|$ leads to \eqref{eq 7 29-08-2017}.
\end{proof}

Then the following bounds are direct:
\begin{lemma}\label{lem 2 31-08-2017}
The following terms:
\begin{equation}\label{eq 2 31-08-2017}
\|(s^2/t)\del^{I'}L^{J'}\delb_s\del^IL^Ju\|_{L^2(\Hcal_s)},\quad \|s\del^{I'}L^{J'}\delb_a\del^IL^Ju\|_{L^2(\Hcal_s)}
\end{equation}
are bounded by
$$
C\sum_{|I''|\leq|I|+|I'|\atop |J''|\leq|J|+|J'|}\Ec(s,\del^{I''}L^{J''}u)^{1/2}.
$$
\end{lemma}

Now we regard the following terms:
\begin{equation}\label{eq 1 24-08-2017}
\|\del^{I'}L^{J'}\left((s^2/t)\delb_s\del^IL^J u\right)\|_{L^2(\Hcal_s)},\quad \|\del^{I'}L^{J'}\left(s\delb_a\del^IL^Ju\right)\|_{L^2(\Hcal_s)}
\end{equation}
where $|I'|+|J'|\leq 2$ and $|I|+|J|\leq N-2$.
\begin{equation}\label{eq 2 24-08-2017}
\|\del^{I'}L^{J'}\left(s^2\delb_s\delb_a \del^IL^J u\right)\|_{L^2(\Hcal_s)},\quad \|\del^{I'}L^{J'}\left(st\delb_a\delb_b\del^IL^Ju\right)\|_{L^2(\Hcal_s)}
\end{equation}
where $|I'|+|J'|\leq 2$ and $|I|+|J|\leq N-3$. We have the following results:
\begin{lemma}\label{lem 1 24-08-2017}
The terms in \eqref{eq 1 24-08-2017} and \eqref{eq 2 24-08-2017} are bounded by
$$
C\sum_{|I'|+|J'|\leq N}\Ec(s,\del^{I'}L^{J'}u)^{1/2}.
$$
\end{lemma}
\begin{proof}
For the first term in \eqref{eq 1 24-08-2017}, we see that by \eqref{eq 4 22-08-2017}, \eqref{eq 5 29-08-2017} and lemma \ref{lem 2 31-08-2017}:
$$
\aligned
&\left\|\del^{I'}L^{J'}\left((s^2/t)\delb_s\del^IL^J u\right)\right\|_{L^2(\Hcal_s)}
\leq \sum_{I_1'+I_2'=I'\atop J_1'+J_2'=J'}\left\|\del^{I_1'}L^{J_1'}(s^2/t)\cdot \del^{I_2'}L^{J_2'}\delb_s\del^IL^Ju\right\|_{L^2(\Hcal_s)}
\\
\leq& C\sum_{I_1'+I_2'=I'\atop J_1'+J_2'=J'}\left\|(s^2/t)\del^{I_2'}L^{J_2'}\delb_s\del^IL^Ju\right\|_{L^2(\Hcal_s)}
\leq C\sum_{|I''|\leq|I'|+|I|\atop |J''|\leq|J'|+|J|}\Ec(s,\del^{I''}L^{J''}u)^{1/2}
\\
=& C\sum_{|I''|+|J''|\leq N}\Ec(s,\del^{I''}L^{J''}u)^{1/2}.
\endaligned
$$
The second term of \eqref{eq 1 24-08-2017} is controlled by \eqref{eq 1 20-08-2017} (combined with the fact that in $\Kcal$, $s\geq t/s$), \eqref{eq 7 29-08-2017} and lemma \ref{lem 2 31-08-2017}:
$$
\aligned
&\|\del^{I'}L^{J'}\left(s\delb_a\del^IL^Ju\right)\|_{L^2(\Hcal_s)}
\\
\leq& C\sum_{I_1'+I_2'=I'\atop J_1'+J_2'=J'}\|\del^{I_1'}L^{J_1'}s\cdot\del^{I_2'}L^{J_2'}\delb_a\del^IL^Ju\|_{L^2(\Hcal_s)}
\leq  C\sum_{I_1'+I_2'=I'\atop J_1'+J_2'=J'}\|s\del^{I_2'}L^{J_2'}\delb_a\del^IL^Ju\|_{L^2(\Hcal_s)}
\\
\leq&  C\sum_{|I''|+|J''|\leq N}\Ec(s,\del^{I''}L^{J''}u)^{1/2}.
\endaligned
$$

The first term in \eqref{eq 2 24-08-2017} is bounded as following. First we remark that
$$
\aligned
\del^{I'}L^{J'}\left(s^2\delb_s\delb_a\del^IL^Ju\right)
=& \del^{I'}L^{J'}\left(s^2\delb_s\left(t^{-1}L_a\del^IL^Ju\right)\right)
\\
=&\del^{I'}L^{J'}\left((s^2/t)\delb_sL_a\del^IL^Ju\right) - \del^{I'}L^{J'}\left((s/t)^3L_a\del^IL^Ju\right)
\\
=& T_1 + T_2.
\endaligned
$$
For $T_1$, we see that by \eqref{eq 4 22-08-2017} and lemma \ref{lem 2 31-08-2017}, we see that
$$
\aligned
\|T_1\|\leq& C\sum_{I_1'+I_2'=I'\atop J_1'+J_2'=J'}\del^{I_1'}L^{J_1'}(s^2/t)\cdot \del^{I_2'}L^{J_2'}\delb_sL_a\del^IL^Ju
\\
\leq& C\sum_{|I_2'|\leq|I'|\atop|J_2'|\leq|J'|}\|(s^2/t)\del^{I_2'}L^{J_2'}\delb_sL_a\del^IL^Ju\|_{L^2(\Hcal_s)}
\\
\leq& C\sum_{|I|+|J|\leq N}\Ec(s,\del^{I''}L^{J''}u)^{1/2}.
\endaligned
$$
Here by
For the term $T_2$, we see that by \eqref{eq 6 29-08-2017} and lemma \ref{lem 2 31-08-2017}:
$$
\aligned
\|T_2\|
\leq& C\sum_{I_1'+I_2'+I_3'+I_4'=I'\atop J_1'+J_2'+J_3'+J_4'=J'}
\big\|\del^{I_1'}L^{J_1'}(s/t)\cdot\del^{I_2'}L^{J_2'}(s/t)\cdot\del^{I_3'}L^{J_3'}(s/t)\cdot\del^{I_4'}L^{J_4'}L_a\del^IL^Ju\big\|_{L^2(\Hcal_s)}
\\
\leq& C\sum_{|I_4'|\leq|I'|\atop |J_4'|\leq|J'|}\|(s/t)\del^{I_4'}L^{J_4'}L_a\del^IL^Ju\|_{L^2(\Hcal_s)}
\leq C\sum_{|I''|\leq|I_4'|+|I|\atop |J''|\leq|J_4''|+|J|+1}\|(s/t)\del^{I_4'}L^{J_4'}L_a\del^IL^Ju\|_{L^2(\Hcal_s)}
\\
\leq& C\sum_{|I''|+|J''|\leq N}\Ec(s,\del^{I''}L^{J''}u)^{1/2}.
\endaligned
$$
The bounds on $T_1$ and $T_2$ give the bounds of the first term of \eqref{eq 2 24-08-2017}.

Now we regard the second term in \eqref{eq 2 24-08-2017}:
$$
\aligned
&\del^{I'}L^{J'}\left(st\delb_a\delb_b\del^IL^Ju\right)
= \del^{I'}L^{J'}\left(st\delb_a\left(t^{-1}L_b\del^IL^Ju\right)\right)
\\
=&\del^{I'}L^{J'}\left(s\delb_aL_b\del^IL^Ju\right) - \del^{I'}L^{J'}\left((sx^a/t^2)L_a\del^IL^Ju\right)
\\
=&\sum_{I_1'+I_2'=I'\atop J_1'+J_2'=J'}\del^{I_1'}L^{J_1'}s\cdot\del^{I_2'}L^{J_2'}\delb_aL_b\del^IL^Ju
-\sum_{I_1'+I_2'+I_3=I'\atop J_1'+J_2'+J_3=J'}\del^{I_1'}L^{J_1'}(s/t)\del^{I_2'}L^{J_2'}(x^a/t)\del^{I_3'}L^{J_3'}L_a\del^IL^Ju
\\
=:& T_3 + T_4.
\endaligned
$$
For $T_3$, we see that by \eqref{eq 1 20-08-2017} together with the fact that in $\Kcal$, $t/s\leq s$  and \ref{lem 2 31-08-2017},
$$
\aligned
\|T_3\|\leq& C\sum_{|I_2'|\leq|I'|\atop|J_2'|\leq|J'|}\|s\del^{I_2'}L^{J_2'}\delb_aL_b\del^IL^Ju\|_{L^2(\Hcal_s)}
\leq C\sum_{|I''|+|J''|\leq N}\Ec(s,\del^{I''}L^{J''}u)^{1/2}.
\endaligned
$$
For $T_4$, we see that by \eqref{eq 2 20-08-2017} and lemma \ref{lem 1 26-06-2017} ($x^a/t$ is homogeneous of degree zero)
$$
\aligned
\|\del^{I_1'}L^{J_1'}(s/t)\del^{I_2'}L^{J_2'}(x^a/t)\del^{I_3'}L^{J_3'}L_a\del^IL^Ju\|_{L^2(\Hcal_s)}
\leq& C\|(s/t)\del^{I_3'}L^{J_3'}L_a\del^IL^Ju\|_{L^2(\Hcal_s)}
\\
\leq& C\sum_{|I''|+|J''|\leq N}\|(s/t)\del^{I''}L^{J''}u\|_{L^2(\Hcal_s)}
\\
\leq& C\sum_{|I''|+|J''|\leq N}\Ec(s,\del^{I''}L^{J''}u)^{1/2}
\endaligned
$$
The bounds on $T_3$ and $T_4$ concludes the second term in \eqref{eq 2 24-08-2017}.
\end{proof}

Now we list out a second group of terms:
\begin{equation}\label{eq 1 22-08-2017}
\|\del^{I'}L^{J'}\left((s^2/t)\del^IL^J\delb_s u\right)\|_{L^2(\Hcal_s)},\quad \|\del^{I'}L^{J'}\left(s\del^IL^J\delb_au\right)\|_{L^2(\Hcal_s)},\quad \|\del^{I'}L^{J'}\left((s/t)\del^IL^Ju\right)\|_{L^2(\Hcal_s)}
\end{equation}
where $|I'|+|J'|\leq 2$, $|I|+|J|\leq N-2$ and
\begin{equation}\label{eq 2 22-08-2017}
\|\del^{I'}L^{J'}\left(s^2\del^IL^J\delb_s\delb_a u\right)\|_{L^2(\Hcal_s)},\quad \|\del^{I'}L^{J'}\left(st\del^IL^J \delb_a\delb_b u\right)\|_{L^2(\Hcal_s)}
\end{equation}
where $|I'|+|J'|\leq 2$, $|I|+|J|\leq N-3$ with $N$ is an integer.  In general we have the following estimates:
\begin{lemma}\label{lem 2 05-09-2017}
In $\Kcal$, the terms in \eqref{eq 1 22-08-2017} and \eqref{eq 2 22-08-2017} are bounded by
$$
C\sum_{|I'|+|J'|\leq N}\Ec(s,\del^{I'}L^{J'}u)^{1/2}.
$$
\end{lemma}
\begin{proof}
These are by lemma \ref{lem 2 24-08-2017} and the corresponding estimates of commutators. For the first term in \eqref{eq 1 22-08-2017}, by \eqref{eq 6 29-08-2017}
$$
\aligned
\del^{I'}L^{J'}\left((s^2/t)\del^IL^J\delb_s u\right)
=& \sum_{I_1'+I_2'=I'\atop J_1'+J_2'=J'}\del^{I_1'}L^{J_1'}(s^2/t)\cdot\del^{I_2'}L^{J_2'}\del^{I_2}L^{J_2}\delb_s u
\\
=& \sum_{I_1'+I_2'=I'\atop J_1'+J_2'=J'}\sum_{|I''|=|I_2'|+|I|\atop|J''|\leq|J_2'|+|J|}
\del^{I_1'}L^{J_1'}(s^2/t)\cdot\theta_{I''J''}^{I_2'J_2'IJ}\del^{I''}L^{J''}\delb_s u.
\endaligned
$$
Then by \eqref{eq 4 22-08-2017} and lemma \ref{lem 2 24-08-2017}, the desired $L^2$ bound is direct.

The second term in \eqref{eq 1 22-08-2017} is bounded similarly by \eqref{eq 6 29-08-2017}, \eqref{eq 1 20-08-2017} and lemma \ref{lem 2 24-08-2017}, we omit the detail.

The third term in \eqref{eq 1 22-08-2017} is by \eqref{eq 6 29-08-2017}, \eqref{eq 2 20-08-2017} and lemma \ref{lem 2 24-08-2017}, we also omit the detail.

The first term in \eqref{eq 2 22-08-2017} is by \eqref{eq 1 20-08-2017} combined with lemma \ref{lem 2 24-08-2017}:
$$
\aligned
\|\del^{I'}L^{J'}\left(s^2\del^IL^J\delb_s\delb_au\right)\|_{L^2(\Hcal_s)}
\leq& \sum_{I_1'+I_2'+I_3'=I'\atop J_1'+J_2'+J_3'=J'}
\|\del^{I_1'}L^{J_1'}s\cdot\del^{I_2'}L^{J_2'}s\cdot\del^{I_3'}L^{J_3'}\delb_s\delb_a\del^IL^Ju\|_{L^2(\Hcal_s)}
\\
\leq& C\sum_{|I''|\leq|I'|\atop|J''|\leq|J'|}\|s^2\del^{I_3'}L^{J_3'}\delb_s\delb_a\del^IL^Ju\|_{L^2(\Hcal_s)}
\endaligned
$$
which is bounded by $C\sum_{|I'|+|J'|\leq N}\Ec(s,\del^{I'}L^{J'}u)^{1/2}$.

The second term in \eqref{eq 2 22-08-2017} is by \eqref{eq 1 20-08-2017} combined with lemma \ref{lem 2 24-08-2017}, we omit the detail.
\end{proof}

\subsection{Decay bounds from global Sobolev inequality and commutators}
In this section by applying lemma \ref{lem 1 24-08-2017} and lemma \ref{lem 2 05-09-2017} combined with \eqref{ineq 1 sobolev}, we will establish a series decay estimates. In general we have the following results:
\begin{proposition}[Decay estimates by energy bounds]\label{prop 1 29-08-2017}
Let $u$ be a function defined in $\Kcal$, sufficiently regular. Then  for $|I|+|J|\leq N-2$, the following terms:
\begin{subequations}
\begin{equation}\label{eq 3 29-08-2017}
\sup_{\Hcal_s}\{st^{1/2}|\del^IL^Ju|\},\quad \sup_{\Hcal_s}\{st^{3/2}\delb_a\del^IL^Ju\},\quad \sup_{\Hcal_s}\{s^2t^{1/2}\delb_s\del^IL^Ju\},
\end{equation}
and
\begin{equation}\label{eq 4 29-08-2017}
\sup_{\Hcal_s}\{st^{3/2}\del^IL^J\delb_au\},\quad \sup_{\Hcal_s}\{s^2t^{1/2}\del^IL^J\delb_su\},
\end{equation}
\end{subequations}
are bounded by
$$
C\sum_{|I|+|J|\leq N}\Ec(s,\del^IL^J u)^{1/2}.
$$

For $|I|+|J|\leq N-3$, the following terms:
\begin{subequations}
\begin{equation}
\sup_{\Hcal_s}\left\{s^2t^{3/2}\delb_s\delb_a\del^IL^J u\right\},\quad \sup_{\Hcal_s}\left\{st^{5/2}\delb_a\delb_b\del^IL^Ju\right\}
\end{equation}
and
\begin{equation}
\sup_{\Hcal_s}\left\{s^2t^{3/2}\del^IL^J\delb_s\delb_a u\right\},\quad \sup_{\Hcal_s}\left\{st^{5/2}\del^IL^J\delb_a\delb_bu\right\}
\end{equation}
\end{subequations}
are bounded by
$$
C\sum_{|I|+|J|\leq N}\Ec(s,\del^IL^J u)^{1/2}.
$$
\end{proposition}
Furthermore, we have the following rough decay:
For $|I|+|J|\leq N-3$
\begin{equation}\label{eq 10 07-09-2017}
\sup_{\Hcal_s}\{st^{3/2}\del^IL^J\delb_s\delb_su\}\leq C\sum_{|I'|+|J'|\leq N}\Ec(s,\del^{I'}L^{J'}u)^{1/2}.
\end{equation}
This is by \eqref{eq 6 07-09-2017} combined with \eqref{eq 2 20-08-2017} and \eqref{eq 4 29-08-2017}.


\section{Estimates on Hessian form}
\subsection{Objective and algebraic preparation}
The purpose of this section is to give better $L^2$ and decay bounds on the following terms:
$$
\del_{\alpha}\del_{\beta}\del^IL^J u,\quad \del^IL^J \del_{\alpha}\del_{\beta}u
$$
We first make the following identities:
$$
\aligned
\del_t\del_t =& (t/s)^2\delb_s\delb_s - s^{-1}(r/s)^2\delb_s,
\endaligned
$$

$$
\aligned
\del_t\del_a =& \del_a\del_t = \left(\delb_a - (x^a/s)\delb_s\right)\del_t = -\frac{tx^a}{s^2}\delb_s\delb_s + s^{-1}L_a\delb_s + \frac{tx^a}{s^3}\delb_s
\endaligned
$$
and
$$
\aligned
\del_a\del_b =& \left(\delb_a - (x^a/s)\delb_s\right)\left(\delb_b - (x^b/s)\delb_s\right)
\\
=& \frac{x^ax^b}{s^2}\delb_s\delb_s + t^{-1}L_a\delb_b - \frac{x^a}{st}L_a\delb_s - \frac{x^a}{st}L_b\delb_s - s^{-1}\delb_s - \frac{x^ax^b}{s^3}\delb_s
\endaligned
$$
These identities show the fact that in the Hessian form, the component $\delb_s\delb_s$ has an essential contribution, because the rest terms as at least $s^{-1}$ as a supplementary decay factor. So we will concentrate on the bounds of $\delb_s\delb_s$. More precisely the above identities lead to the following result:
\begin{lemma}\label{lem 1 01-09-2017}
For $u$ a function defined in $\Kcal$, sufficiently regular, the following estimate holds:
\begin{equation}
\|s^2(s/t)^3\del_{\alpha}\del_{\beta}u\|_{L^2(\Hcal_s)}\leq C\|s^2(s/t)\delb_s\delb_s u\|_{L^2(\Hcal_s)} + C\sum_{|J|\leq 1}\Ec(s,L^J u)^{1/2}.
\end{equation}
\end{lemma}
\begin{proof}
This is by direct calculation. We see that
$$
\aligned
\|s^2(s/t)^3\del_t\del_tu\|_{L^2(\Hcal_s)}\leq& \|s^2(s/t)\delb_s\delb_su\|_{L^2(\Hcal_s)} + \|s(s/t)\delb_s u\|_{L^2(\Hcal_s)}
\\
\leq& \|s^2(s/t)\delb_s\delb_su\|_{L^2(\Hcal_s)} + C\Ec(s,u)^{1/2}.
\endaligned
$$
For $\del_t\del_a u$, we need to apply lemma \ref{lem 2 24-08-2017}:
$$
\aligned
\|s^2(s/t)^3\del_t\del_au\|_{L^2(\Hcal_s)}\leq& \|s^2(s/t)\delb_s\delb_su\|_{L^2(\Hcal_s)} + \|s(s/t)^2L_a\delb_su\|_{L^2(\Hcal_s)} + \|s(s/t)\delb_su\|_{L^2(\Hcal_s)}
\\
\leq& \|s^2(s/t)\delb_s\delb_su\|_{L^2(\Hcal_s)} + C\sum_{|J|\leq 1}\Ec(s,L^Ju)^{1/2}.
\endaligned
$$
The term $\del_a\del_bu$ is also by \eqref{lem 2 24-08-2017}:
$$
\aligned
\|s^2(s/t)^3\del_a\del_bu\|_{L^2(\Hcal_s)}\leq& \|s^2(s/t)\delb_s\delb_su\|_{L^2(\Hcal_s)}
 \\
 &+ \|s(s/t)^4L_a\delb_bu\|_{L^2(\Hcal_s)} + \|s(s/t)^2L_b\delb_s\|_{L^2(\Hcal_s)}
 \\
 &+ \|s(s/t)^3\delb_su\|_{L^2(\Hcal_s)} + \|s(s/t)\delb_su\|_{L^2(\Hcal_s)}
 \\
\leq& \|s^2(s/t)\delb_s\delb_su\|_{L^2(\Hcal_s)} + C\sum_{|J|\leq 1}\Ec(s,L^J u)^{1/2}.
\endaligned
$$
\end{proof}

Then we remark the following identity (see \eqref{eq 1 12-05-2017}):
$$
\aligned
\Box =& \delb_s\delb_s + 2(x^a/s)\delb_s\delb_a - \sum_a\delb_a\delb_a + \frac{3}{s}\delb_s
\\
=& \delb_s\delb_s + (2x^a/s)t^{-1}L_a\delb_s - t^{-1}L_a\delb_a + 3s^{-1}\delb_s
\endaligned
$$
This leads to
\begin{equation}\label{eq 2 02-07-2017}
\delb_s\delb_s u = \Box u + H_1 [u].
\end{equation}
where
$$
H_1[u]:= - (2x^a/s)t^{-1}L_a\delb_s + t^{-1}L_a\delb_a - 3s^{-1}\delb_s
$$

\subsection{$L^2$ bounds}
We remark that from \eqref{eq 2 02-07-2017},
$$
(s^3/t)\delb_s\delb_s \del^IL^J u = (s^3/t)\Box \del^IL^J u + (s^3/t)H_1[\del^IL^J u]
$$
We remark the following property:
\begin{lemma}\label{lem 1 31-08-2017}
Let $u$ be a function defined in $\Kcal$, sufficiently regular. Then the following estimate holds:
\begin{equation}\label{eq 1 31-08-2017}
\|(s^3/t)H_1[\del^IL^J u]\|_{L^2(\Hcal_s)}\leq C\sum_{|I'|+|J'|\leq|I|+|J|+1}\Ec(s,\del^{I'}L^{J'}u)^{1/2}.
\end{equation}
\end{lemma}
\begin{proof}
We see that
$$
\aligned
(s^3/t)\big|H_1[\del^IL^Ju]\big| \leq& |(s^2/t)L_a\delb_s\del^IL^Ju| + |sL_a\delb_a\del^IL^Ju| + 3|(s^2/t)\delb_s\del^IL^Ju|.
\endaligned
$$
Then by lemma \ref{lem 2 31-08-2017}, the result is proved.
\end{proof}

On the other hand,
$$
(s^3/t)\del^IL^J\delb_s\delb_s u = (s^2/t)\del^IL^J\Box u + \del^IL^J(H_1[u])
$$
So we establish the following result:
\begin{lemma}\label{lem 3 31-08-2017}
Let $u$ be a function defined in $\Kcal$, sufficiently regular. Then the following estimate holds:
\begin{equation}\label{eq 1 31-08-2017}
\|(s^3/t)\del^IL^J H_1[u]\|_{L^2(\Hcal_s)}\leq C\sum_{|I'|+|J'|\leq|I|+|J|+1}\Ec(s,\del^{I'}L^{J'}u)^{1/2}.
\end{equation}
\end{lemma}
\begin{proof}[Proof of lemma \ref{lem 3 31-08-2017}]
For the first term in $H_1$, we see that
$$
\aligned
(s^3/t)\del^IL^J\left((x^a/t)s^{-1}L_a\delb_su\right) =& \sum_{I_1+I_2+I_3\atop J_1+J_2+J_3=J}(s^3/t)\del^{I_1}L^{J_1}(x^a/t)\del^{I_2}L^{J_2}(s^{-1})\del^{I_3}L^{J_3}L_a\delb_su.
\endaligned
$$
Recall that \eqref{eq 3 31-08-2017} and the fact that $x^a/t$ is homogeneous of degree zero, we see that (by lemma \ref{lem 2 24-08-2017})
$$
\aligned
\|(s^3/t)\del^IL^J\left((x^a/t)s^{-1}L_a\delb_su\right)\|_{\Hcal_s}\leq& C\sum_{|I_3|\leq||I\atop |J_3|\leq|J|}\|(s^2/t)\del^{I_3}L^{J_3}L_a\delb_su\|_{L^2(\Hcal_s)}
\\
\leq& C\sum_{|I'|\leq|I|\atop|J'|\leq|J|+1}\Ec(s,\del^{I'}L^{J'}u)^{1/2}.
\endaligned
$$

The rest terms are bounded similarly and we omit the detail.
\end{proof}

Finally we conclude by the following estimate:
\begin{proposition}\label{prop 1 31-08-2017}
Let $u$ be a function defined in $\Kcal$, sufficiently regular. Then the following bounds hold:
\begin{equation}\label{eq 4 31-08-2017}
\aligned
\|(s^3/t)\del^IL^J\delb_s\delb_su\|_{L^2(\Hcal_s)} + \|(s^3/t)\delb_s\delb_s\del^IL^Ju\|_{L^2(\Hcal_s)}\leq& C\|(s^3/t)\Box \del^IL^J u\|_{L^2(\Hcal_s)}
\\
&+ C\sum_{|I'|\leq|I|\atop|J'|\leq|J|+1}\Ec(s,\del^{I'}L^{J'}u)^{1/2}.
\endaligned
\end{equation}
\end{proposition}

\subsection{$L^{\infty}$ bounds}
The $L^{\infty}$ bounds are based on the above $L^2$ bounds and the global Sobolev inequality \eqref{ineq 1 sobolev}. We first establish the following results:
\begin{lemma}\label{lem 5 31-08-2017}
Let $u$ be a function defined in $\Kcal$, sufficiently regular. Then the following bounds hold:
\begin{equation}\label{eq 5 31-08-2017}
\aligned
\|\del^{I'}L^{J'}\left((s^3/t)\del^IL^J\delb_s\delb_su\right)\|_{L^2(\Hcal_s)}
\leq& C\sum_{|I''|\leq|I'|+|I|\atop|J''|\leq|J'|+|J|}\|(s^3/t)\Box \del^{I''}L^{J''} u\|_{L^2(\Hcal_s)}
\\
&+ C\sum_{|I''|\leq|I|+|I'|\atop|J''|\leq|J|+|J'|+1}\Ec(s,\del^{I''}L^{J''}u)^{1/2}.
\endaligned
\end{equation}
\begin{equation}\label{eq 7 31-08-2017}
\aligned
\|\del^{I'}L^{J'}\left((s^3/t)\delb_s\delb_s\del^IL^Ju\right)\|_{L^2(\Hcal_s)}\leq& C\sum_{|I''|\leq|I'|+|I|\atop|J''|\leq|J'|+|J|}\|(s^3/t)\Box \del^{I''}L^{J''} u\|_{L^2(\Hcal_s)}
\\
&+ C\sum_{|I''|\leq|I|+|I'|\atop|J''|\leq|J|+|J'|+1}\Ec(s,\del^{I''}L^{J''}u)^{1/2}.
\endaligned
\end{equation}
\end{lemma}
To prove this we first remark the following bound:
\begin{equation}\label{eq 6 31-08-2017}
|\del^IL^J (s^3/t)|\leq C(s^3/t).
\end{equation}
This is checked by applying \eqref{eq 1 20-08-2017} and \eqref{eq 4 22-08-2017}:
$$
|\del^IL^J (s^3/t)| = \del^IL^J(s\cdot s^2/t) \leq  C\sum_{I_1+I_2=I\atop J_1+J_2=J}|\del^{I_1}L^{J_1}s\del^{I_2}L^{J_2}(s^2/t)|
\leq Cs(s^2/t)\leq Cs^3/t.
$$

\begin{proof}[Proof of lemma \ref{lem 5 31-08-2017}]
For \eqref{eq 5 31-08-2017}, we see that by \eqref{eq 6 31-08-2017}:
$$
\aligned
\|\del^{I'}L^{J'}\left((s^3/t)\del^IL^J\delb_s\delb_su\right)\|_{L^2(\Hcal_s)}
\leq& C\sum_{I_1'+I_2'=I'\atop J_1'+J_2'=J'}\|\del^{I_1'}L^{J_1'}(s^3/t)\cdot \del^{I_2'}L^{J_2'}\del^IL^J\delb_s\delb_su\|_{L^2(\Hcal_s)}
\\
\leq& C\sum_{|I_2'|\leq|I'|\atop |J_2'|\leq|J'|}\|(s^3/t)\del^{I_2'}L^{J_2'}\del^IL^J\delb_s\delb_su\|_{L^2(\Hcal_s)}
\\
\leq& C\sum_{|I''|\leq|I|+|I'|\atop |J''|\leq|J|+|J'|}\|(s^3/t)\del^{I''}L^{J''}\delb_s\delb_s u\|_{L^2(\Hcal_s)}
\endaligned
$$
where in the last inequality we have applied lemma \ref{lem 2 29-08-2017}. Then by \eqref{eq 4 31-08-2017}, we see that \eqref{eq 5 31-08-2017} is established.

\eqref{eq 7 31-08-2017} is more complicated. We see that by \eqref{eq 6 31-08-2017},
$$
\aligned
\|\del^{I'}L^{J'}\left((s^3/t)\delb_s\delb_s\del^IL^Ju\right)\|_{L^2(\Hcal_s)}
\leq& C\sum_{|I_2'|\leq |I'|\atop|J_2'|\leq |J'|}\|(s^3/t)\del^{I_2'}L^{J_2'}\delb_s\delb_s\del^IL^Ju\|_{L^2(\Hcal_s)}.
\endaligned
$$
To bound this term we observe that by \eqref{eq 8 31-08-2017}
$$
\aligned
\|(s^3/t)\del^{I_2'}L^{J_2'}\delb_s\delb_s\del^IL^Ju\|_{L^2(\Hcal_s)}
\leq& C\sum_{\alpha,\beta,|I''|\leq|I_2'|,|J''|\leq|J_2'|\atop|I''|+|J''|<|I_2'|+|J_2'|}
\|s^2(s/t)^3\del_{\alpha}\del_{\beta}\del^{I''}L^{J''}\del^IL^Ju\|_{L^2(\Hcal_s)}
\\
&+ C\sum_{\alpha,|I''|\leq|I_2'|,|J''|\leq|J_2'|\atop|I''|+|J''|<|I_2'|+|J_2'|}\|s(s/t)^2\del_{\alpha}\del^{I''}L^{J''}\del^IL^Ju\|_{\Hcal_s}.
\endaligned
$$
Then we apply \eqref{eq 6 29-08-2017} and see that by lemma \ref{lem 1 01-09-2017}
$$
\aligned
\|(s^3/t)\del^{I_2'}L^{J_2'}\delb_s\delb_s\del^IL^Ju\|_{L^2(\Hcal_s)}\leq& C\sum_{\alpha,\beta,|I''|\leq|I|+|I'|\atop|J''|\leq|J'|+|J|}
\|s^2(s/t)^3\del_{\alpha}\del_{\beta}\del^{I''}L^{J''}u\|_{L^2(\Hcal_s)}
\\
&+ C\sum_{\alpha,|I''|\leq|I'|+|I|\atop |J''|\leq|J'|+|J|}\|s(s/t)^2\del_{\alpha}\del^{I''}L^{J''}u\|_{\Hcal_s}
\\
\leq& \sum_{|I''|\leq|I'|+|I|\atop |J''|\leq|J'|+|J|}\|(s^3/t)\delb_s\delb_s\del^{I''}L^{J''}u\|_{\Hcal_s}
\\
&+\sum_{|I''|\leq|I'|+|I|\atop |J''|\leq|J'|+|J|+1} \Ec(s,\del^{I''}L^{J''}u)^{1/2}
\endaligned
$$
Then combined with \eqref{eq 4 31-08-2017}, the desired result is established.
\end{proof}

Now, combined with the global Sobolev's inequality \eqref{ineq 1 sobolev}, we have the following decay estimates:
\begin{proposition}
The $L^{\infty}$ bounds are based on the above $L^2$ bounds and the global Sobolev inequality \eqref{ineq 1 sobolev}. We first establish the following results:
\begin{equation}\label{eq 3 01-09-2017}
\aligned
\sup_{\Hcal_s}\{s^3t^{1/2}|\del^IL^J\delb_s\delb_su|\} + \sup_{\Hcal_s}\{s^3t^{1/2}|\delb_s\delb_s\del^IL^Ju|\} \leq& C\sum_{|I'|+|J'|\leq|I|+|J|+2}\|(s^3/t)\Box \del^{I'}L^{J'} u\|_{L^2(\Hcal_s)}
\\
&+ C\sum_{|I'|+|J'|\leq |I|+|J|+3}\Ec(s,\del^{I'}L^{J'}u)^{1/2}.
\endaligned
\end{equation}
\end{proposition}

\section{Null condition in hyperbolic frame}
\subsection{Objective and basic calculations}
The objective of this section is to give a first analysis on the following terms:
$$
\del^IL^J\left(Q^{\alpha\beta\gamma}\del_{\gamma}u\del_{\alpha}\del_{\beta}u\right),\quad [\del^IL^J,Q^{\alpha\beta\gamma}\del_{\gamma}u\del_{\alpha}\del_{\beta}]u.
$$
which will play essential role in the following sections. In this section we always suppose that $u$ is a function defined in $\Kcal$, sufficiently regular and vanishes near the conical boundary $\del\Kcal = \{t = |x|+1\}$.

The first two subsections are preparations for the last one. Here we analyse the property of the quantity $r/t$ in the region $\{t/2<|x|<t\}$. We establish the following bounds:
\begin{lemma}\label{lem 3 02-09-2017}
In the region $\Kcal\cap \{t/2<|x|<t\}$, the following bounds hold with a constant $C$ determined by $I,J$:
\begin{equation}\label{eq 5 02-09-2017}
\left|\del^IL^J (r/t)\right|\leq Ct^{-|I|},\quad \left|\del^IL^J\left(t^{1/2}(t+r)^{-1/2}\right)\right|\leq Ct^{-|I|}
\end{equation}
\begin{equation}\label{eq 5' 02-09-2017}
\del^IL^J\left((t-r)^{1/2}(t+r)^{-1/2}\right)\leq
\left\{
\aligned
&Cs^{-1},\quad |I|\geq 1,
\\
&Cs/t,\quad |I|=0.
\endaligned
\right.
\end{equation}
\end{lemma}
\begin{proof}
We recall that $r^2/t^2$ is homogeneous of degree zero. For the convenience of expression, we de note by
$$
\aligned
f: (1/4,+\infty) &\rightarrow \RR
\\
x &\rightarrow x^{1/2}
\endaligned
$$
and $v := (r/t)^2$. Thus we see that $r/t = f(v)$.
Then (let $|I|+|J|=N\geq 1$)
$$
\aligned
\del^IL^J (r/t) = \del^IL^J\left(f(v)\right) = \sum_{1\leq k\leq N}\sum_{I_1+I_2+\cdots+I_k = I\atop J_1+J_2+\cdots+J_k=J}f^{(n)}(v)\cdot \del^{I_1}L^{J_1}v\cdots\del^{I_k}L^{J_k}v
\endaligned
$$
Thus we see that because $v$ is homogeneous of degree zero, and the fact that $f^{(n)}$ is bounded on $(1/4,+\infty)$ by a constant $C$ (determined by $n\geq 1$). Then we see that for $1/2<r/t<1$,  $\del^IL^J (r/t)$ is bounded by $Ct^{-|I|}$. For $N=0$, we see that $r/t<1$. Thus the first term is correctly bounded.

For the second term in \eqref{eq 5 02-09-2017}, we see that
$$
\left(\frac{t}{t+r}\right)^{1/2} = (1+r/t)^{-1/2} = f(1+r/t).
$$
Then we see that
$$
\del^IL^J\left(f(1+r/t)\right) = \sum_{1\leq k\leq N}\sum_{I_1+I_2+\cdots+I_k = I\atop J_1+J_2+\cdots+J_k=J}f^{(n)}(1+(r/t))\cdot\del^{I_1}L^{J_1}(1+(r/t))\cdots\del^{I_k}L^{J_k}(1+(r/t)).
$$
We see that $3/2<1+r/t<2$, and by the bounds on $r/t$, we see that the second bound is established.

For the third term, we observe that
$$
\del^IL^J\left((t-r)^{1/2}(t+r)^{-1/2}\right) = \del^IL^J\left((s/t)\frac{t}{t+r}\right) = \sum_{I_1+I_2=I\atop J_1+J_2=J}\del^{I_1}L^{J_1}(s/t)\cdot \del^{I_2}L^{J_2}\frac{t}{t+r}.
$$
We recall \eqref{eq 2 20-08-2017} and for the second factor, we see that
$$
\del^IL^J\left(\frac{t}{t+r}\right) = \del^IL^J\left((1+(r/t))^{-1}\right) = \sum_{I_1+I_2=I\atop J_1+J_2=J}\del^{I_1}L^{J_1}\left((1+(r/t))^{-1/2}\right)\del^{I_2}L^{J_2}\left(1+(r/t)^{-1/2}\right).
$$
Then by the bound of the second term in \eqref{eq 5 02-09-2017}, the desired bound is established.
\end{proof}
\begin{corollary}
By lemma \ref{lem 3 02-09-2017}, for a integer $k$, in the region $1/2\leq r/t\leq 1$,
\begin{equation}\label{eq 6 02-09-2017}
\bigg|\del^IL^J\left(\left(\frac{t-r}{t+r}\right)^k\right)\bigg|\leq
\left\{
\aligned
&C(s/t)^{k-1}s^{-1},\quad |I|\geq 1.
\\
&C(s/t)^k ,\quad |I|=0
\endaligned
\right.
\end{equation}
\begin{equation}\label{eq 6' 02-09-2017}
\bigg|\del^IL^J\left(\left(\frac{t}{t+r}\right)^k\right)\bigg|\leq Ct^{-|I|}.
\end{equation}
\end{corollary}

\subsection{Estimates on null forms}
Recall the transition relation between $\Tb^{\alpha\beta}, \Qb^{\alpha\beta\gamma}$ and $T^{\alpha\beta},Q^{\alpha\beta\gamma}$, we see that the following terms are homogeneous of degree zero:
\begin{equation}\label{eq 11 01-09-2017}
\Tb^{ab}, \quad (s/t)\Tb^{a0},\quad (s/t)\Tb^{0a},\quad (s/t)^2\Tb^{00}
\end{equation}
and
\begin{equation}\label{eq 2 02-09-2017}
\aligned
&\Qb^{abc},\quad (s/t)\Qb^{0bc},\quad (s/t)\Qb^{a0c},\quad (s/t)\Qb^{ab0},
\\
&(s/t)^2\Qb^{a00},\quad (s/t)^2\Qb^{0b0},\quad (s/t)^2\Qb^{00c},\quad (s/t)^3\Qb^{000}.
\endaligned
\end{equation}
Based on the above observation, we have
\begin{lemma}\label{lem 2 02-09-2017}
In $\Kcal$, the following quantities are bounded by a constant $C$ which is determined by $I,J$:
$$
\aligned
&\del^IL^J\Qb^{abc},\quad \del^IL^J\Tb^{ab}
\\
&(s/t)\del^IL^J\Qb^{0bc},\quad (s/t)\del^IL^J\Qb^{a0c},\quad (s/t)\del^IL^J\Qb^{ab0},\quad (s/t)\del^IL^J\Tb^{0b},\quad (s/t)\del^IL^J\Tb^{a0},
\\
&(s/t)^2\del^IL^J\Qb^{00c},\quad (s/t)^2\del^IL^J\Qb^{0b0},\quad (s/t)^2\del^IL^J\Qb^{a00},\quad (s/t)^2\del^IL^J\Tb^{00},
\\
&(s/t)^3\del^IL^J\Qb^{000},
\endaligned
$$
and
$$
\aligned
&t\del_{\alpha}\del^IL^J\Qb^{abc},\quad t\del_{\alpha}\del^IL^J\Tb^{ab}
\\
&s(s/t)^2\del_{\alpha}\del^IL^J\Qb^{0bc},\quad s(s/t)^2\del_{\alpha}\del^IL^J\Qb^{a0c},\quad s(s/t)^2\del_{\alpha}\del^IL^J\Qb^{ab0},
\\
&s(s/t)^2\del_{\alpha}\del^IL^J\Tb^{0b},\quad s(s/t)^2\del_{\alpha}\del^IL^J\Tb^{a0},
\\
&s(s/t)^3\del_{\alpha}\del^IL^J\Qb^{00c},\quad s(s/t)^3\del_{\alpha}\del^IL^J\Qb^{0b0},\quad s(s/t)^3\del_{\alpha}\del^IL^J\Qb^{a00},\quad s(s/t)^3\del_{\alpha}\del^IL^J\Tb^{00},
\\
&s(s/t)^4\del_{\alpha}\del^IL^J\Qb^{000},
\endaligned
$$
\end{lemma}
\begin{proof}
This is by applying \eqref{eq 3 02-09-2017} and the fact that the terms in \eqref{eq 11 01-09-2017} and \eqref{eq 2 02-09-2017} are homogeneous of degree zero. We remark the following calculation: let $f$ be a homogeneous function of degree zero. We see that
$$
\del^IL^J\left((t/s)^nf\right) = \sum_{I_1+I_2=I\atop J_1+J_2=J}\del^{I_1}L^{J_1}\left((t/s)^n\right)\cdot \del^{I_2}L^{J_2}f
$$
and by applying \eqref{eq 3 02-09-2017} on the fist factor, we see that
$$
\left|\del^IL^J\left((t/s)^nf\right)\right|\leq
\left\{
\aligned
&C(t/s)^{n+1}s^{-1},\quad |I|\geq 1,
\\
&C(t/s)^n,\quad |I| = 0.
\endaligned
\right.
$$

Then we observe that for $\Tb^{\alpha\beta}$ or $\Qb^{\alpha\beta\gamma}$, the expression
$$
(t/s)^m \Tb^{\alpha\beta},\quad (t/s)^n\Qb^{\alpha\beta\gamma}
$$
are homogeneous of degree zero where $m,n$ are the number of zero in $\alpha,\beta$ or $\alpha,\beta,\gamma$ respectively. This concludes the desired result.

%
\end{proof}

Remark the relation
$$
\delb_a = t^{-1}L_a,\quad \delb_s = (s/t)\del_t
$$
the following bounds are direct:
\begin{lemma}\label{lem 1 16-09-2017}
In $\Kcal$, the following terms are bounded by $C$:
\begin{equation}\label{eq 4 16-09-2017}
\aligned
&t(t/s)\delb_s\Qb^{abc},\quad t(t/s)\delb_s\Tb^{ab},
\\
&s(s/t)\delb_s\Qb^{0bc},\quad s(s/t)\delb_s\Qb^{a0c},\quad s(s/t)\delb_s\Qb^{ab0},\quad s(s/t)\delb_s\Tb^{0b},\quad s(s/t)\delb_s\Tb^{a0},
\\
&s(s/t)^2\delb_s\Qb^{00c},\quad s(s/t)^2\delb_s\Qb^{0b0},\quad s(s/t)^2\delb_s\Qb^{a00},\quad s(s/t)^2\delb_s\Tb^{00},
\\
&s(s/t)^3\delb_s\Qb^{000}.
\endaligned
\end{equation}
and
\begin{equation}\label{eq 5 17-09-2017}
\aligned
&t\delb_a\Qb^{abc},\quad t\delb_a\Tb^{ab},
\\
&s\delb_a\Qb^{0bc},\quad s\delb_a\Qb^{a0c},\quad s\delb_a\Qb^{ab0},\quad s\delb_a\Tb^{0b},\quad s\delb_a\Tb^{a0},
\\
&s(s/t)\delb_a\Qb^{00c},\quad s(s/t)\delb_a\Qb^{0b0},\quad s(s/t)\delb_a\Qb^{a00},\quad s(s/t)\delb_a\Tb^{00},
\\
&s(s/t)^2\delb_a\Qb^{000}.
\endaligned
\end{equation}
\end{lemma}
%
%
%

Now we introduce the following notion of the null form. Let $T$ be a quadratic form defined in $\Kcal$ with constant coefficient (with respect to the canonical frame). We call $T$ a {\bf null quadratic form}, if for any $\xi\in\RR^4$ satisfying
\begin{equation}\label{eq 8 01-09-2017}
\xi_0^2 - \sum_{a=1}^3\xi_a^2 = 0
\end{equation}
the following equation holds:
\begin{equation}\label{eq 7 01-09-2017}
T^{\alpha\beta}\xi_{\alpha}\xi_{\beta} = 0.
\end{equation}
We can also define the null condition for a cubic form: let $Q$ be a constant cubic form defined in $\Kcal$ and for any $\xi$ satisfying \eqref{eq 8 01-09-2017}, the following condition holds:
\begin{equation}\label{eq 9 01-09-2017}
Q^{\alpha\beta\gamma}\xi_{\alpha}\xi_{\beta}\xi_{\gamma} = 0.
\end{equation}

Then we establish the following important result:
\begin{proposition}[Null condition in hyperbolic frame]\label{prop 1 02-09-2017}
Let $T$ and $Q$ be bull quadratic and cubic form respectively. Then the following bounds hold:
\begin{equation}\label{eq 10 01-09-2017}
\big|\del^IL^J\Tb^{00}\big|\leq C,\quad \big|\del^IL^J\Qb^{000}\big|\leq C(t/s)
\end{equation}
and
\begin{equation}\label{eq 10' 01-09-2017}
\big|\del_{\alpha}\del^IL^J\Tb^{00}\big|\leq Cs^{-1},\quad \big|\del_{\alpha}\del^IL^J\Qb^{000}\big|\leq C(t/s)^2s^{-1}
\end{equation}
Furthermore, the following estimates hold:
\begin{equation}\label{eq 6 17-09-2017}
|\delb_\alpha \Tb^{00}|\leq Ct^{-1},\quad |\delb_s\Qb^{000}|\leq C(t/s)s^{-1},\quad |\delb_a\Qb^{000}|\leq Cs^{-1}.
\end{equation}
\end{proposition}
\begin{proof}[Proof of proposition \ref{prop 1 02-09-2017}]
We observe that in the region $\{(t,x)\in\RR^4||x|\leq t/2\}$, this is a direct result of the fact that $(s/t)^3\Qb^{000}$ and $(s/t)^2\Tb^{00}$ are homogeneous of degree zero. To see this, we denote by
$$
f := (s/t)^3\Qb^{000},\quad g := (s/t)^2\Tb^{00}.
$$
Then we see that
$$
\del^IL^J \Qb^{000} = \del^IL^J \left((t/s)^3f\right) = \sum_{I_1+I_2=I\atop J_1+J_2=J}\del^{I_1}L^{J_1}(t/s)^3\cdot\del^{I_2}L^{J_2}f
$$
then recalling that $|\del^IL^Jf|\leq C$, and by \eqref{eq 3 02-09-2017}, we see that $|\del^IL^J \Qb^{000}|\leq C(t/s)^3$. Remark that in the region $\{(t,x)\in\RR^4||x\leq t/2|\}$, $t/s\leq 4/3$. Thus the desired result is established. For $\Tb^{00}$ the argument is similar and we omit the detail.

Then we discuss the region $\Kcal\cap\{t/2<|x|<t\}$. Let
$$
\zeta_{\alpha} := \Psib_\alpha^0, \quad \xi = (r/t,x^1/t, x^2/t, x^3/t) .
$$
We see that $\xi$ satisfies \eqref{eq 8 01-09-2017}. Furthermore,
$$
\nu:= \zeta - (t/s)\xi = ((t-r)/s,0,0,0).
$$

Now we see that
$$
\aligned
\Tb^{00} =& T^{\alpha\beta}\Psib_{\alpha}^0\Psib_{\beta}^0 = T^{\alpha\beta}(\nu_{\alpha} + (t/s)\xi_{\alpha})(\nu_{\beta} + (t/s)\xi_{\beta})
\\
=& T^{\alpha\beta}\nu_{\alpha}\nu_{\beta} + (t/s)T^{\alpha\beta}\nu_{\alpha}\xi_{\beta} + (t/s)T^{\alpha\beta}\xi_{\alpha}\nu_{\beta} + (t/s)^2T^{\alpha\beta}\xi_{\alpha}\xi_{\beta}
\\
=& T^{\alpha\beta}\nu_{\alpha}\nu_{\beta} + (t/s)T^{\alpha\beta}\nu_{\alpha}\xi_{\beta} + (t/s)T^{\alpha\beta}\xi_{\alpha}\nu_{\beta}
\\
=& \frac{t-r}{t+r}T^{00} + \frac{t}{t+r}T^{0\beta}\xi_{\beta} + \frac{t}{t+r}T^{\alpha 0}\xi_{\alpha}.
\endaligned
$$
where we have applied the null condition $T^{\alpha\beta}\xi_{\alpha}\nu_{\beta}=0$. Recall that $\xi_{\alpha}$ are homogeneous of degree zero, combined with \eqref{eq 6 02-09-2017},
$$
\big|\del^IL^J\Tb^{00}\big|\leq
\left\{
\aligned
&Cs^{-1},\quad |I|\geq 1,
\\
&C,\quad |I| = 0.
\endaligned
\right.
$$

Then we regard the cubic form. We see that similar to the quadratic case:
$$
\aligned
\Qb^{000}  =& Q^{\alpha\beta\gamma}\nu_{\alpha}\nu_{\beta}\nu_{\gamma}+ (t/s)Q^{\alpha\beta\gamma}\left(\nu_{\alpha}\nu_{\beta}\xi_{\gamma} + \nu_{\alpha}\xi_{\beta}\nu_{\gamma} + \xi_{\alpha}\nu_{\beta}\nu_{\gamma}\right)
\\
&+ (t/s)^2Q^{\alpha\beta\gamma}\left(\nu_{\alpha}\xi_{\beta}\xi_{\gamma} + \xi_{\alpha}\nu_{\beta}\xi_{\gamma} + \xi_{\alpha}\xi_{\beta}\nu_{\gamma}\right)
+(t/s)^3Q^{\alpha\beta\gamma}\xi_{\alpha}\xi_{\beta}\xi_{\gamma}
\\
=& Q^{\alpha\beta\gamma}\nu_{\alpha}\nu_{\beta}\nu_{\gamma}+ (t/s)Q^{\alpha\beta\gamma}\left(\nu_{\alpha}\nu_{\beta}\xi_{\gamma} + \nu_{\alpha}\xi_{\beta}\nu_{\gamma} + \xi_{\alpha}\nu_{\beta}\nu_{\gamma}\right)
\\
&+ (t/s)^2Q^{\alpha\beta\gamma}\left(\nu_{\alpha}\xi_{\beta}\xi_{\gamma} + \xi_{\alpha}\nu_{\beta}\xi_{\gamma} + \xi_{\alpha}\xi_{\beta}\nu_{\gamma}\right)
\\
=&(t/s)Q^{\alpha\beta\gamma}\left(\nu_{\alpha}\nu_{\beta}\xi_{\gamma} + \nu_{\alpha}\xi_{\beta}\nu_{\gamma} + \xi_{\alpha}\nu_{\beta}\nu_{\gamma}\right)
\\
&+ (t/s)^2Q^{\alpha\beta\gamma}\left(\nu_{\alpha}\xi_{\beta}\xi_{\gamma} + \xi_{\alpha}\nu_{\beta}\xi_{\gamma} + \xi_{\alpha}\xi_{\beta}\nu_{\gamma}\right)
\endaligned
$$
where we have applied the null condition. We see that
$$
(t/s)Q^{\alpha\beta\gamma}\left(\nu_{\alpha}\nu_{\beta}\xi_{\gamma} + \nu_{\alpha}\xi_{\beta}\nu_{\gamma} + \xi_{\alpha}\nu_{\beta}\nu_{\gamma}\right) = \frac{t(t-r)}{s^2}f = \frac{t}{t+r}f
$$
where $f$ is a homogeneous function of degree zero. Also,
$$
(t/s)^2Q^{\alpha\beta\gamma}\left(\nu_{\alpha}\xi_{\beta}\xi_{\gamma} + \xi_{\alpha}\nu_{\beta}\xi_{\gamma} + \xi_{\alpha}\xi_{\beta}\nu_{\gamma}\right)
=\frac{t^2}{(t+r)^2}\left(\frac{t+r}{t-r}\right)^{1/2}f.
$$
Then also by \eqref{eq 6 02-09-2017}, the bound \eqref{eq 10 01-09-2017} is established.

For \eqref{eq 6 17-09-2017}, it is by direct calculation and the following relation in $\Kcal\cap \{r\geq t/2\}$:
\begin{equation}\label{eq 7 17-09-2017}
\left|\delb_s\left(\frac{t-r}{t+r}\right)\right|\leq C(s/t)t^{-1},\quad \left|\delb_a\left(\frac{t-r}{t+t}\right)\right|\leq C(s/t)^2t^{-1}.
\end{equation}
and
\begin{equation}\label{eq 7 17-09-2017}
\left|\del_\alpha\left(\frac{r}{t+r}\right)\right|\leq Ct^{-1}.
\end{equation}
\end{proof}

\subsection{Analysis on null quadratic form}

We first remark the following null decomposition for $u$ defined in $\Kcal$, sufficiently regular:
\begin{equation}\label{eq 5 03-09-2017}
\aligned
Q^{\alpha\beta\gamma}\del_{\gamma}u\del_{\alpha}\del_{\beta}u =& \Qb^{\alpha\beta\gamma}\delb_{\gamma}u\delb_{\alpha}\delb_{\beta}u + Q^{\alpha\beta\gamma}\del_{\alpha}\Psib_{\beta}^{\beta'}\del_{\gamma}u\delb_{\beta'}u
\\
=& \Qb^{000}\delb_su\delb_s\delb_su
\\
&+ \Qb^{00c}\delb_cu\delb_s\delb_su + \Qb^{0b0}\delb_su\delb_s\delb_bu + \Qb^{a00}\delb_su\delb_a\delb_su
\\
&+ \Qb^{0bc}\delb_cu\delb_s\delb_bu + \Qb^{a0c}\delb_cu\delb_a\delb_su + \Qb^{ab0}\delb_su\delb_a\delb_bu + \Qb^{abc}\delb_cu\delb_a\delb_bu
\\
&+ Q^{\alpha\beta\gamma}\del_{\alpha}\Psib_{\beta}^{\beta'}\del_{\gamma}u\delb_{\beta'}u
\\
=:& \Qb^{000}\delb_su\delb_s\delb_su + \Qb_1^{\alpha\beta}(\delb u)\delb_{\alpha}\delb_{\beta}u + Q^{\alpha\beta\gamma}\del_{\alpha}\Psib_{\beta}^{\beta'}\del_{\gamma}u\delb_{\beta'}u.
\endaligned
\end{equation}
For the last term in the right-hand-side of the above equation, we recall \eqref{eq 2 14-08-2017} and \eqref{eq 3 14-08-2017}:
\begin{equation}\label{eq 3 08-09-2017}
\aligned
&-s\del_t\Psib_0^0 = \Psib_0^0\Psib_0^0-1,\quad -s\del_t\Psib_a^0 = \Psib_0^0\Psib_a^0,
\\
&-s\del_a\Psib_0^0 = \Psib_0^0\Psib_a^0,\quad -s\del_a\Psib_b^0 = \delta_{ab} + \Psib_a^0\Psib_b^0
\endaligned
\end{equation}
The rest component of $\del_{\alpha}\Psib_{\beta}^{\gamma}$ are zero. Thus we see that
$$
\aligned
Q^{\alpha\beta\gamma}\del_{\alpha}\Psib_{\beta}^{\beta'}\del_{\gamma}u\delb_{\beta'}u
=&Q^{\alpha\beta\gamma}\del_{\alpha}\Psib_{\beta}^{0}\cdot\Psib_{\gamma}^{\gamma'} \delb_{\gamma'}u\delb_su
\\
=&-s^{-1}Q^{00\gamma}(\Psib_0^0\Psib_0^0-1)\Psib_{\gamma}^{\gamma'}\delb_{\gamma'}u\delb_su
- s^{-1}Q^{a0\gamma}\Psib_a^0\Psib_0^0\Psib_{\gamma}^{\gamma'}\delb_{\gamma'}u\delb_su
\\
&-s^{-1}Q^{0a\gamma}\Psib_0^0\Psib_a^0\Psib_{\gamma}^{\gamma'}\delb_{\gamma'}u\delb_su
-s^{-1}Q^{ab\gamma}\left(\delta_{ab} + \Psib_a^0\Psib_b^0\right)\Psib_{\gamma}^{\gamma'}\delb_{\gamma'}u\delb_su
\\
=&-s^{-1}\Qb^{00\gamma}\delb_{\gamma}u\delb_su + s^{-1}Q^{00\gamma}\del_{\gamma}u\delb_su - s^{-1}\sum_aQ^{aa\gamma}\del_{\gamma}u\delb_su
\endaligned
$$
that is
\begin{equation}\label{eq 4 08-09-2017}
Q^{\alpha\beta\gamma}\del_{\alpha}\Psib_{\beta}^{\beta'}\del_{\gamma}u\delb_{\beta'}u
=-s^{-1}\Qb^{00\gamma}\delb_{\gamma}u\delb_su + s^{-1}Q^{00\gamma}\del_{\gamma}u\delb_su - s^{-1}\sum_aQ^{aa\gamma}\del_{\gamma}u\delb_su
\end{equation}

We first concentrate on $\del^IL^J\left(Q^{\alpha\beta\gamma}\del_{\gamma}u\del_{\alpha}\del_{\beta}u\right)$. Applying lemma \ref{lem 2 02-09-2017} and proposition \ref{prop 1 02-09-2017}, we see that by \eqref{eq 5 03-09-2017} and \eqref{eq 4 08-09-2017} ,it is bounded by the sum of the following terms (modulo a constant determined by $I,J$):
\begin{equation}\label{eq 1 03-09-2017}
\aligned
&(s/t)^{-1}\big|\del^{I_1}L^{J_1}\delb_su\del^{I_2}L^{J_2}\delb_s\delb_su\big|,\quad (s/t)^{-2}\big|\del^{I_1}L^{J_1}\delb_su\del^{I_2}L^{J_2}\delb_s\delb_au\big|,
\\
&(s/t)^{-2}\big|\del^{I_1}L^{J_1}\delb_cu\del^{I_2}L^{J_2}\delb_s\delb_su\big|,
\\
&(s/t)^{-1}\big|\del^{I_1}L^{J_1}\delb_cu\del^{I_2}L^{J_2}\delb_a\delb_su\big|,\quad (s/t)^{-1}\big|\del^{I_1}L^{J_1}\delb_su\del^{I_2}L^{J_2}\delb_a\delb_bu\big|,
\\
&\big|\del^{I_1}L^{J_1}\delb_cu\del^{I_2}L^{J_2}\delb_a\delb_bu\big|,
\\
& (s^2/t)^{-1}\big|\del^{I_1}L^{J_1}\delb_su\del^{I_2}L^{J_2}\delb_su\big|,\quad s^{-1}(s/t)^{-2}\big|\del^{I_1}L^{J_1}\delb_au\del^{I_2}L^{J_2}\delb_su\big|,
\\
& s^{-1}|\del^{I_1}L^{J_1}\del_{\gamma}u\del^{I_2}L^{J_2}\delb_su|
\endaligned
\end{equation}
where $|I_1|+|I_2|\leq|I|$ and $|J_1|+|J_2|\leq |J|$. For the last term we have applied the fact that $\del_{\alpha}\Psib_{\beta}^{\beta'}$ is homogeneous of degree $-1$.

Then we regard $[\del^IL^J,Q^{\alpha\beta\gamma}\del_{\gamma}u\del_{\alpha}\del_{\beta}]u$. We see that by \eqref{eq 5 03-09-2017}, it is the sum of the following terms:
$$
[\del^IL^J,\Qb^{000}\delb_su\delb_s\delb_s]u,\quad [\del^IL^J, \Qb_1^{\alpha\beta}(\delb u)\delb_{\alpha}\delb_{\beta}]u, \quad [\del^IL^J,Q^{\alpha\beta\gamma}\del_{\alpha}\Psib_{\beta}^{\beta'}\del_{\gamma}u\delb_{\beta'}]u
$$
We see that in general the following calculation holds:
\begin{equation}\label{eq 2 03-09-2017}
\aligned
\,[\del^IL^J,\Qb^{\alpha\beta\gamma}\delb_{\gamma}u\delb_{\alpha}\delb_{\beta}]u
=& \sum_{{I_1+I_2+I_3=I\atop J_1+J_2+J_3=J}\atop |I_3|+|J_3|<|I|+|J|}
\del^{I_1}L^{J_1}\Qb^{\alpha\beta\gamma}\del^{I_2}L^{J_2}\delb_{\gamma}u\del^{I_3}L^{J_3}\delb_{\alpha}\delb_{\beta}u
\\
&+\Qb^{\alpha\beta\gamma}\delb_{\gamma}u[\del^IL^J,\delb_{\alpha}\delb_{\beta}]u
\endaligned
\end{equation}
and
\begin{equation}\label{eq 3 03-09-2017}
\aligned
\,[\del^IL^J,Q^{\alpha\beta\gamma}\del_{\alpha}\Psib_{\beta}^{\beta'}\del_{\gamma}u\delb_{\beta'}]u
=& -[\del^IL^J,s^{-1}\Qb^{00\gamma}\delb_{\gamma}u\delb_s]u + [\del^IL^J,s^{-1}Q^{00\gamma}\del_{\gamma}u\delb_s]u
\\
&- \sum_a[\del^IL^J,s^{-1}Q^{aa\gamma}\del_{\gamma}u\delb_s]u.
\endaligned
\end{equation}

In the section \ref{sec 2 global 2} we will make $L^2$ estimates on these terms based on the bootstrap bounds. As an preparation, we establish the following bounds:
\begin{lemma}\label{lem 1 07-09-2017}
Let $u$ be a function defined in $\Kcal$, sufficiently regular. Then the following estimates hold:
\begin{equation}\label{eq 2 07-09-2017}
\aligned
\big\|s^2(s/t)[\del^IL^J,\delb_s\delb_s]u\big\|_{L^2(\Hcal_s)}\leq&  C\sum_{|I'|\leq|I|,|J'|\leq|J|\atop |I'|+|J'|<|I|+|J|}\|(s^3/t)\Box\del^{I'}L^{J'}u\|_{L^2(\Hcal_s)}
\\
 &+C\sum_{|I'|+|J'|\leq|I|+|J|}\!\!\!\!\!\!\!\!\Ec(s,\del^{I'}L^{J'}u)^{1/2} .
\endaligned
\end{equation}
\begin{equation}\label{eq 3 07-09-2017}
\|s^2[\del^IL^J,\delb_a\delb_s]u\|_{L^2(\Hcal_s)}\leq C\sum_{|I'|+|J'|\leq|I|+|J|}\!\!\!\!\!\!\!\!\Ec(s,\del^{I'}L^{J'}u)^{1/2},
\end{equation}
and
\begin{equation}\label{eq 4 07-09-2017}
\|st[\del^IL^J,\delb_a\delb_b]u\|_{L^2(\Hcal_s)}\leq C\sum_{|I'|+|J'|\leq|I|+|J|}\!\!\!\!\!\!\!\!\Ec(s,\del^{I'}L^{J'}u)^{1/2}.
\end{equation}
\end{lemma}
\begin{proof}
We see first that by \eqref{eq 8 31-08-2017} and lemma \ref{lem 1 01-09-2017}:
$$
\aligned
&\big\|s^2(s/t)[\del^IL^J,\delb_s\delb_s]u\big\|_{L^2(\Hcal_s)}
\\
\leq&  C\sum_{\alpha,\beta,|I'|\leq|I|,|J'|\leq|J|\atop |I'|+|J'|<|I|+|J|}\!\!\!\!\!\!\!\!
\big\|s^2(s/t)^3\del_{\alpha}\del_{\beta}\del^{I'}L^{J'}u\big\|_{L^2(\Hcal_s)}
+ C\sum_{\alpha,|I'|\leq|I|,|J'|\leq|J|\atop |I'|+|J'|<|I|+|J|}\!\!\!\!\!\!\!\!
\big\|s(s/t)^2\del_{\alpha}\del^{I'}L^{J'}u\big\|_{L^2(\Hcal_s)}
\\
\leq& C\sum_{|I'|\leq|I|,|J'|\leq|J|\atop |I'|+|J'|<|I|+|J|}
\!\!\!\!\!\!\!\!\|s^2(s/t)\delb_s\delb_s\del^{I'}L^{J'}u\|_{L^2(\Hcal_s)}
+C\sum_{|I'|+|J'|\leq|I|+|J|}\!\!\!\!\!\!\!\!\Ec(s,\del^{I'}L^{J'}u)^{1/2}
\endaligned
$$
Then by proposition \ref{prop 1 31-08-2017}, we see that
$$
\aligned
\big\|s^2(s/t)[\del^IL^J,\delb_s\delb_s]u\big\|_{L^2(\Hcal_s)}
\leq& C\sum_{|I'|\leq|I|,|J'|\leq|J|\atop |I'|+|J'|<|I|+|J|}\|(s^3/t)\Box\del^{I'}L^{J'}u\|_{L^2(\Hcal_s)}
\\
 &+C\sum_{|I'|+|J'|\leq|I|+|J|}\!\!\!\!\!\!\!\!\Ec(s,\del^{I'}L^{J'}u)^{1/2} .
\endaligned
$$

The rest to estimates are direct by lemma \ref{lem 3 03-09-2017} and lemma \ref{lem 1 05-09-2017}, we omit the detail.
\end{proof}

\section{Global existence: bootstrap argument}\label{sec 1 global 1}

\subsection{The bootstrap bounds}
We consider the main equation of interest together with initial data :
\begin{equation}\label{eq 1 30-06-2017}
\left\{
\aligned
&\Box u + Q^{\alpha\beta\gamma}\del_{\gamma}u\del_{\alpha}\del_{\beta}u = 0
\\
&u|_{\Hcal_2}= u_0, \quad \del_tu|_{\Hcal_2}=u_1.
\endaligned
\right.
\end{equation}
where $u_i$ are sufficiently regular functions defined on the hyperboloid $\Hcal_2$ and supported in $\Hcal_2\cap \Kcal$.
\begin{remark}
The fact that the initial data are posed on hyperboloid is not a standard but we can see it in the following way: we pose the initial data set on the hyperplane $\{t=2\}$ and supported in the unit disc. Then by standard local existence result, the associated local solution extends to region $\{(t,x)| 2\leq t\leq \sqrt{r^2+4}\}\cap \Kcal$. Then we can restrict the solution on $\Hcal_2$. Thus we for global result we can pose our initial data on $\Hcal_2$. For more detail, see for example \cite{H1} or \cite{PLF-MY-book}.
\end{remark}

We will apply the so-called bootstrap argument, explained in detail here: Let $u$ be the local-in-time solution associated to \eqref{eq 1 30-06-2017}. Assume that the largest hyperbolic time interval of existence is $\Kcal_{[2,s^*)}$.

We define the bootstrap bounds for $s\in [2,s^*]$:,
\begin{equation}\label{eq 5 24-08-2017}
\sum_{|I|+|J|\leq N}E_{con}(s, \del^IL^Ju)^{1/2}\leq C_1\epsilon
\end{equation}
with $(C_1,\vep)$ a pair of positive constant to be determined.
Then we define $s_1$ to be the largest (hyperbolic) time where $u$ satisfies this condition, that is,
$$
s_1 :=\sup_s \left\{ s^*> s \geq 2| \text{\eqref{eq 5 24-08-2017} holds on} [s_0,s]\right\}
$$
We suppose that
\begin{equation}\label{eq 6 24-08-2017}
\sum_{|I|+|J|\leq N}E_{con}(2, \del^IL^Ju)^{1/2}\leq C_0\epsilon.
\end{equation}
which can be guaranteed by the smallness of the initial data. We see that when taking $C_1>C_0$, by continuity, $s_1>2$.

To argue by contradiction, we suppose that $s_1<s^*$. If we could deduce, for a suitable pair $(C_1,\vep_0)$, a {\sl improved} bound for all $0< \vep\leq \vep_0$:
\begin{equation}\label{eq 2 29-08-2017}
\sum_{|I|+|J|\leq N}E_{con}(s, \del^IL^Ju)^{1/2}\leq \frac{1}{2}C_1\epsilon \  \  \text{for}\ s\in [2,s_1],
\end{equation}
On the other hand, we see that by continuity,
$$
\sum_{|I|+|J|\leq N}E_{con}(s_1, \del^IL^Ju)^{1/2}= C_1\epsilon .
$$
This contradiction leads to the fact that $s_1=s^*$, then
$$
\sum_{|I|+|J|\leq N}E_{con}(s^*, \del^IL^Ju)^{1/2}\leq C_1\epsilon.
$$
Then by standard local-in-time theory (with $N$ sufficiently large), we see that $s^*$ could not be finite. This leads to the global existence result.
%
%
%

Now we state the main result of this article:
\begin{theorem}\label{thm main}
There exists a constant $\vep_0>0$, determined only by the system \eqref{eq 1 30-06-2017}, such for all $0\leq \vep\leq \vep_0$, if
\begin{equation}\label{eq 1 thm cond}
\|u_0\|_{H^{N+1}(\Hcal_s)} + \|u_1\|_{H^{N}(\Hcal_s)}\leq \vep
\end{equation}
holds for $N$ sufficiently large ($N\geq 9$ is enough), then the associated local-in-time solution extends to time infinity.
\end{theorem}

Based on the above discussion on bootstrap argument, we see that the above result is deduced from the following proposition:
\begin{proposition}\label{prop 1 01-07-2017}
There exists a pair of positive constant $(C_1,\vep_0)$, determined by only by the system \eqref{eq 1 30-06-2017} such that if the initial data set satisfies \eqref{eq 6 24-08-2017} with $0< \vep\leq \vep_0$, then \eqref{eq 5 24-08-2017} leads to \eqref{eq 2 29-08-2017}.
\end{proposition}
The following sections from \ref{sec 1 global 1} to \ref{sec 4 global 4} are devoted to the proof of this proposition.

\subsection{Basic bounds}
The following bounds hold in the region $\Kcal_{[2,s_1]}$.

The following terms are bounded by $CC_1\vep $ for $|I|+|J|\leq N$:
\begin{equation}\label{eq 3 01-07-2017}
\|s\delb_a\del^IL^Ju\|_{L^2(\Hcal_s)},\quad \|(s^2/t)\delb_s\del^IL^Ju\|_{L^2(\Hcal_s)},\quad \|(s/t)\del^IL^Ju\|_{L^2(\Hcal_s)}.
\end{equation}
Then by \eqref{eq 3 24-08-2017}, \eqref{eq 6 05-09-2017} and \eqref{eq 5 07-09-2017}, the following bounds are also bounded by $CC_1\vep $:
\begin{equation}\label{eq 4 01-07-2017}
\aligned
&\|s\del^IL^J\delb_au\|_{L^2(\Hcal_s)},\quad \|(s^2/t)\del^IL^J\delb_su\|_{L^2(\Hcal_s)}
\\
&\|s^2\del^IL^J\delb_s\delb_au\|_{L^2(\Hcal_s)},\quad \|st\del^IL^J\delb_a\delb_bu\|_{L^2(\Hcal_s)},\quad \|s\del^IL^J\delb_s\delb_su\|_{L^2(\Hcal_s)}.
\endaligned
\end{equation}

By proposition \ref{prop 1 29-08-2017} and the global Sobolev inequality, for $|I|+|J|\leq N-2$, the following terms are bounded by $CC_1\vep $:
\begin{equation}\label{eq 5 01-07-2017}
\sup_{\Hcal_s}\left\{t^{3/2}s\delb_a\del^IL^J u\right\},\quad \sup_{\Hcal_s}\left\{t^{1/2}s^2 \delb_s\del^IL^J u\right\},\quad
\sup_{\Hcal_s}\left\{t^{1/2}s\del^IL^J u\right\},
\end{equation}
\begin{equation}\label{eq 4 01-09-2017}
\sup_{\Hcal_s}\left\{t^{3/2}s\del^IL^J\delb_au\right\},\quad \sup_{\Hcal_s}\left\{t^{1/2}s^2\del^IL^J \delb_su\right\},
\end{equation}
and for $|I|+|J|\leq N-3$, the following terms are bounded by $CC_1\vep $:
\begin{subequations}
\begin{equation}\label{eq 7 07-09-2017}
\sup_{\Hcal_s}\left\{s^2t^{3/2}\delb_s\delb_a\del^IL^J u\right\},\quad \sup_{\Hcal_s}\left\{st^{5/2}\delb_a\delb_b\del^IL^Ju\right\}
\end{equation}
\begin{equation}\label{eq 8 07-09-2017}
\sup_{\Hcal_s}\left\{s^2t^{3/2}\del^IL^J\delb_s\delb_a u\right\},\quad \sup_{\Hcal_s}\left\{st^{5/2}\del^IL^J\delb_a\delb_bu\right\}
\end{equation}
and
\begin{equation}\label{eq 14 07-09-2017}
\sup_{\Hcal_s}\left\{st^{3/2}\del^IL^J\delb_s\delb_su\right\}.
\end{equation}
where for the last term in the above list we applied \eqref{eq 10 07-09-2017}
\end{subequations}

\section{Global existence: refined bounds}\label{sec 2 global 2}
\subsection{Estimates on Hessian form}
We combine \eqref{eq 5 24-08-2017} together with proposition \eqref{prop 1 31-08-2017}:
\begin{equation}\label{eq 5 01-09-2017}
\|(s^3/t)\del^IL^J\delb_s\delb_su\|_{L^2(\Hcal_s)} + \|(s^3/t)\delb_s\delb_s\del^IL^Ju\|_{L^2(\Hcal_s)}\leq C\|(s^3/t)\Box \del^IL^J u\|_{L^2(\Hcal_s)} + CC_1\vep .
\end{equation}
Similar bounds hold for the combination of \eqref{eq 5 24-08-2017}with \eqref{eq 3 01-09-2017}. Thus we need to control $\|(s^3/t)\Box \del^IL^J u\|_{L^2(\Hcal_s)}$. This is by the following lemma:
\begin{lemma}\label{lem 2 01-09-2017}
Under the bootstrap bound \eqref{eq 5 24-08-2017}, the following estimate holds for $|I|+|J|\leq N-1$:
\begin{equation}\label{eq 6 01-09-2017}
\big\|(s^3/t)\del^IL^J\left(Q^{\alpha\beta\gamma}\del_{\gamma}u\del_{\alpha}\del_{\beta}u\right)\big\|_{L^2(\Hcal_s)}\leq C(C_1\vep)^2 .
\end{equation}
\end{lemma}
\begin{proof}
This is based on the $L^2$ bounds and $L^{\infty}$ bounds established in the last section. We need to bound each term in the list \eqref{eq 1 03-09-2017}.

For each term in \eqref{eq 1 03-09-2017}, for $|I_1|+|J_1|\leq N-2$, we apply the decay bounds \eqref{eq 4 01-09-2017} on the first factor and the apply the $L^2$ bounds \eqref{eq 4 01-07-2017} on the second factor. We can check that for each term, the $L^2$ norm is bounded as
$$
\|(s^3/t)X\|_{L^2(\Hcal_s)}\leq C(C_1\vep)^2
$$
where $X$ represents a term in \eqref{eq 4 01-09-2017}.

When $|I_1|+|J_1|\geq N-1$, we see that $|I_2|+|J_2|\leq 1\leq N-3$. Thus in the similar way, we apply the decay estimates of \eqref{eq 8 07-09-2017}, \eqref{eq 14 07-09-2017} on the second factor and \eqref{eq 4 01-07-2017} (the first two terms) on the first factor.
\end{proof}

Now we are ready to establish the refined bound on $\del^IL^J\delb_s\delb_su$.
\begin{lemma}\label{lem 3 07-09-2017}
Under the bootstrap assumption, we see that
\begin{equation}\label{eq 11 07-09-2017}
\|(s^3/t)\del^IL^J\delb_s\delb_su\|_{L^2(\Hcal_s)} + \|(s^3/t)\delb_s\delb_s\del^IL^Ju\|_{L^2(\Hcal_s)}\leq CC_1\vep
\end{equation}
and
\begin{equation}\label{eq 12 07-09-2017}
\sup_{\Hcal_s}\{t^{1/2}s^3\del^IL^J\delb_s\delb_su\} + \sup_{\Hcal_s}\{s^3t^{1/2}\delb_s\delb_s\del^IL^Ju\}\leq CC_1\vep
\end{equation}
for $|I|+|J|\leq N-1$.
\end{lemma}
\begin{proof}
This is by using the equation. We see that
$$
\Box \del^IL^Ju = \del^IL^J\left(Q^{\alpha\beta\gamma}\del_{\gamma}u\del_{\alpha}\del_{\beta}u\right).
$$
Thus by \eqref{eq 6 01-09-2017}
\begin{equation}\label{eq 12' 07-09-2017}
\|(s^3/t)\Box \del^IL^J u\|_{L^2(\Hcal_s)}\leq C(C_1\vep)^2 .
\end{equation}
Now we apply \eqref{eq 5 31-08-2017} and \eqref{eq 3 01-09-2017} together with the above bounds and the bootstrap bound on energy, and we see that the desired bounds is established.
\end{proof}

\subsection{Estimates on null form}
In this section we concentrate on the $L^2$ bounds on $[\del^IL^J,Q^{\alpha\beta\gamma}\del_{\gamma}u\del_{\alpha}\del_{\beta}]u$. To get started we combine the bootstrap bounds with lemma \ref{lem 1 07-09-2017}, and we see that the following terms are bounded by $CC_1\vep $:
\begin{equation}\label{eq 2' 07-09-2017}
\aligned
\big\|s^2(s/t)[\del^IL^J,\delb_s\delb_s]u\big\|_{L^2(\Hcal_s)}, \quad\|s^2[\del^IL^J,\delb_a\delb_s]u\|_{L^2(\Hcal_s)},\quad \|st[\del^IL^J,\delb_a\delb_b]u\|_{L^2(\Hcal_s)}
\endaligned
\end{equation}
where $|I|+|J|\leq N$.
Based on these bounds, we establish the following estimates:
\begin{lemma}\label{lem 4 07-09-2017}
Under the bootstrap assumption and assume that $Q^{\alpha\beta\gamma}$ be a null cubic form, then the following bounds hold:
\begin{equation}\label{eq 13 07-09-2017}
\|s\Qb^{\alpha\beta\gamma}\delb_{\gamma}u\cdot[\del^IL^J,\delb_{\alpha}\delb_{\beta}u]\|_{L^2(\Hcal_s)}\leq C(C_1\vep)^2 s^{-2 }.
\end{equation}
\end{lemma}
\begin{proof}
For $(\alpha,\beta,\gamma) = (0,0,0)$, we see that by \eqref{eq 10 01-09-2017}, \eqref{eq 4 01-09-2017} and \eqref{eq 2' 07-09-2017} (recall that in $\Kcal, s^2\geq t$):
$$
\aligned
\|s\Qb^{000}\delb_su\cdot[\del^IL^J,\delb_s\delb_su]\|_{L^2(\Hcal_s)}
\leq& CC_1\vep\|s(t/s)t^{-1/2}s^{-2+\delta} \cdot (t/s)s^{-2}\cdot s^2(s/t)[\del^IL^J,\delb_s\delb_s]u\|_{L^2(\Hcal_s)}
\\
\leq& C(C_1\vep)^2s^{-2 }.
\endaligned
$$

For the rest components, we apply lemma \ref{lem 2 02-09-2017}, \eqref{eq 4 01-09-2017} and \eqref{eq 2' 07-09-2017}. We omit the detail.
\end{proof}

\begin{lemma}\label{lem 5 07-09-2017}
Under the bootstrap assumption, the $L^2$ norm of the following term
$$
s\del^{I_1}L^{J_1}\Qb^{\alpha\beta\gamma}\del^{I_2}L^{J_2}\delb_{\gamma}u\del^{I_3}L^{J_3}\delb_{\alpha}\delb_{\beta}u
$$
is controlled by $C(C_1\vep)^2s^{-2 }$, where
$$
I_1+I_2+I_3=I,\quad J_1+J_2+J_3=J,\quad |I_3|+|J_3|<|I|+|J|\leq N.
$$
\end{lemma}
\begin{proof}
This is also by applying \eqref{eq 10 01-09-2017} (for $(\alpha,\beta,\gamma) = (0,0,0)$) or lemma \ref{lem 2 02-09-2017} (for $(\alpha,\beta,\gamma)\neq (0,0,0)$). We see that when $|I_1|+|J_1|\leq N-2$, we apply \eqref{eq 4 01-09-2017} on the factor $\del^{I_2}L^{J_2}\delb_{\gamma}u$, \eqref{eq 4 01-07-2017} (the third and forth term) on the factor $\del^{I_3}L^{J_3}\delb_a\delb_{\beta}u$ and \eqref{eq 11 07-09-2017} on the factor $\del^{I_3}L^{J_3}\delb_s\delb_su$. For example for the component $(0,0,0)$, we see that
$$
\aligned
&\|s\del^{I_1}L^{J_1}\Qb^{000}\del^{I_2}L^{J_2}\delb_su\del^{I_3}L^{J_3}\delb_s\delb_su\|_{L^2(\Hcal_s)}
\\
\leq& CC_1\vep\|s(t/s)\cdot t^{-1/2}s^{-2+\delta}\cdot(ts^{-3})\cdot (s^3/t)\del^{I_3}L^{J_3}\delb_s\delb_su\|_{L^2(\Hcal_s)}\leq C(C_1\vep)^2s^{-2 }.
\endaligned
$$
The rest components are verified similarly and we omit de detail.

When $|I_1|+|I_2|\geq N-1$, we see that $|I_3|+|J_3|\leq 1\leq N-3$. Thus we apply \eqref{eq 4 01-07-2017} (the first two terms) on the factor $\del^{I_2}L^{J_2}\delb_\gamma u$, \eqref{eq 8 07-09-2017} on the factor $\del^{I_3}L^{J_3}\delb_a\delb_{\beta}u$ and \eqref{eq 12 07-09-2017} for $\del^{I_3}L^{J_3}\delb_a\delb_su$. For example for the component $(0,0,0)$
$$
\aligned
&\|s\del^{I_1}L^{J_1}\Qb^{000}\del^{I_2}L^{J_2}\delb_su\del^{I_3}L^{J_3}\delb_s\delb_su\|_{L^2(\Hcal_s)}
\\
\leq& CC_1\vep\|s(t/s)\cdot t^{-1/2}s^{-3+2\delta}\cdot (t/s^2)\cdot (s^2/t)\del^{I_2}L^{J_2}\delb_su \|_{L^2(\Hcal_s)}\leq C(C_1\vep)^2s^{-2 }.
\endaligned
$$
The rest components are verified similarly and we omit the detail.
\end{proof}

\begin{lemma}\label{lem 6 07-09-2017}
Under the bootstrap assumption, the following estimate holds:
\begin{equation}\label{eq 1 08-09-2017}
\|s[\del^IL^J,Q^{\alpha\beta\gamma}\del_{\alpha}\Psib_{\beta}^{\beta'}\del_{\gamma}u\delb_{\beta'}]u\|_{L^2(\Hcal_s)}\leq C(C_1\vep)^2s^{-3/2 }.
\end{equation}
\end{lemma}
\begin{proof}
We recall \eqref{eq 3 03-09-2017}. For the first term in right-hand-side, we see that
$$
\aligned
\,[\del^IL^J,s^{-1}\Qb^{00\gamma}\delb_{\gamma}u\delb_s]u
=& \sum_{I_1+I_2=I,J_1+J_2=J\atop |I_2|+|J_2|<|I|+|J|} \del^{I_1}L^{J_1}\left(s^{-1}\Qb^{00\gamma}\delb_{\gamma}u\right)\del^{I_2}L^{J_2}\delb_su
\\
&+s^{-1}\Qb^{00\gamma}\delb_{\gamma}u[\del^IL^J,\delb_s]u
\\
=:& T_1 + T_2.
\endaligned
$$

For $T_1$, we see that
$$
\aligned
\del^{I_1}L^{J_1}\left(s^{-1}\Qb^{00\gamma}\delb_{\gamma}u\right)
=& \del^{I_{11}}L^{J_{11}}\left(s^{-1}\right)\del^{I_{12}}L^{J_{12}}\Qb^{00\gamma}\cdot\del^{I_{13}}L^{J_{13}}\delb_{\gamma}u
\endaligned
$$
Then applying \eqref{eq 3 31-08-2017} together with lemma \ref{lem 2 02-09-2017} (for $\gamma>0$) or \eqref{eq 10 01-09-2017} (for $\gamma=0$):
$$
|T_1|\leq
\left\{
\aligned
&C(t/s^2)\sum_{|I_1|+|I_2|\leq|I|\atop|J_1|+|J_2|\leq|J|}|\del^{I_1}L^{J_1}\delb_su\del^{I_2}L^{J_2}\delb_su|, \gamma = 0,
\\
&Cs^{-1}(t/s)^2\sum_{a,|I_1|+|I_2|\leq|I|\atop|J_1|+|J_2|\leq|J|}|\del^{I_1}L^{J_1}\delb_au\del^{I_2}L^{J_2}\delb_su|,\gamma = a >0.
\endaligned
\right.
$$

Now, for $|I_1|+|J_1|\leq N-2$, we apply decay estimate \eqref{eq 4 01-09-2017} on the first factor and $L^2$ bounds \eqref{eq 4 01-07-2017} (the first two terms) on the second factor. When $|I_1|+|J_1|\geq N-1$, we see that $|I_2|+|J_2|\leq 1\leq N-2$. In this case we apply \eqref{eq 4 01-09-2017} on the second factor and \eqref{eq 4 01-07-2017} on the first factor. This leads to:
$$
\|sT_1\|_{L^2(\Hcal_s)}\leq C(C_1\vep)^2s^{-2 }.
$$

For the term $T_2$, we see that by \eqref{eq 4 01-09-2017} applied on $\del_{\gamma}u$, lemma \ref{lem 2 02-09-2017} (for $\gamma >0$) or \eqref{eq 10 01-09-2017} (for $\gamma = 0$)
$$
|T_2|\leq\left\{
\aligned
&Cs^{-4}t^{1/2}|[\del^IL^J,\delb_s]u|,\quad \gamma = 0,
\\
&Cs^{-4}t^{1/2}|[\del^IL^J,\delb_a]u|,\quad \gamma>0.
\endaligned
\right.
$$
Then we combine \eqref{eq 2 23-08-2017}, \eqref{eq 8 22-08-2017} together with \eqref{eq 3 01-07-2017}, the following bound is established
$$
\|sT_2\|_{L^2(\Hcal_s)}\leq C(C_1\vep)^2s^{-2 }.
$$

The rest terms in right-hand-side of \eqref{eq 3 03-09-2017} are bounded similarly, we omit the detail.
\end{proof}

Now we are ready to conclude the following result:
\begin{proposition}\label{prop 1 08-9-2017}
Under the bootstrap assumption, for $|I|+|J|\leq N$, the following estimate holds:
\begin{equation}\label{eq 2 08-09-2017}
\|s[\del^IL^J,Q^{\alpha\beta\gamma}\del_{\gamma}u\del_{\alpha}\del_{\beta}]u\|_{L^2(\Hcal_s)}\leq C(C_1\vep)^{2}s^{-2 }.
\end{equation}
\end{proposition}

\section{Global existence: conclusion of bootstrap argument}\label{sec 4 global 4}
Now we are ready to prove proposition \ref{prop 1 01-07-2017}. To do so we need to guarantee \eqref{eq 3 23-05-2017} and the bounds on $M_g(s)$ (with the notion in \eqref{eq 4 23-05-2017}). We first remark that by the notation in subsection \ref{subsec energy-indentity}
$$
\hb^{\alpha\beta} = \Qb^{\alpha\beta\gamma}\delb_{\gamma}u.
$$
For the convenience of discussion, we list out the following bounds on $\hb^{\alpha\beta}$. These are by \eqref{eq 4 01-09-2017} combined with lemma \ref{lem 2 02-09-2017}, \eqref{eq 10 01-09-2017}:
\begin{equation}\label{eq 5 08-09-2017}
\aligned
&|\hb^{00}| \leq CC_1\vep t^{1/2}s^{-3 },\quad |\hb^{a0}| + |\hb^{0a}|\leq CC_1\vep t^{3/2}s^{-4},
\\
&|\hb^{ab}|\leq CC_1\vep t^{1/2}s^{-3 },\quad |h^{\alpha\beta}|\leq CC_1\vep t^{1/2}s^{-3 }.
\endaligned
\end{equation}
Furthermore, we see that
\begin{equation}\label{eq 1 25-10-2017}
\delb_{\mu}\hb^{\alpha\beta} = \delb_{\mu}\Qb^{\alpha\beta\gamma}\cdot \delb_\gamma u + \Qb^{\alpha\beta\gamma}\delb_{\mu}\delb_{\gamma}u.
\end{equation}
We apply, lemma \ref{lem 2 02-09-2017} and \eqref{eq 10 01-09-2017} combined with \eqref{eq 4 01-09-2017}, \eqref{eq 8 07-09-2017} and \eqref{eq 12 07-09-2017}.
\begin{equation}\label{eq 6 08-09-2017}
\aligned
&|\delb_s\hb^{00}|\leq C(C_1\vep)t^{1/2}s^{-4 },\quad |\delb_s\hb^{a0}| + |\delb_s\hb^{0a}|\leq C(C_1\vep)t^{3/2}s^{-5},\quad |\delb_s\hb^{ab}|\leq C(C_1\vep)t^{1/2}s^{-4 },
\\
&|\delb_a\hb^{00}|\leq C(C_1\vep)t^{-1/2}s^{-3 },\quad |\delb_c\hb^{a0}| + |\delb_c\hb^{0a}|\leq C(C_1\vep)t^{1/2}s^{-4 },\quad |\delb_c\hb^{ab}|\leq C(C_1\vep)t^{-1/2}s^{-3 }.
\endaligned
\end{equation}
We also remark that
$$
\delb_{\delta}h^{\alpha\beta} = Q^{\alpha\beta\gamma}\delb_{\delta}\del_{\gamma}u
$$
that is
$$
\aligned
\delb_sh^{\alpha\beta} =& Q^{\alpha\beta0}\delb_s\del_tu + Q^{\alpha\beta c}\delb_s\del_cu
\\
=&Q^{\alpha\beta0}\delb_s\left((t/s)\delb_su\right) + Q^{\alpha\beta c}\delb_s\left(\delb_cu - (x^c/s)\delb_s u\right)
\\
=&\left((t/s)Q^{\alpha\beta0} - (x^c/s)Q^{\alpha\beta c}\right)\delb_s\delb_su + Q^{\alpha\beta c}\delb_s\delb_cu
+  \left((x^c/s^2)Q^{\alpha\beta c}- \frac{r^2}{ts^2}Q^{\alpha\beta 0}\right)\delb_su,
\endaligned
$$
which leads to
\begin{equation}\label{eq 1 26-10-2017}
|\delb_sh^{\alpha\beta}|\leq C(t/s)|\delb_s\delb_su| + C\sum_a|\delb_s\delb_a u| + C(t/s^2)|\delb_su|.
\end{equation}
Similar calculation shows that
\begin{equation}\label{eq 2 26-10-2017}
|\delb_ah^{\alpha\beta}|\leq C(t/s)|\delb_a\delb_su| + C\sum_c|\delb_a\delb_c u| + s^{-1}|\delb_su|.
\end{equation}
Then combined with \eqref{eq 12 07-09-2017}, \eqref{eq 7 07-09-2017} and \eqref{eq 5 01-07-2017}, we see that
\begin{equation}\label{eq 3 26-10-2017}
|\delb_sh^{\alpha\beta}|\leq Ct^{1/2}s^{-4 },\quad |\delb_ah^{\alpha\beta}|\leq Ct^{-1/2}s^{-3 }.
\end{equation}

Then we establish the following bounds:
\begin{lemma}\label{lem 1 26-10-2017}
Under the bootstrap assumption \eqref{eq 5 24-08-2017},
$$
|\delb_\alpha N_g|\leq CC_1\vep t^{1/2}s^{-4 }.
$$
\end{lemma}
\begin{proof}
We recall that
$$
N_g - 2 = h^{00} - \sum_ah^{aa} - 2\hb^{00} - s\delb_s\hb^{00}
$$
thus
$$
|\delb_{\alpha}N_g|\leq 2|\delb_{\alpha}\hb^{00}| + \sum_a|\delb_{\alpha}h^{aa}| + |\delb_{\alpha}h^{00}| + |\delb_{\alpha}(s\delb_s\hb^{00})|
$$
The first three terms are bounded by \eqref{eq 6 08-09-2017} and  \eqref{eq 3 26-10-2017}. For the last term, we remark the following relation:
\begin{equation}\label{eq 4 26-10-2017}
\aligned
\big|\delb_s\left(s\delb_s\hb^{00}\right)\big|\leq& C\big|\delb_s\hb^{00}\big| + C\big|\delb_s\delb_s\Qb^{000}\delb_su\big| + C\big|\delb_s\Qb^{000}\delb_s\delb_su\big| + C|\Qb^{000}\delb_s\delb_s\delb_s u|
\\
&+C\big|\delb_s\delb_s\Qb^{00c}\delb_cu\big| + C\big|\delb_s\Qb^{00c}\delb_cu\big| + C\big|\Qb^{00c}\delb_s\delb_s\delb_c u\big|
\endaligned
\end{equation}
and
\begin{equation}\label{eq 5 26-10-2017}
\aligned
\big|\delb_a\left(s\delb_s\hb^{00}\right)\big| =& s|\delb_a\delb_s\hb^{00}|\leq C\big|\delb_a\delb_s\Qb^{000}\delb_su\big| + C\big|\delb_a\Qb^{000}\delb_s\delb_su\big| + C|\Qb^{000}\delb_a\delb_s\delb_s u|
\\
&+C\big|\delb_a\delb_s\Qb^{00c}\delb_cu\big| + C\big|\delb_a\Qb^{00c}\delb_cu\big| + C\big|\Qb^{00c}\delb_a\delb_s\delb_c u\big|
\endaligned.
\end{equation}
And we see that by lemma \ref{lem 2 02-09-2017} and proposition \ref{prop 1 02-09-2017}, we see that
\begin{equation}\label{eq 6 26-10-2017}
\aligned
&|\delb_s\delb_s\Qb^{000}|\leq (s/t)|\delb_s\del_t\Qb^{000}| + s^{-1}|\delb_s\Qb^{000}|\leq CC_1\vep s^{-1},
\\
&|\delb_s\delb_s\Qb^{00c}|\leq (s/t)|\delb_s\del_t\Qb^{00c}| + s^{-1}|\delb_s\Qb^{00c}|\leq CC_1(t/s)s^{-1}.
\endaligned
\end{equation}
Similar calculation shows that
\begin{equation}\label{eq 7 26 10-2017}
|\delb_a\delb_s\Qb^{000}|\leq CC_1\vep s^{-2},\quad |\delb_a\delb_s\Qb^{00c}|\leq CC_1\vep (t/s)s^{-2}.
\end{equation}
We also see that by \eqref{eq 12 07-09-2017}:
\begin{equation}\label{eq 8 26-10-2017}
|\delb_s\delb_s\delb_su|\leq (s/t)|\del_t\delb_s\delb_su|\leq CC_1\vep t^{-3/2}s^{-2 },\quad
|\delb_a\delb_s\delb_su|\leq t^{-1}|L_a\delb_s\delb_su|\leq CC_1\vep t^{-3/2}s^{-3 }.
\end{equation}
and by \eqref{eq 7 07-09-2017},
\begin{equation}\label{eq 9 26-10-2017}
|\delb_a\delb_s\delb_c u|\leq t^{-1}|L_a\delb_c\delb_su|\leq CC_1\vep t^{-5/2}s^{-2 }.
\end{equation}
Now we substitute the bounds \eqref{eq 6 08-09-2017} together with \eqref{eq 6 26-10-2017}, \eqref{eq 7 26 10-2017}, lemma \ref{lem 2 02-09-2017}, proposition \ref{prop 1 02-09-2017}, \eqref{eq 8 26-10-2017}, \eqref{eq 9 26-10-2017}, \eqref{eq 12 07-09-2017}, \eqref{eq 8 07-09-2017} and \eqref{eq 5 01-07-2017}, we see that
\begin{equation}\label{eq 10 26-10-2017}
\big|\delb_s\left(s\delb_s\hb^{00}\right)\big|\leq CC_1\vep t^{1/2}s^{-4 },\quad
\big|\delb_a\left(s\delb_s\hb^{00}\right)\big|\leq CC_1\vep t^{-1/2}s^{-3 }.
\end{equation}

Now combine \eqref{eq 6 08-09-2017}, \eqref{eq 3 26-10-2017} and \eqref{eq 10 26-10-2017}, we see that the desired bound is established.

\end{proof}

Then we establish the following bound:
\begin{lemma}\label{lem 1 08-09-2017}
Under the bootstrap assumption with $\vep$ sufficiently small, \eqref{eq 3 23-05-2017} holds for a $\kappa>1$.
\end{lemma}
\begin{proof}
This is by verifying proposition \ref{prop 1 24-05-2017}. We see that by \eqref{eq 5 08-09-2017} and \eqref{eq 6 08-09-2017}, \eqref{eq 1 24-05-2017} is verified with $\vep\leq \frac{\vep_s}{CC_1}$.
%
\end{proof}

\begin{lemma}\label{lem 2 08-09-2017}
Under the bootstrap assumption, we have
$$
M_g(s,\del^IL^Ju) = C(C_1\vep)^2 s^{-2}
$$
where $M_g$ is defined as in \eqref{eq 4 23-05-2017}.
\end{lemma}
\begin{proof}

Recall \eqref{eq R(du,du)}, we see that
$$
\aligned
\|R_g(\nabla \del^IL^Ju,\nabla \del^IL^Ju)\|_{L^1(\Hcal_s)}
\leq & \|s^{-1}\left(L_g^{ab}-L_m^{ab}\right)\|_{L^{\infty}(\Hcal_s)}\|s^2\delb_a\del^IL^Ju\delb_b\del^IL^Ju\|_{L^1(\Hcal_s)}
\\
&+ \|s^{-1}(N_g-N_m)\gb^{ab}\|_{L^{\infty}(\Hcal_s)}\|s^2\delb_a\del^IL^Ju\delb_b\del^IL^Ju\|_{L^1(\Hcal_s)}
 \\
 &+ 2\|s^{-1}\hb^{ab}\|_{L^{\infty}(\Hcal_s)}\|s^2\delb_a\del^IL^Ju\delb_b\del^IL^Ju\|_{L^1(\Hcal_s)}
\\
&+ \frac{1}{2}\|\delb_s\left(\gb^{00}\gb^{ab}\right)\|_{L^{\infty}(\Hcal_s)}\|s^2\delb_a\del^IL^Ju\delb_b\del^IL^Ju\|_{L^1(\Hcal_s)}
\\
&+ \|(t/s)\delb_a\gb^{00}\gb^{ab}\|_{L^{\infty}(\Hcal_s)} \cdot \|s^2(s/t)\delb_s\del^IL^Ju\delb_b\del^IL^Ju\|_{L^1(\Hcal_s)}
\\
\leq& CM_1\Ec(s,\del^IL^Ju)
\endaligned
$$
where
$$
\aligned
M_1: =& \|s^{-1}\left(L_g^{ab}-L_m^{ab}\right)\|_{L^{\infty}(\Hcal_s)} + \|s^{-1}(N_g-N_m)\gb^{ab}\|_{L^{\infty}(\Hcal_s)} + \|s^{-1}\hb^{ab}\|_{L^{\infty}(\Hcal_s)}
\\
& + \|\delb_s\left(\gb^{00}\gb^{ab}\right)\|_{L^{\infty}(\Hcal_s)} + \|(t/s)\delb_a\gb^{00}\gb^{ab}\|_{L^{\infty}(\Hcal_s)}.
\endaligned
$$
We remark the following bounds (by \eqref{eq 6 08-09-2017}):
$$
\aligned
&\|L_g^{ab} - L_m^{ab}\|_{L^{\infty}(\Hcal_s)}\leq CC_1\vep s^{-2 },\quad \|N_g-N_m\|_{L^{\infty}(\Hcal_s)}\leq CC_1\vep s^{-2 },
\\
&\|\hb^{ab}\|_{L^{\infty}(\Hcal_s)}\leq CC_1\vep s^{-2 },
\\
&\|\delb_s(\gb^{00}\gb^{ab})\|_{L^{\infty}(\Hcal_s)}\leq CC_1\vep s^{-3 },\quad \|(t/s)\delb_a\gb^{00}\gb^{ab}\|_{L^{\infty}(\Hcal_s)} \leq CC_1\vep s^{-3 }
\endaligned
$$
then we see that
$$
M_1\leq CC_1\vep s^{-3 }.
$$

To analysis the term $(\mathscr{K}_g+N_g) \del^IL^Ju\cdot S_g[\nabla \del^IL^Ju]$, we see that
$$
\aligned
&\big\|(\mathscr{K}_g + N_g)\del^IL^Ju\cdot S_g[\nabla \del^IL^Ju]\big\|_{L^1(\Hcal_s)}
\leq&C\|S_g[\nabla \del^IL^J u]\|_{L^2(\Hcal_s)}\cdot \Ec(s,\del^IL^Ju)^{1/2}.
\endaligned
$$
Then we see that by \eqref{eq 6 08-09-2017} and especially the bound on $\hb^{a0},\delb_s\hb^{a0}$
$$
\aligned
&\|S_g[\nabla \del^IL^J u]\|_{L^2(\Hcal_s)}
\\
=& \|2s^{-1}\delb_s(s\mb^{a0}) + 2s^{-1}\delb_s(s\hb^{a0}) + \delb_b\gb^{ab}\|_{L^{\infty}(\Hcal_s)}\|s\delb_a\del^IL^Ju\|_{L^2(\Hcal_s)}
\\
=&\|2s^{-1}\delb_s(s\hb^{a0}) + \delb_b\gb^{ab}\|_{L^{\infty}(\Hcal_s)}\|s\delb_a\del^IL^Ju\|_{L^2(\Hcal_s)}
\\
\leq& CC_1\vep s^{-2}\Ec(s,\del^IL^Ju)^{1/2}
\endaligned
$$
Then we see that
$$
\big\|\left(\mathscr{K}_g+N_g\right) \del^IL^Ju\cdot S_g[\nabla \del^IL^Ju]\big\|_{L^1(\Hcal_s)}\leq CC_1\vep s^{-2}\Ec(s,\del^IL^Ju).
$$

For the term $T_g[\del^IL^Ju]$, we see that
$$
\aligned
T_g[\del^IL^J u] =& - (t/s)\delb_sN_g\cdot(s/t)\del^IL^Ju\cdot(\mathscr{K}_g+N_g)\del^IL^Ju
\\
&-(t/s)\gb^{ab}\delb_aN_g\cdot s\delb_b\del^IL^Ju\cdot(s/t)\del^IL^Ju
\endaligned
$$
thus we see that
$$
\|T_g[\del^IL^J u]\|_{L^1(\Hcal_s)}\leq C\|(t/s)\delb_\alpha N_g\|_{L^{\infty}(\Hcal_s)}\cdot \Ec(s,\del^IL^Ju).
$$
Then we apply lemma \ref{lem 1 26-10-2017} and \eqref{eq 6 08-09-2017}, we see that
$$
\|(t/s)\delb_\alpha N_g\|_{L^{\infty}(\Hcal_s)}\leq CC_1\vep s^{-2}.
$$

Thus we see that the bound on $M_g$ is bounded by $CC_1\vep s^{-2}\Ec(s,\del^IL^Ju)^{1/2}$. Then we apply the bootstrap bound \eqref{eq 5 24-08-2017} and the desired result is established.

\end{proof}

Now we are ready to establish the improved energy bound.
\begin{proof}[Proof of proposition \ref{prop 1 01-07-2017}]
This is by applying the energy estimate \eqref{eq energy} on the following equation:
\begin{equation}\label{eq energy final}
\Box \del^IL^J u + Q^{\alpha\beta\gamma}\del_\gamma u \del_{\alpha}\del_{\beta}\del^IL^J u =  - [\del^IL^J,Q^{\alpha\beta\gamma}\del_{\gamma}u\del_{\alpha}\del_{\beta}]u.
\end{equation}
We see that for $s\in[2,s_1]$,
\begin{equation}\label{eq 4 28-10-2017}
\Ec(s,\del^IL^J u)^{1/2} \leq C\Ec(2,\del^IL^J u) + C\int_{2}^s \
\left(\tau\|[\del^IL^J,Q^{\alpha\beta\gamma}\del_{\gamma}u\del_{\alpha}\del_{\beta}]u\|_{L^2(\Hcal_\tau)} + M_g(\tau)\right)\ d\tau
\end{equation}
Recall that the initial energy is determined by the initial data and thus can be bounded by $CC_0\vep$. Then we substitute the bounds \eqref{eq 2 08-09-2017} and lemma \ref{lem 2 08-09-2017}, we see that
$$
\aligned
\Ec(s,\del^IL^J u)^{1/2} \leq& CC_0\vep + C(C_1\vep)^2\int_{2}^{s}\tau^{-2 }d\tau \leq CC_0\vep + C(C_1\vep)^2.
\endaligned
$$
Then we see that we chose $C_1>2CC_0$ and $\vep_0\leq  \frac{C_1-2CC_0}{CC_1^2}$. With this choice, we obtain that
$$
\Ec(s,\del^IL^J u)^{1/2}\leq \frac{1}{2}C_1\vep
$$
and this concludes the desired result.
\end{proof}

\end{document}